\documentclass[11pt, twoside, leqno]{article}

\usepackage{amssymb}
\usepackage{amsmath}
\usepackage{amsthm}
\usepackage{color}
\usepackage{mathrsfs}
\usepackage{indentfirst}
\usepackage{txfonts}
\usepackage{anysize}

\allowdisplaybreaks

\newtheorem{theorem}{Theorem}[section]
\newtheorem{lemma}[theorem]{Lemma}
\newtheorem{corollary}[theorem]{Corollary}

\newtheorem{proposition}[theorem]{Proposition}
\theoremstyle{definition}
\newtheorem{example}[theorem]{Example}
\newtheorem{remark}[theorem]{Remark}
\newtheorem{definition}[theorem]{Definition}

\numberwithin{equation}{section}

\begin{document}
\title{\bf\Large Matrix-Weighted Besov--Triebel--Lizorkin Spaces of
Optimal Scale: Boundedness of Pseudo-Differential, Trace, and
Calder\'{o}n--Zygmund Operators
\footnotetext{\hspace{-0.35cm} 2020 {\it Mathematics Subject Classification}.
Primary 42B20; Secondary 46E35, 35S05, 47A56, 46E40, 42B35.
\endgraf {\it Key words and phrases.}
matrix weight, generalized Besov--Triebel--Lizorkin-type space,
pseudo-differential operator, trace operator, Calder\'{o}n--Zygmund operator.
\endgraf This project is supported by
the National Key Research and Development Program of China
(Grant No. 2020YFA0712900),
the National Natural Science Foundation of China
(Grant Nos. 12431006 and 12371093),
the Fundamental Research Funds
for the Central Universities (Grant No. 2233300008).}}
\date{\small\emph{This article is dedicated to Professor Serge\u{\i} 
Mikha\u{\i}lovich Nikol'ski\u{\i} 
on his 120th birth anniversary}}
\author{Dachun Yang\footnote{Corresponding author, E-mail:
\texttt{dcyang@bnu.edu.cn}/{\color{red}\today}/Final version.}
,\ Wen Yuan and Mingdong Zhang}
\maketitle
\date{}
\maketitle

\vspace{-0.7cm}

\begin{center}
\begin{minipage}{13cm}
{\small {\bf Abstract}\quad
This article is a continuation of our work on generalized
matrix-weighted Besov--Triebel--Lizorkin-type
spaces with matrix $\mathcal{A}_{\infty}$ weights.
In this article, we establish the boundedness of
pseudo-differential, trace, and Calder\'{o}n--Zygmund operators
on these spaces. The main tools involved in this article are
the molecular and the wavelet characterizations of these spaces.
Since generalized matrix-weighted
Besov--Triebel--Lizorkin-type spaces
include many classical function spaces such as
matrix-weighted Besov--Triebel--Lizorkin spaces,
all the results in this article are of wide generality.}
\end{minipage}
\end{center}


\vspace{0.2cm}

\section{Introduction\label{s0}}
This article is a continuation of \cite{byyz24}
in which we introduced  the generalized matrix-weighted
Besov--Triebel--Lizorkin-type spaces with matrix
$\mathcal{A}_{\infty}$ weights  and established their  $\varphi$-transform characterization,
molecular, and wavelet characterizations, as well as
the Sobolev-type embedding properties. Based on these work,
in this article, we aim to establish the boundedness of
pseudo-differential, trace, and Calder\'{o}n--Zygmund operators
on these spaces.
Throughout the whole article, we work in $\mathbb{R}^n$ and,
unless necessary,
we will not explicitly specify this underlying space.

We first present a brief history of matrix weights. Recall that the  matrix-weighted
Lebesgue space $L^2(W)$ on $[0,2\pi]$ was first introduced by Wiener and Masani
\cite[Section 4]{wm58} in their study on the prediction theory of
multivariate stochastic processes. Later,
Treil and Volberg \cite{tv97} formulated
the matrix $\mathcal{A}_2$ condition on $\mathbb{R}$
and showed that the Hilbert transform is bounded on
$L^2(W)$ if and only if $W$ satisfies the matrix
$\mathcal{A}_2$ condition. Subsequently,
for any $p\in(1, \infty)$, Nazarov and Treil \cite{nt97}
and Volberg \cite{vol97}  determined the
matrix $\mathcal{A}_p$ condition on $\mathbb{R}$
and generalized the above results
from $p=2$ to any $p\in(1,\infty)$.
For more results related to the matrix $\mathcal{A}_p$
weights, we refer to \cite{dhl20, gol03, llor23, llor24, nie25}.
It is worth pointing out that, for any $p\in(1, \infty)$,
Volberg \cite{vol97} also introduced the
matrix $\mathcal{A}_{p,\infty}$ class,
which can be regarded as a natural counterpart of the
scalar $A_{\infty}$ class of Muckenhoupt
in the matrix-weighted setting. Recently,
Bu et al. \cite{bhyy4} studied the matrix
$\mathcal{A}_{p,\infty}$ class with $p\in(0,\infty)$  and
obtained several characterizations and properties of
matrix $\mathcal{A}_{p,\infty}$ weights.

It is well known that Besov spaces $\dot{B}^s_{p,q}$
and Triebel--Lizorkin spaces $\dot{F}^s_{p,q}$
are fundamental and important in harmonic analysis
and partial differential equations
(see, for example, \cite{hms22,hs12,hs13,ky94,lxy14,
ly13,maz03}). These spaces
unify many classical spaces such as Lebesgue spaces,
Lipschitz spaces, Hardy spaces, and $\operatorname{BMO}$
(the space of all locally integrable functions on $\mathbb{R}^n$
with bounded mean oscillation). For more details about
Besov--Triebel--Lizorkin spaces, we refer to the
monographs \cite{tri83, tri92, tri06} of Triebel.
Recently, BTL spaces were also extended to the mixed-norm
and product settings (see, for example, \cite{cgn19,gjn17,gkp21,gn16}).
Later, to solve an open question raised in \cite{dx04}
on the relation between $Q$ spaces and classical function spaces,
Yang et al. \cite{yy08,yy10} introduced
Besov-type spaces $\dot{B}^{s,\tau}_{p,q}$
and Triebel--Lizorkin-type spaces $\dot{F}^{s,\tau}_{p,q}$
with the new Morrey index $\tau\in[0, \infty)$
and developed the real-variable theory of these spaces.
For more study of Besov--Triebel--Lizorkin-type spaces,
we refer to \cite{yy13,yyz14,ysy10,yhmsy15,yhsy15a,yhsy15b}.

Based on the study of $L^p(W)$, a natural question is
to extend the theory of Besov--Triebel--Lizorkin spaces
and further Besov--Triebel--Lizorkin-type spaces
to the matrix-weighted setting. For any $s\in\mathbb{R}$,
$p\in(0, \infty)$, $q\in(0,\infty]$,
and $W\in\mathcal{A}_{p}$,
Frazier and Roudenko \cite{fr04, rou03, rou04} first studied
the matrix-weighted Besov space $\dot{B}^{s}_{p, q}(W)$.
Later, for any $s\in\mathbb{R}$, $p\in(0,\infty)$,
$q\in(0,\infty]$, and $W\in\mathcal{A}_{p}$,
Frazier and Roudenko \cite{fr21}
investigated the matrix-weighted Triebel--Lizorkin space
$\dot{F}^{s}_{p, q}(W)$ and obtained the
Littlewood--Paley theory of $L^p(W)$ for any $p\in(1, \infty)$.
Recently, for any $A\in\{B, F\}$, $s\in\mathbb{R}$, $\tau\in[0,\infty)$,
$p\in(0,\infty)$, $q\in(0,\infty]$, and $W\in\mathcal{A}_{p}$,
Bu et al. \cite{bhyy1,bhyy2,bhyy3}
introduced the matrix-weighted Besov--Triebel--Lizorkin-type space
$\dot{A}^{s,\tau}_{p,q}(W)$ and developed their real-variable theory including
the $\varphi$-transform characterization,
the molecular and the wavelet characterizations, and the
boundedness of pseudo-differential, trace, and Calder\'{o}n--Zygmund
operators. In particular, $\dot{A}^{s,0}_{p,q}(W)$
is exactly the matrix-weighted Besov--Triebel--Lizorkin space
$\dot{A}^{s}_{p,q}(W)$. For more progress on function spaces
associated with matrix $\mathcal{A}_p$ weights, we refer to
\cite{bx24a,bx24b,bx24c,byy23,lyy24a,lyy24b,wgx25a,wgx25b,wyy23}.
Recently, for any $A\in\{B, F\}$, $s\in\mathbb{R}$,
$\tau\in[0,\infty)$, $p\in(0,\infty)$, $q\in(0,\infty]$,
and $W\in\mathcal{A}_{p, \infty}$,
Bu et al. \cite{bhyy5} also studied matrix-weighted
Besov--Triebel--Lizorkin-type spaces $A^{s,\tau}_{p,q}(W)$.
To establish the well-known invariance property
of Triebel--Lizorkin spaces on the integrable index
(see \cite[Corollary 5.7]{fj90}) in the matrix-weighted setting, for any $s\in\mathbb{R}$,
$p\in(0,\infty)$, $q\in(0,\infty]$, and $W\in\mathcal{A}_{p, \infty}$,
we \cite{byyz24} introduced the generalized
matrix-weighted Besov--Triebel--Lizorkin-type space $\dot{A}^{s,\upsilon}_{p,q}(W)$ and
characterized them in terms of the $\varphi$-transform,
the Peetre-type maximal function, the Littlewood--Paley functions,
as well as  molecules and  wavelets,
where  $A\in\{B, F\}$ and $\upsilon$ is a growth function.
The spaces $\dot{A}^{s,\upsilon}_{p,q}(W)$
contain matrix-weighted Besov--Triebel--Lizorkin-type
spaces $\dot{A}^{s,\tau}_{p,q}(W)$, particularly matrix-weighted
Besov--Triebel--Lizorkin spaces $\dot{A}^{s}_{p,q}(W)$, as special cases,
and hence are of wide generality.

As a continuation of the study of $\dot{A}^{s,\upsilon}_{p,q}(W)$
in \cite{byyz24}, in this article, using
the molecular and the wavelet characterizations of
$\dot{A}^{s,\upsilon}_{p,q}(W)$ obtained in \cite{byyz24},
we establish the boundedness of some important operators
on $\dot{A}^{s,\upsilon}_{p,q}(W)$, including pseudo-differential,
trace, and Calder\'{o}n--Zygmund operators.
Due to the generality of $\dot{A}^{s,\upsilon}_{p,q}(W)$
(see \cite[Subsection 2.3]{byyz24}), these
boundedness results of operators  are  of wide generality and, moreover,
even reduced to the scalar-valued setting, they improve
some related known results in \cite{syy10}
on the spaces $\dot{A}^{s,\tau}_{p,q}$.

The organization of the remainder of this article is as follows.

In Section \ref{s1}, we first recall the concepts of matrix
$\mathcal{A}_{p,\infty}$ weights and growth functions.
Next, we present the definition of
generalized matrix-weighted Besov--Triebel--Lizorkin-type spaces
$\dot{A}^{s,\upsilon}_{p,q}(W)$
with $W\in\mathcal{A}_{p,\infty}$. In Section \ref{s2},
we first recall the defintion of  generalized matrix-weighted
Besov--Triebel--Lizorkin-type sequence spaces
$\dot{a}^{s,\upsilon}_{p,q}(W)$ corresponding to
$\dot{A}^{s,\upsilon}_{p,q}(W)$ and the result of
the boundedness of almost diagonal operators on
$\dot{a}^{s,\upsilon}_{p,q}(W)$ obtained in  \cite{byyz24}.
Applying this, in Subsection \ref{s-2-1},
we establish the boundedness of pseudo-differential operators
on $\dot{A}^{s,\upsilon}_{p,q}(W)$ (see Theorem \ref{thm-T-pd})
whose proof relies on the molecular characterization of
$\dot{A}^{s,\upsilon}_{p,q}(W)$.
In Subsection \ref{s-2-2}, via the wavelet characterization of $\dot{A}^{s,\upsilon}_{p,q}(W)$,
we obtain the boundedness of the trace operator and the extension
operator on $\dot{A}^{s,\upsilon}_{p,q}(W)$
(see Theorems \ref{thm-Tr-B}, \ref{thm-Ext-B}, \ref{thm-Tr-F}, and  \ref{thm-Ext-F}). In Subsection \ref{s-2-3}, we establish the boundedness of Calder\'{o}n--Zygmund operators on $\dot{A}^{s,\upsilon}_{p,q}(W)$
(see Theorem \ref{thm-CZ}) by determining the conditions on
Calder\'{o}n--Zygmund operators which ensure that  these operators map
sufficient smooth atoms to synthesis molecules for
$\dot{A}^{s,\upsilon}_{p,q}(W)$.

At the end of this section, we make some conventions on notation.
Let $\mathbb{N}:=\{1,2,\dots\}$, $\mathbb{Z}_+:=\mathbb{N}\cup\{0\}$, and $\mathbb{Z}$ be the set of all integers.
All the cubes $Q\subset\mathbb{R}^n$ in this article
are always assumed to have edges parallel to the coordinate
axes. For any cube $Q\subset\mathbb{R}^n$,
let $c_Q$ be its \emph{center}, $\ell(Q)$
be its \emph{edge length}, and $j_Q:=-\log_{2}\ell(Q)$.
For any cube $Q\subset\mathbb{R}^n$ and any $r\in(0,\infty)$,
let $rQ$ be the cube with the same center as $Q$
and the edge length $r\ell(Q)$.
Let $\mathcal{D}:=\{Q_{j, k}\}_{j\in\mathbb{Z}, k\in{\mathbb{Z}}^n}
:=\{2^{-j}([0, 1)^n+k)\}_{j\in\mathbb{Z}, k\in{\mathbb{Z}}^n}$
be the set of all \emph{dyadic cubes} in $\mathbb{R}^n$.
For any $j\in\mathbb{Z}$, let
$\mathcal{D}_j:=\{Q_{j, k}:\ k\in{\mathbb{Z}}^n\}$.
Let $\varphi$ be a complex-valued function on $\mathbb{R}^n$.
For any $j\in\mathbb{Z}$ and $x\in\mathbb{R}^n$,
let $\varphi_j(x):=2^{jn}\varphi(2^j x)$.
For any $j\in\mathbb{Z}$, $k\in\mathbb{Z}^n$,
and $Q:=Q_{j, k}\in\mathcal{D}$, let
$x_Q:=2^{-j}k$ and, for any $x\in\mathbb{R}^n$, let
\begin{align}\label{eq-phi_Q}
\varphi_Q(x):=2^{\frac{jn}{2}} \varphi\left(2^j x-k\right)=|Q|^{\frac{1}{2}}
\varphi_j\left(x-x_Q\right).
\end{align}
For any $p,q\in\mathbb{R}$,
let $p\wedge q:=\min\{p, q\}$ and $p\vee q:=\max\{p, q\}$.
Let $\mathbf{0}$ denote the \emph{origin}
of $\mathbb{R}^n$ or $\mathbb{C}^m$.
For any measurable set $E\subset\mathbb{R}^n$
with $|E|\in(0,\infty)$ and any
measurable function $f$ on $\mathbb{R}^n$, let
$\fint_E f(x)\,dx:=\frac{1}{|E|}\int_{E}f(x)\,dx$.
For any $p\in(0,\infty]$ and any measurable set
$E\subset\mathbb{R}^n$, the \emph{Lebesgue space} $L^p(E)$
is defined to be the set of all complex-valued
measurable functions $f$ on $E$ such that
\begin{align*}
\|f\|_{L^p(E)}:=
\begin{cases}
\displaystyle{\left[\int_E|f(x)|^p\, dx\right]^{\frac{1}{p}}}
& \text{if } p \in(0, \infty),\\
\displaystyle{\mathop{\operatorname{ess\,sup}}_{x\in E}|f(x)|}
& \text{if } p=\infty
\end{cases}
\end{align*}
is finite. In accordance with our previous agreement,
we denote $L^p(\mathbb{R}^n)$ simply by $L^p$.
For any $x\in\mathbb{R}^n$ and $r\in(0,\infty)$,
let $B(x,r):=\{y\in\mathbb{R}^n:\,|x-y|<r\}$.
For any multi-index $\gamma:=(\gamma_1,\dots,\gamma_n)\in\mathbb{Z}_{+}^n$
and any $x:=(x_1,\dots,x_n)\in\mathbb{R}^n$, let
$|\gamma|:=\gamma_1+\dots+\gamma_n$,
$x^{\gamma}:=x^{\gamma_1}_1\cdots x^{\gamma_n}_n$, and
$\partial^{\gamma}:=(\frac{\partial}{\partial x_1})^{\gamma_1}
\cdots(\frac{\partial}{\partial x_n})^{\gamma_n}$.
The symbol $C$ denotes a positive constant that is independent of the main
parameters involved, but may vary from line to line.
The symbol $A\lesssim B$ means that $A\leq CB$ for
some positive constant $C$, while $A\sim B$ means $A\lesssim B\lesssim A$.
Finally, in all proofs, we consistently retain the notation
introduced in the original theorem (or related statement).

\section{Generalized Matrix-Weighted Besov-Type and Triebel--Lizorkin-type Spaces\label{s1}}

In this section, we first  recall the concepts of matrix
$\mathcal{A}_{p,\infty}$ weights and growth functions. Then we give the definition of
generalized matrix-weighted Besov-type and Triebel--Lizorkin-type spaces.

We begin with some concepts related to square matrices.
In what follows, we always use $m\in\mathbb{N}$ to
denote the dimension of vectors or the order of matrices.
The set of all $m\times m$ complex-valued matrices
is denoted by $M_m(\mathbb{C})$. For any $A\in M_m(\mathbb{C})$,
let $A^*$ be the conjugate transpose of $A$.
For any $A\in M_m(\mathbb{C})$, we say that
$A$ is a \emph{unitary matrix} if $A^*A=I_m$,
where $I_m$ is the identity matrix of the order $m$.
A matrix $A\in M_m(\mathbb{C})$ is said to
be \emph{positive semidefinite} if,
for any $\vec{z}\in\mathbb{C}^m$, $\vec{z}^*A\vec{z}\geq0$
and to be \emph{positive definite}
if, for any $\vec{z}\in\mathbb{C}^m\setminus
\{\mathbf{0}\}$, $\vec{z}^*A\vec{z}>0$ (see, for example,
\cite[(7.1.1a) and (7.1.1b)]{hj13}).
Using \cite[Theorems 2.5.6 and 7.2.1]{hj13},
we conclude that, for any given
positive definite matrix $A\in M_m(\mathbb{C})$,
there exists a unitary matrix $U\in M_m(\mathbb{C})$ such that
$A=U\operatorname{diag}(\lambda_1,\ldots,\lambda_m)U^{*}$,
where $\{\lambda_i\}_{i=1}^m$ in $(0,\infty)$ are
all the eigenvalues of $A$.
For any $\alpha\in\mathbb{R}$, let
$A^{\alpha}:=U\operatorname{diag}
(\lambda^{\alpha}_1,\ldots,\lambda^{\alpha}_m)U^{*}$.
By the argument in \cite[p.\,408]{hj94},
we find that $A^\alpha$ is independent of the
choice of $U$ and hence well-defined.

We denote the set of all $m\times m$ positive semidefinite
complex-valued matrices by $D_m(\mathbb{C})$.
A matrix-valued function $W:\,\mathbb{R}^n\to D_m(\mathbb{C})$
is called a \emph{matrix weight} if,
for almost every $x\in\mathbb{R}^n$,
$W(x)$ is positive definite and all the entries of
$W$ are locally integrable functions on $\mathbb{R}^n$
(see, for example, \cite{nt97,vol97}).
For any $A\in M_m(\mathbb{C})$,
the \emph{operator norm} $\|A\|$ of $A$ is defined by setting
$\|A\|:=\sup_{\vec z\in\mathbb{C}^m, |\vec z|=1}|A\vec z|$.
For any $p\in(1,\infty)$, the matrix $\mathcal{A}_{p,\infty}$
class was introduced in \cite[(2.2)]{vol97}.
The following equivalent definition of the
matrix $\mathcal{A}_{p,\infty}$ class for any $p\in(0,\infty)$
was established in \cite[Proposition 3.7]{bhyy4}.

\begin{definition}
Let $p\in(0,\infty)$. A matrix weight
$W$ is called an $\mathcal{A}_{p,\infty}(\mathbb{R}^n,
\mathbb{C}^m)$-\emph{matrix weight},
denoted by $W\in\mathcal{A}_{p,\infty}(\mathbb{R}^n,
\mathbb{C}^m)$, if $W$ satisfies that,
for any cube $Q\subset\mathbb{R}^n$,
\begin{align*}
\max\left\{\log\left(\fint_Q\left\|W^{\frac{1}{p}}(x)
W^{-\frac{1}{p}}(\cdot)\right\|^p\,dx\right), 0\right\}\in L^1(Q)
\end{align*}
and
\begin{align*}
[W]_{\mathcal{A}_{p,\infty}(\mathbb{R}^n,\mathbb{C}^m)}
:=\sup_{\mathrm{cube}\,Q\subset\mathbb{R}^n}\exp\left(\fint_Q\log
\left(\fint_Q\left\|W^{\frac{1}{p}}(x)W^{-\frac{1}{p}}(y)
\right\|^p\,dx\right)\,dy\right)<\infty.
\end{align*}
When no confusion arises, we simply write
$W\in\mathcal{A}_{p,\infty}$.
\end{definition}

Recall that a locally integrable function $w$ on $\mathbb{R}^n$
is called a \emph{scalar weight} if, for almost every
$x\in\mathbb{R}^n$, $w(x)\in(0, \infty)$
(see, for example, \cite[p.\,499]{gra14a}).
In the study of scalar weights, an important concept
is the scalar $A_{\infty}$ class of Muckenhoupt,
which is defined to be the set of all scalar weights $w$ such that
\begin{align*}
\left[w\right]_{A_{\infty}}&:=
\sup_{\mathrm{cube}\,Q\subset\mathbb{R}^n}\fint_Qw(x)\,dx
\exp\left(\fint_Q \log\left(\left[w(x)
\right]^{-1}\right)\, dx\right)<\infty
\end{align*}
(see, for example, \cite[Definition 7.3.1 and Theorem 7.3.3]{gra14a}
for more equivalent definitions of $A_\infty$).
Observe that, for any $p\in(0,\infty)$,
$\mathcal{A}_{p,\infty}(\mathbb{R}^n,\mathbb{C})=A_\infty$.
Based on this fact, for any $p\in(0,\infty)$,
the matrix $\mathcal{A}_{p,\infty}$ class
can be regarded as a natural counterpart of $A_{\infty}$
in the matrix-weighted setting.

To give the definition of
generalized matrix-weighted Besov--Triebel--Lizorkin-type spaces
$\dot{A}^{s,\upsilon}_{p,q}(W)$,
we also need the following concept of growth functions, which
was introduced in \cite[Definition 2.3]{byyz24}.

\begin{definition}\label{def-grow-func}
Let $\delta_1, \delta_2,\omega\in\mathbb{R}$.
A function $\upsilon:\,\mathcal{D}\to (0, \infty)$
is called a \emph{$(\delta_1, \delta_2; \omega)$-order
growth function} if there exists a positive constant
$C$ such that, for any $Q, R\in\mathcal{D}$,	
\begin{align*}
\frac{\upsilon(Q)}{\upsilon(R)}\leq C
\left[1+\frac{|x_Q-x_R|}
{\ell(Q)\vee \ell(R)}\right]^{\omega}
\begin{cases}
\displaystyle{\left(\frac{|Q|}{|R|}\right)^{\delta_1}}
& \text{if } \ell(Q) \leq \ell(R), \\
\displaystyle{\left(\frac{|Q|}{|R|}\right)^{\delta_2}}
& \text{if } \ell(R)<\ell(Q).
\end{cases}
\end{align*}
The class of all $(\delta_1, \delta_2; \omega)$-order
growth functions is denoted by
$\mathcal{G}(\delta_1, \delta_2; \omega)$.
\end{definition}

There are some typical examples of  growth functions (see \cite[Example 2.14]{byyz24}).

\begin{example}\label{examp}
\begin{itemize}
\item[{\rm (i)}]Let $\tau\in[0, \infty)$ and,
for any $Q\in\mathcal{D}$, let $\upsilon(Q):={|Q|}^{\tau}$.
Then $\upsilon\in\mathcal{G}(\tau, \tau; 0)$.
\item[{\rm (ii)}]
Let $p\in(0, \infty)$ and $\mathcal{G}_p$
be the set of all nondecreasing functions
$g:\,(0, \infty)\to(0, \infty)$ such that,
for any $t_1, t_2\in(0,\infty)$ with $t_1\leq t_2$,
$g(t_1){t_1}^{-\frac{n}{p}}\geq g(t_2){t_2}^{-\frac{n}{p}}$.
Let $g\in\mathcal{G}_p$ and, for any $Q\in\mathcal{D}$,
let $\upsilon(Q):=g(\ell(Q))$.
Then $\upsilon\in\mathcal{G}(0, \frac1{p}; 0)$.
\item[{\rm (iii)}] Let $w\in A_\infty$ and, for any
$Q\in\mathcal{D}$, let $\upsilon(Q):=\int_{Q}w(x)\,dx$.
Then there exist $p\in[1, \infty)$ and $\delta\in(0, 1)$
such that $\upsilon\in\mathcal{G}(\delta, p; n(p-\delta))$.
\end{itemize}
\end{example}

Let $\mathcal{S}$ be the set of
all Schwartz functions on $\mathbb{R}^n$
equipped with the well-known topology
induced by a countable family of norms
(see, for example, \cite[Proposition 8.2]{fol99})
and $\mathcal{S}'$ be the set of
all tempered distributions on $\mathbb{R}^n$
equipped with the weak-$\ast$ topology.
For any function $f$ on $\mathbb{R}^n$,
let $\operatorname{supp}f:=\overline{\{x\in\mathbb{R}^n:\,f(x)\neq0\}}$
be the \emph{support} of $f$. For any $f\in L^1$,
its \emph{Fourier transform} $\widehat{f}$
is defined by setting, for any $\xi\in\mathbb{R}^n$,
$\widehat{f}(\xi):=\int_{\mathbb{R}^n}f(x)
e^{-i x\cdot \xi}\,dx$, where $i:=\sqrt{-1}$.
Let $\varphi\in\mathcal{S}$ satisfy
\begin{align}\label{cond1}
\operatorname{supp}\widehat{\varphi}\subset
\left\{\xi\in\mathbb{R}^n:\,\frac{1}{2}\leq|\xi|\leq2\right\}
\ \text{and}\ \min\left\{\left|\widehat{\varphi}(\xi)\right|:\,
\frac{3}{5}\leq|\xi|\leq\frac{5}{3}\right\}>0.
\end{align}
From \cite[Lemma (6.9)]{fjw91}, it follows that
there also exists $\psi\in\mathcal{S}$
satisfying \eqref{cond1} such that,
for any $\xi\in\mathbb{R}^n\backslash\{\mathbf{0}\}$,
\begin{align}\label{cond3}
\sum_{j\in\mathbb{Z}}\overline{\widehat{\varphi}\left(2^j\xi\right)}
\widehat{\psi}\left(2^j\xi\right)=1.
\end{align}

Suppose that $p,q\in(0,\infty]$. For any sequence
$\{f_j\}_{j\in\mathbb{Z}}$ of measurable functions
on $\mathbb{R}^n$, let
\begin{align*}
\|\{f_j\}_{j\in\mathbb{Z}}\|_{L\dot{B}_{p, q}}
:=\left[\sum_{j\in\mathbb{Z}}\|f_j\|
_{L^p}^q\right]^{\frac{1}{q}}
\text{\ \ and\ \ }
\|\{f_j\}_{j\in\mathbb Z}\|_{L\dot{F}_{p, q}}
:=\left\|\left(\sum_{j\in\mathbb{Z}}|f_j|^q
\right)^{\frac{1}{q}}\right\|_{L^p}
\end{align*}
(with the usual modification made if $q=\infty$).
In what follows, for any $j_0\in\mathbb{Z}$,
let $\mathbf{1}_{j\geq j_0}:=\mathbf{1}_{[j_0,\infty)}(j)$.
Furthermore, assume that $A\in\{B, F\}$ and
$\upsilon$ is a positive function defined on
$\mathcal{D}$. For any sequence
$\{f_j\}_{j\in\mathbb{Z}}$ of measurable functions
on $\mathbb{R}^n$, let
\begin{align}\label{LA_nu}
\|\{f_j\}_{j\in\mathbb{Z}}\|_{L\dot{A}_{p, q}^{\upsilon}}
:=\sup_{P\in\mathcal{D}}\frac{1}{\upsilon(P)}\|\{f_j\mathbf{1}_P
\mathbf{1}_{j\geq j_P}\}_{j\in\mathbb Z}\|_{L\dot{A}_{p, q}}.
\end{align}
By \cite[Propositions 2.7 and 2.11]{byyz24}, we find that,
when considering \eqref{LA_nu} with
$\upsilon$ being a growth function  as in
Definition \ref{def-grow-func}, it is enough to
suppose $\upsilon\in\mathcal{G}(\delta_1,
\delta_2; \omega)$ with
\begin{align}\label{eq-delta1<0}
\delta_2\in[0, \infty),\ \delta_1\in(-\infty, \delta_2],
\text{ and } \omega\in[0, n(\delta_2-\delta_1)].
\end{align}

As in \cite{fjw91}, let
\begin{align*}
\mathcal{S}_{\infty}:=
\left\{\varphi\in\mathcal{S}
:\,\int_{\mathbb{R}^n}\varphi(x)x^{\gamma}\,dx=0
\ {\rm for\ any}\ \gamma\in{\mathbb{Z}}_{+}^n\right\}
\end{align*}
equipped with the same topology as $\mathcal{S}$
and let $\mathcal{S}_{\infty}'$ be the
\emph{dual space} of $\mathcal{S}_{\infty}$
equipped with the weak-$\ast$ topology.
It is well known that $\mathcal{S}_{\infty}'
=\mathcal{S}'/\mathcal{P}$
as topological spaces, where $\mathcal{P}$ denotes
the set of all polynomials on $\mathbb{R}^n$
(see \cite[Propostion 1.1.3]{gra14b} and
\cite[Proposition 8.1]{ysy10} for more details).
For any $\vec{f}:=(f_1, \dots, f_m)^{T}
\in(\mathcal{S}'_{\infty})^m$
and $\varphi\in\mathcal{S}_{\infty}$,
let $\varphi*\vec{f}:=(\varphi*f_1, \dots, \varphi*f_m)^{T}$.

We now recall the definition of generalized matrix-weighted
Besov--Triebel--Lizorkin-type spaces $\dot{A}^{s,\upsilon}_{p,q}(W)$ introduced in \cite[Definition 2.1]{byyz24}.

\begin{definition}\label{MWBTL}
Let $s\in\mathbb{R}$, $p\in(0, \infty)$,
$q\in(0, \infty]$, and $W\in\mathcal{A}_{p,\infty}$.
Assume that $\delta_1,\delta_2,\omega$
satisfy \eqref{eq-delta1<0},
$\upsilon\in\mathcal{G}(\delta_1, \delta_2; \omega)$,
and $\varphi\in\mathcal{S}$ satisfies \eqref{cond1}.
The \emph{generalized matrix-weighted Besov-type space}
$\dot{B}_{p,q}^{s, \upsilon}(W)$ and the
\emph{generalized matrix-weighted Triebel--Lizorkin-type space}
$\dot{F}_{p,q}^{s, \upsilon}(W)$ are respectively defined to be the sets
of all $\vec{f}\in(\mathcal{S}'_{\infty})^m$ with
\begin{align*}	
\left\|\vec{f}\right\|_{\dot{A}_{p,q}^{s, \upsilon}(W)}	
:=\left\|\left\{2^{js}\left|W^{\frac1{p}}
\left(\varphi_j*\vec{f}\right)\right|\right\}
_{j\in\mathbb{Z}}\right\|_{L\dot{A}_{p, q}^{\upsilon}}<\infty,
\end{align*}
where $A\in\{B, F\}$ and $\|\cdot\|_{L\dot{A}_{p, q}^{\upsilon}}$
is as in \eqref{LA_nu}.
\end{definition}

\begin{remark}
Let all the symbols be the same as in Definition \ref{MWBTL}.
\begin{itemize}
\item[{\rm (i)}]  We first point out that
the space $\dot{A}_{p,q}^{s, \upsilon}(W)$ is indeed
independent of the choice of $\varphi$
(see \cite[Theorem 2.5(ii)]{byyz24} for more details).
\item[{\rm (ii)}] Let $\tau\in[0, \infty)$ and,
for any $Q\in\mathcal{D}$, let $\upsilon(Q):=|Q|^{\tau}$.
From Example \ref{examp}(i), it follows that
$\upsilon\in\mathcal{G}(\tau, \tau; 0)$; then
the space $\dot{A}_{p,q}^{s,\upsilon}(W)$ in this case
reduces to the space $\dot{A}_{p,q}^{s, \tau}(W)$
introduced in \cite[Definition 3.5]{bhyy1}.
Furthermore, the spaces $\dot{B}_{p,q}^{s, 0}(W)$
and $\dot{F}_{p,q}^{s,0}(W)$ are precisely matrix-weighted
Besov--Triebel--Lizorkin spaces
$\dot{B}_{p,q}^{s}(W)$ and $\dot{F}_{p,q}^{s}(W)$
introduced, respectively, in \cite[Definition 1.1]{rou03}
and \cite[p.\,489, (i)]{fr21}.
For more examples of spaces in Definition \ref{MWBTL},
we refer to \cite[Subsection 2.3]{byyz24}.
Moreover, when $m=1$ (the scalar-valued setting) and $W\equiv1$,
the space $\dot{A}_{p,q}^{s, \tau}(W)$ reduces to
the classical Besov--Triebel--Lizorkin-type space
$\dot{A}_{p,q}^{s, \tau}$
studied in \cite{yy08, yy10, ysy10}.
\end{itemize}
\end{remark}

\section{Boundedness of Pseudo-Differential, Trace, and \\
Calder\'{o}n--Zygmund Operators on
$\dot{A}_{p, q}^{s, \upsilon}(W)$}\label{s2}

In this section, we establish
the boundedness of pseudo-differential, trace,
and Calder\'{o}n--Zygmund operators on $\dot{A}_{p, q}^{s, \upsilon}(W)$,
respectively, in Subsections \ref{s-2-1}, \ref{s-2-2}, and \ref{s-2-3}.
To this end, we fully use the molecular and
the wavelet characterizations of $\dot{A}_{p, q}^{s, \upsilon}(W)$
established in \cite{byyz24}.

It is well known that many real-variable characterizations
of function spaces can be reduced to
the boundedness of almost diagonal operators on
their corresponding sequence spaces
(see, for example, \cite{bow05, bow07, bh06,
bhyy2, bhyy3, bhyy5, fj90, syy24}).
Following this idea, Bu et al. \cite{byyz24} used
the boundedness of almost diagonal operators
on $\dot{a}_{p, q}^{s, \upsilon}(W)$ [sequence spaces
corresponding to $\dot{A}_{p, q}^{s, \upsilon}(W)$]
to obtain the molecular and
the wavelet characterizations of $\dot{A}_{p, q}^{s, \upsilon}(W)$.
On the other hand, the indices of almost diagonal operators,
for which these operators are bounded on $\dot{a}_{p, q}^{s, \upsilon}(W)$,
are also important because these indices naturally
appear in the molecular and the wavelet characterizations
of $\dot{A}_{p, q}^{s, \upsilon}(W)$ and hence in
the results on the boundedness of
pseudo-differential, trace, and Calder\'{o}n--Zygmund
operators on $\dot{A}_{p, q}^{s, \upsilon}(W)$.

To give the result on the boundedness of
almost diagonal operators on $\dot{a}_{p,q}^{s,\upsilon}(W)$,
we start with some symbols. Let $U:=\{u_{Q,R}\}_{Q,R\in\mathcal{D}}$
be a sequence in $\mathbb{C}$.
For any sequence $\vec{t}:=\{\vec{t}_R\}_{R\in\mathcal{D}}$
in $\mathbb{C}^m$,
we define $U\vec{t}:=\{(U\vec{t})_Q\}_{Q\in\mathcal{D}}$
by setting, for any $Q\in\mathcal{D}$,
$(U\vec{t})_Q:=\sum_{R\in\mathcal{D}}u_{Q,R}\vec{t}_R$
if this summation is absolutely convergent.
The following concept of almost diagonal operators
was introduced in \cite[Definition 4.1]{bhyy2},
which is a slight generalization of the traditional one
as in \cite[(3.1)]{fj90}.

\begin{definition}
Let $D,E,F\in\mathbb{R}$.
The infinite matrix $U^{DEF}:=
\{u^{DEF}_{Q,R}\}_{Q, R\in\mathcal{D}}$ is
defined by setting,
for any $Q, R\in\mathcal{D}$,
\begin{align*}
u^{DEF}_{Q,R}:=\left[1+\frac{|x_Q-x_R|}
{\ell(Q)\vee \ell(R)}\right]^{-D}
\begin{cases}
\displaystyle{\left[\frac{\ell(Q)}{\ell(R)}\right]^E}
& \text{if } \ell(Q)\leq\ell(R), \\
\displaystyle{\left[\frac{\ell(R)}{\ell(Q)}\right]^F}
& \text{if } \ell(R)<\ell(Q).
\end{cases}
\end{align*}
A complex infinite matrix $U:=\{u_{Q,R}\}_{Q,R\in\mathcal{D}}$
is said to be $(D, E, F)$-\emph{almost diagonal}
if there exists a positive constant $C$ such that,
for any $Q, R\in\mathcal{D}$, $|u_{Q,R}| \leq C u^{DEF}_{Q,R}$.
\end{definition}

Before presenting the  result on the boundedness of
almost diagonal operators on $\dot{a}_{p,q}^{s,\upsilon}(W)$,
we also need to recall the concept of reducing operators,
which was introduced in \cite[(3.1)]{vol97}.

\begin{definition}\label{def-red-ope}
Let $p\in(0,\infty)$ and $W$ be a matrix weight.
A sequence $\{A_Q\}_{Q\in\mathcal{D}}$ of
positive definite matrices is called
\emph{a sequence of reducing operators of order $p$ for $W$}
if, for any $Q\in\mathcal{D}$ and $\vec{z}\in\mathbb{C}^m$,
\begin{align*}
\left|A_Q\vec z\right|
\sim\left[\fint_Q\left|W^{\frac{1}{p}}(x)\vec{z}\right|^p
\,dx\right]^{\frac{1}{p}},
\end{align*}
where the positive equivalence constants
are independent of $Q$ and $\vec{z}$.
\end{definition}

For the existence of reducing operators, we refer to
\cite[Proposition 1.2]{gol03} for any $p\in(1, \infty)$
and \cite[p.\,1237]{fr04} for any $p\in(0, 1]$.
For any $p\in(0, \infty)$, any $W\in\mathcal{A}_{p,\infty}$,
and any sequence $\{A_Q\}_{Q\in\mathcal{D}}$
of reducing operators of order $p$ for $W$,
a sharp estimate for $\|A_QA^{-1}_{R}\|$ was
established in \cite{bhyy4}. To present it,
we first recall the concept  of
$\mathcal{A}_{p,\infty}$ dimensions  introduced in \cite[Definition 6.2]{bhyy4}.

\begin{definition}\label{def-upp-low-dimension}
Let $p\in(0,\infty)$ and $d\in\mathbb{R}$.
A matrix weight $W$ is said to have
\emph{$\mathcal{A}_{p,\infty}$-lower dimension $d$}
if there exists a positive constant $C$ such that,
for any $\lambda\in[1,\infty)$ and any cube $Q\subset\mathbb{R}^n$,
\begin{align*}
\exp\left(\fint_{\lambda Q}\log\left(\fint_Q
\left\|W^{\frac{1}{p}}(x)W^{-\frac{1}{p}}(y)
\right\|^p\,dx\right)\,dy\right)
\leq C{\lambda}^d.
\end{align*}
A matrix weight $W$ is said to have
\emph{$\mathcal{A}_{p,\infty}$-upper dimension $d$}
if there exists a positive constant $C$ such that,
for any $\lambda\in[1,\infty)$ and any cube $Q\subset\mathbb{R}^n$,
\begin{align*}
\exp\left(\fint_Q\log\left(\fint_{\lambda Q}
\left\|W^{\frac{1}{p}}(x)W^{-\frac{1}{p}}(y)
\right\|^p\,dx\right)\,dy\right)\leq C\lambda^d.
\end{align*}
\end{definition}

Suppose $p\in(0,\infty)$. From \cite[Propositions 6.4(ii)
and 6.5(ii)]{bhyy4}, it follows that,
for any $W\in\mathcal{A}_{p,\infty}$,
$W$ has $\mathcal{A}_{p,\infty}$-lower dimension $d_1\in[0,n)$
and $\mathcal{A}_{p,\infty}$-upper dimension $d_2\in[0,\infty)$.
Applying \cite[Propositions 6.4(i) and 6.5(i)]{bhyy4},
we conclude that the $\mathcal{A}_{p,\infty}$-lower and
$\mathcal{A}_{p,\infty}$-upper dimensions are both nonnegative.
By these results, for any $W\in\mathcal{A}_{p,\infty}$, let
\begin{align}\label{eq-low-dim}
d_{p,\infty}^{\mathrm{lower}}(W):=
\inf\left\{d\in[0,n):\ W\text{ has }\mathcal{A}_{p,\infty}
\text{-lower dimension } d\right\}
\end{align}
and
\begin{align}\label{eq-upp-dim}
d_{p,\infty}^{\mathrm{upper}}(W)
:=\inf\left\{d\in[0,\infty):\ W\text{ has }\mathcal{A}_{p,\infty}
\text{-upper dimension } d\right\}.
\end{align}
Furthermore, let
\begin{align*}
[\![d_{p,\infty}^{\mathrm{lower}}(W),n):=
\begin{cases}
[d_{p,\infty}^{\mathrm{lower}}(W),n)
&\text{if } W\text{ has the }\mathcal{A}_{p,\infty}
\text{-lower dimension }d_{p,\infty}^{\mathrm{lower}}(W),\\
(d_{p,\infty}^{\mathrm{lower}}(W),n)
&\text{otherwise}
\end{cases}
\end{align*}
and
\begin{align*}
[\![d_{p,\infty}^{\mathrm{upper}}(W),\infty):=
\begin{cases}
[d_{p,\infty}^{\mathrm{upper}}(W),\infty)
&\text{if } W\text{ has the }\mathcal{A}_{p,\infty}
\text{-upper dimension }d_{p,\infty}^{\mathrm{upper}}(W),\\
(d_{p,\infty}^{\mathrm{upper}}(W),\infty)
&\text{otherwise}.
\end{cases}
\end{align*}

The following sharp estimate for reducing operators generated
by matrix $\mathcal{A}_{p,\infty}$ weights was
established in \cite[Lemma 6.8(i)]{bhyy4}.

\begin{lemma}\label{growEST}
Let $p\in(0,\infty), W\in\mathcal{A}_{p,\infty}$,
and $\{A_Q\}_{Q\in\mathcal{D}}$
be a sequence of reducing operators of order $p$ for $W$.
If $\beta_1 \in \llbracket d_{p, \infty}^{\mathrm{lower}}(W),\infty)$
and $\beta_2\in \llbracket d_{p, \infty}^{\mathrm{upper}}(W),\infty)$,
then there exists a positive constant $C$ such that,
for any $Q, R\in\mathcal{D}$,
$$
\left\|A_Q A_R^{-1}\right\|^p\leq C\max
\left\{\left[\frac{\ell(R)}{\ell(Q)}\right]^{\beta_1},
\left[\frac{\ell(Q)}{\ell(R)}\right]^{\beta_2}\right\}
\left[1+\frac{|x_Q-x_R|}{\ell(Q)\vee\ell(R)}\right]^{\beta_1+\beta_2}.
$$
\end{lemma}

We next present sequence spaces corresponding to
$\dot{A}_{p, q}^{s, \upsilon}(W)$, which were introduced in
\cite[Definition 2.4]{byyz24}.

\begin{definition}
Let $s\in\mathbb{R}$, $p\in(0, \infty)$,
$q\in(0, \infty]$, and $W\in\mathcal{A}_{p,\infty}$.
Assume that $\delta_1,\delta_2$, and $\omega$ satisfy \eqref{eq-delta1<0}
and $\upsilon\in\mathcal{G}(\delta_1, \delta_2; \omega)$.
The \emph{generalized matrix-weighted Besov-type sequence space}
$\dot{b}_{p,q}^{s,\upsilon}(W)$ and the
\emph{generalized matrix-weighted
Triebel--Lizorkin-type sequence space}
$\dot{f}_{p,q}^{s,\upsilon}(W)$ are respectively
defined to be the sets of all sequences
$\vec{t}:=\{\vec{t}_Q\}_{Q\in\mathcal{D}}$
in $\mathbb{C}^m$ such that
\begin{align*}	
\left\|\vec{t}\right\|_{\dot{a}_{p,q}^{s,\upsilon}(W)}	
:=\left\|\left\{2^{js}\left|W^{\frac1{p}}
\vec{t}_j\right|\right\}_{j\in\mathbb{Z}}
\right\|_{L\dot{A}_{p, q}^{\upsilon}}<\infty,
\end{align*}
where $(A, a)\in\{(B, b), (F, f)\}$,
$\|\cdot\|_{L\dot{A}_{p, q}^{\upsilon}}$
is as in \eqref{LA_nu}, and, for any
$j\in\mathbb{Z}$ and $Q\in\mathcal{D}$,
$\widetilde{\mathbf{1}}_Q:=|Q|^{-\frac{1}{2}}\mathbf{1}_Q$ and
\begin{align}\label{veclambda_j}
\vec{t}_j:=\sum_{Q\in\mathcal{D}_j}
\widetilde{\mathbf{1}}_Q\vec{t}_Q.
\end{align}
\end{definition}

Suppose that $a\in\{b, f\}$, $p, q\in(0, \infty]$,
$\delta_1,\delta_2$, and $\omega$ satisfy \eqref{eq-delta1<0},
and $\upsilon\in\mathcal{G}(\delta_1, \delta_2; \omega)$.
The symbol $J_{\dot{a}_{p,q}^{s,\upsilon}}$,
introduced in \cite[(5.2)]{byyz24}, is given by
\begin{align}\label{J_nu}
J_{\dot{a}_{p,q}^{s,\upsilon}}:=
\begin{cases}
n
&{\rm if}\ \delta_1>\frac1{p}
\ {\rm or\ if}\ (\delta_1, q)=(\frac1{p}, \infty)
\ (\text{``supercritical case''}),\\
\displaystyle\frac{n}{\min \{1, q\}}	
&{\rm if}\ a=f,\  \delta_1=\delta_2=\frac1{p}, \text{ and } q<\infty
\ (\text{``critical case''}),\\
\displaystyle\frac{n}{1\wedge\Gamma_{p,q}}
&{\rm if}\ \delta_1<\frac1{p},\
{\rm or\ if}\ a=b,\ \delta_1=\delta_2=\frac1{p}, \text{ and } q<\infty,\\
&{\rm\quad or\ if}\ \delta_2>\delta_1=\frac1{p} \text{ and } q<\infty
\ (\text{``subcritical case''}).
\end{cases}
\end{align}
where
\begin{align*}
\Gamma_{p,q}:=
\begin{cases}
\displaystyle p	
&\text{if } a=b,\\
\displaystyle p\wedge q
&\text{if } a=f.
\end{cases}
\end{align*}

The following theorem is exactly \cite[Theorem 5.6]{byyz24},
which establishes the boundedness of
almost diagonal operators on $\dot{a}_{p,q}^{s,\upsilon}(W)$.

\begin{theorem}\label{a(W)adopebound}
Let $a\in\{b, f\}$, $s\in\mathbb{R}$, $p\in(0,\infty)$,
$q\in(0, \infty]$, $W\in \mathcal{A}_{p,\infty}$, and
$D,E,F\in\mathbb{R}$. Assume that $\delta_1,\delta_2$, and $\omega$
satisfy \eqref{eq-delta1<0} and
$\upsilon\in\mathcal{G}(\delta_1, \delta_2; \omega)$. Let
\begin{align*}
&\Delta:=\left[\delta_2-\frac1{p}+
\frac{d^{\operatorname{lower}}_{p, \infty}(W)}{np}\right]_{+},
\ D_{\dot{a}_{p,q}^{s,\upsilon}(W)}:=
J_{\dot{a}_{p,q}^{s,\upsilon}}+
\left[n\Delta\wedge\left(\omega+
\frac{d^{\operatorname{lower}}_{p, \infty}(W)}{p}\right)\right]
+\frac{d^{\operatorname{upper}}_{p, \infty}(W)}{p},\\
&E_{\dot{a}_{p,q}^{s,\upsilon}(W)}:=
\frac{n}{2}+s+n\Delta,\text{\ and\ }
F_{\dot{a}_{p,q}^{s,\upsilon}(W)}:=
J_{\dot{a}_{p,q}^{s,\upsilon}}
-\frac{n}{2}-s-n\left(\delta_1-\frac{1}{p}\right)_+
+\frac{d^{\operatorname{upper}}_{p, \infty}(W)}{p},
\end{align*}
where $J_{\dot{a}_{p,q}^{s,\upsilon}},
d^{\operatorname{lower}}_{p, \infty}(W)$,
and $d^{\operatorname{ upper}}_{p, \infty}(W)$ are
as, respectively, in \eqref{J_nu}, \eqref{eq-low-dim},
and \eqref{eq-upp-dim}. If
\begin{align*}
\displaystyle D>D_{\dot{a}_{p,q}^{s,\upsilon}(W)},\
\displaystyle E>E_{\dot{a}_{p,q}^{s,\upsilon}(W)},\
{\rm and}\
\displaystyle F>F_{\dot{a}_{p,q}^{s,\upsilon}(W)},	
\end{align*}	
then any $(D, E, F)$-almost diagonal operator
is bounded on $\dot{a}_{p,q}^{s,\upsilon}(W)$.
\end{theorem}

\begin{remark}\label{rmk-compare}
Let $p\in(0, \infty)$ and $\tau\in[0, \infty)$.
Let $\mathcal{A}_p$ denote the matrix Muckenhoupt class (see, for example, \cite[Definition 2.10]{bhyy4} for its definition).
By \cite[Proposition 4.2]{bhyy4}, we conclude that,
$\mathcal{A}_{p}\subsetneqq\mathcal{A}_{p,\infty}$.

Recall that, when $W\in\mathcal{A}_{p}$ and
$\upsilon(Q):=|Q|^{\tau}$ for any $Q\in\mathcal{D}$,
the space $\dot{a}_{p,q}^{s,\upsilon}(W)$
is precisely $\dot{a}_{p,q}^{s,\tau}(W)$ studied in \cite{bhyy1,bhyy2,bhyy3}.
In \cite[Theorem 13.1]{bhyy2}, for any    $W\in\mathcal{A}_{p}$, Bu et al. proved that
any $(D, E, F)$-almost diagonal operator
is bounded on $\dot{a}_{p,q}^{s,\tau}(W)$ if
\begin{align*}
D>J_{\dot{a}_{p,q}^{s,\tau}}+
\left[\left(n\widehat{\tau}\right)\wedge\frac{d}{p}\right],\
E>\frac{n}{2}+s+n\widehat{\tau},
\text{\ and\ }
F>J_{\dot{a}_{p,q}^{s,\tau}}
-\frac{n}{2}-s-n\left(\tau-\frac{1}{p}\right)_+,
\end{align*}
where
\begin{align*}
\widehat{\tau}:=\left[\tau-\frac{1}{p}
+\frac{d}{n p}\right]_{+},	
\end{align*}
$J_{\dot{a}_{p,q}^{s,\tau}}$ is as in \eqref{J_nu},
and $d\in[0, n)$ is an $\mathcal{A}_p$-dimension of $W$
(see \cite[Definition 2.22]{bhyy1} for the details).

By Example \ref{examp}(i), we find that $\upsilon\in\mathcal{G}(\tau,\tau;0)$.
In this case,
Theorem \ref{a(W)adopebound} gives that
any $(D, E, F)$-almost diagonal operator
is bounded on $\dot{a}_{p,q}^{s,\tau}(W)$ if
\begin{align*}
D>J_{\dot{a}_{p,q}^{s,\tau}}+
\left[n\Delta\wedge\frac{d^{{\rm lower}}_{p,\infty}(W)}{p}\right]
+\frac{d^{{\rm upper}}_{p, \infty}(W)}{p},\
E>\frac{n}{2}+s+n\Delta,
\end{align*}
and
\begin{align*}
F>J_{\dot{a}_{p,q}^{s,\tau}}
-\frac{n}{2}-s-n\left(\tau-\frac{1}{p}\right)_+
+\frac{d^{{\rm upper}}_{p, \infty}(W)}{p},
\end{align*}
where
$$
\Delta:=\left[\tau-\frac1{p}+
\frac{d^{{\rm lower}}_{p, \infty}(W)}{np}\right]_{+}
$$
and $J_{\dot{a}_{p,q}^{s,\tau}}, d^{{\rm lower}}_{p,\infty}(W)$,
and $d^{{\rm upper}}_{p,\infty}(W)$ are as, respectively, in
\eqref{J_nu}, \eqref{eq-low-dim}, and \eqref{eq-upp-dim}.

%
%
Comparing $\Delta$ in Theorem \ref{a(W)adopebound}
and $\widehat{\tau}$ in \cite[Theorem 13.1]{bhyy2},
we find that the $\mathcal{A}_p$-dimension $d$ of $W$
in $\widehat{\tau}$ is replaced by $d_{p, \infty}^{\text{lower}}(W)$
in $\Delta$. To understand the relation between $d$
and $d_{p, \infty}^{\text {lower}}(W)$, let
$$
d_p(W):=\inf \left\{d\in[0, n): W \text { has } \mathcal{A}_p
\text{-dimension}\ d\right\}.
$$
From Definition \ref{def-upp-low-dimension}, \eqref{eq-low-dim},
the definition of $\mathcal{A}_{p}$-dimensions,
and Jensen's inequality, it follows that,
for any $d\in[0, n)$ such that $W$ has
$\mathcal{A}_p$-dimension $d$, $W$ has
$\mathcal{A}_{p,\infty}$-lower dimension $d$ and hence
$d_{p, \infty}^{\mathrm{lower}}(W) \leq d_p(W)$.
Consequently, the replacement of $d$ by
$d_{p, \infty}^{\text {lower}}(W)$ in $\Delta$
at least does not negatively affect
the boundedness of almost diagonal operators on $\dot{a}_{p,q}^{s,\tau}(W)$
in Theorem \ref{a(W)adopebound}.

Notice that, compared to \cite[Theorem 13.1]{bhyy2},
the ranges of $D$ and $F$ in Theorem \ref{a(W)adopebound}
contain an additional term $\frac{d_{p, \infty}^{\mathrm{upper}}(W)}{p}$.
Using \cite[Proposition 2.27]{bhyy1},
we conclude that, for any $p\in(0, \infty)$ and $W\in\mathcal{A}_p$,
\begin{align}\label{eq-d-upp}
\frac{d_{p, \infty}^{\mathrm{upper}}(W)}{p}\leq
\begin{cases}
\displaystyle\frac{d_{p^{\prime}}(W^{-\frac{p'}{p}})}{p^{\prime}}
&\displaystyle\text{if }p\in(1, \infty),\\
0&\text{if }p\in(0, 1],
\end{cases}
\end{align}
where, for any $p\in(1, \infty)$,
$W^{-\frac{p'}{p}}\in\mathcal{A}_{p'}$
is the dual weight of $W$ and
$d_{p^{\prime}}(W^{-\frac{p'}{p}})\in[0,n)$ is an
$\mathcal{A}_{p'}$-dimension of $W^{-\frac{p'}{p}}$.
By \eqref{eq-d-upp}, we find that,
for any $p\in(0,1]$ and $W\in\mathcal{A}_{p}$,
$d_{p, \infty}^{\mathrm{upper}}(W)=0$,
which further implies that, in the case where $p\in(0,1]$,
Theorem \ref{a(W)adopebound} coincides with \cite[Theorem 13.1]{bhyy2}.
However, when $p\in(1, \infty)$,
the term $d_{p, \infty}^{\mathrm{upper}}(W)$ is probably
positive and hence the ranges of $D$ and $F$
in Theorem \ref{a(W)adopebound} may not
coincide with those in \cite[Theorem 13.1]{bhyy2}.

The reason yielding the above phenomenon is that,
compared to \cite[Theorem 13.1]{bhyy2}, Theorem
\ref{a(W)adopebound} is proved to be valid
for the larger matrix weight class $\mathcal{A}_{p,
\infty}\supsetneqq\mathcal{A}_p$ .
At the technical level, the origin of
the additional term $\frac{d_{p, \infty}^{\mathrm{upper}}(W)}{p}$
can be traced back to \cite[Theorem 5.6]{byyz24}
and more essentially to \cite[Proposition 6.6]{bhyy4}
which establishes the estimate for reducing operators generated
by matrix $\mathcal{A}_{p,\infty}$ weights.
From this perspective, when considering
the boundedness of almost diagonal operators
on $\dot{a}_{p,q}^{s,\tau}(W)$ for any
$p\in(1, \infty)$ and $W\in\mathcal{A}_p$,
\cite[Theorem 13.1]{bhyy2} gives better results
than Theorem \ref{a(W)adopebound}.
The same deficiency appears in the boundedness of
pseudo-differential, trace, and Calder\'{o}n--Zygmund operators
on $\dot{A}_{p, q}^{s, \upsilon}(W)$ when compared to
the corresponding results in \cite{bhyy3}.
This is because, to establish the above boundedness,
we use the molecular and the wavelet characterizations of
$\dot{A}_{p, q}^{s, \upsilon}(W)$ obtained in \cite{byyz24}
which strongly depend on Theorem \ref{a(W)adopebound}.
\end{remark}

\subsection{Pseudo-Differential Operators}\label{s-2-1}

We begin with the concept of pseudo-differential
operators with homogeneous symbols (see, for example, \cite[(1)]{gt99}).
\begin{definition}
Let $\eta\in\mathbb{Z}_{+}$ and $\dot{S}_{1,1}^\eta$
be the set of all infinitely differentiable functions
$\theta$ on $\mathbb{R}^n\times(\mathbb{R}^n \backslash\{\mathbf{0}\})$
such that, for any $\alpha,\beta\in\mathbb{Z}_{+}^n$,
\begin{align*}
\sup _{x \in \mathbb{R}^n,\,\xi\in\mathbb{R}^n
\backslash\{\mathbf{0}\}}|\xi|^{-\eta-|\alpha|+|\beta|}
\left|\partial_{\xi}^\beta \partial_x^\alpha
\theta(x, \xi)\right|<\infty.
\end{align*}
For any $\theta\in\dot{S}_{1,1}^\eta$,
the \emph{pseudo-differential operator}
$T_{\theta}$ with symbol
$\theta$ is defined by setting, for any
$f\in\mathcal{S}_{\infty}$
and $x\in\mathbb{R}^n$,
$$
(T_{\theta}f)(x):=\int_{\mathbb{R}^{n}} \theta(x, \xi)
\widehat{f}(\xi) e^{i x\cdot\xi}d\,\xi.
$$
\end{definition}

The next lemma is exactly \cite[Lemma 2.1]{gt99}.

\begin{lemma}\label{lem-T-pd-pre}
Let $\eta\in\mathbb{Z}_{+}$, $\theta\in\dot{S}_{1,1}^\eta$,
and  $T_{\theta}$ be the pseudo-differential operator
with symbol $\theta$. Then $T_{\theta}$ is a continuous linear map from
$\mathcal{S}_{\infty}$ to $\mathcal{S}$.
In particular, the \emph{formal adjoint}
$T_{\theta}^*$ of $T_{\theta}$ defined by setting,
for any $f\in\mathcal{S}'$
and $\phi\in\mathcal{S}_{\infty}$,
$$
\left\langle T_{\theta}^*f, \phi\right\rangle
:=\langle f, T_{\theta}\phi\rangle
$$
is  a continuous linear map from $\mathcal{S}'$
to $\mathcal{S}'_{\infty}$.
\end{lemma}

The following theorem is the main result of this subsection
in which we establish the boundedness of
pseudo-differential operators on $\dot{A}_{p, q}^{s, \upsilon}(W)$.

\begin{theorem}\label{thm-T-pd}
Let $(A, a)\in\{(B, b), (F, f)\}$, $s\in\mathbb{R}$,
$p\in(0,\infty)$, $q\in(0, \infty]$, and $W\in \mathcal{A}_{p,\infty}$.
Suppose that $\delta_1,\delta_2$, and $\omega$ satisfy \eqref{eq-delta1<0},
$\upsilon\in\mathcal{G}(\delta_1, \delta_2; \omega)$, and
$\varphi, \psi\in\mathcal{S}$ satisfy \eqref{cond1} and \eqref{cond3}.
Assume that $\eta\in\mathbb{Z}_{+}$, $\theta\in\dot{S}_{1,1}^\eta$,
and $T_{\theta}$ is the pseudo-differential operator
with symbol $\theta$. Then the following two statements hold.
\begin{itemize}
\item[{\rm (i)}] If $F_{\dot{a}_{p,q}^{s,\upsilon}(W)}<\frac{n}{2}$,
then, for any $\vec{f}\in\dot{A}_{p, q}^{s+\eta, \upsilon}(W)$,
\begin{align}\label{eq-T-pd}
\widetilde{T}_{\theta}\left(\vec{f}\right):=\sum_{Q\in\mathcal{D}}
\left\langle \vec{f},\varphi_Q\right\rangle T_{\theta}\left(\psi_Q\right)
\end{align}
converges in $(\mathcal{S}'_{\infty})^m$,
where $F_{\dot{a}_{p,q}^{s,\upsilon}(W)}$
is as in Theorem \ref{a(W)adopebound}.
Moreover, $\widetilde{T}_{\theta}$ is bounded from
$\dot{A}_{p, q}^{s+\eta, \upsilon}(W)$
to $\dot{A}_{p, q}^{s, \upsilon}(W)$ and
$\widetilde{T}_{\theta}=T_{\theta}$ on
$(\mathcal{S}_{\infty})^m$.
\item[{\rm (ii)}] If $F_{\dot{a}_{p,q}^{s,\upsilon}(W)}\geq\frac{n}{2}$
and, for any $\gamma\in\mathbb{Z}_{+}^n$ with
$|\gamma|\leq F_{\dot{a}_{p,q}^{s,\upsilon}(W)}-\frac{n}{2}$,
$T_{\theta}^*\left(x^\gamma\right)=0$,
then, for any $\vec{f}\in\dot{A}_{p, q}^{s+\eta, \upsilon}(W)$,
\eqref{eq-T-pd} also converges in
$(\mathcal{S}'_{\infty})^m$.
Furthermore, $\widetilde{T}_{\theta}$ is bounded from
$\dot{A}_{p, q}^{s+\eta, \upsilon}(W)$
to $\dot{A}_{p, q}^{s, \upsilon}(W)$ and
$\widetilde{T}_{\theta}=T_{\theta}$ on
$(\mathcal{S}_{\infty})^m$.
\end{itemize}
\end{theorem}

Before giving the proof of Theorem \ref{thm-T-pd},
we first present a remark to compare Theorem \ref{thm-T-pd}
with some known results.
\begin{remark}
Let all the symbols be the same as in Theorem \ref{thm-T-pd}.
Let $\tau\in[0,\infty)$ and, for any $Q\in\mathcal{D}$,
let $\upsilon(Q):=|Q|^{\tau}$. In this case,
the space $\dot{A}_{p, q}^{s, \upsilon}(W)$
is precisely the matrix-weighted Besov--Triebel--Lizorkin-type space
$\dot{A}_{p, q}^{s,\tau}(W)$ with $W$ being a
matrix $\mathcal{A}_{p, \infty}$
weight. On the one hand, Theorem \ref{thm-T-pd}
is a homogeneous variant of \cite[Theorem 5.13]{bhyy5}
in which Bu et al. established the boundedness of
pseudo-differential operators on
the inhomogeneous space $A_{p, q}^{s,\tau}(W)$;
on the other hand, Bu et al. \cite[Theorem 3.23]{bhyy3} also obtained
the boundedness of pseudo-differential operators
on $\dot{A}_{p, q}^{s, \tau}(W)$
with $W$ being a matrix $\mathcal{A}_{p}$ weight.
However, Theorem \ref{thm-T-pd} only coincides with
\cite[Theorem 3.23]{bhyy3} for any $p\in(0, 1]$ and
$W\in\mathcal{A}_p$. When $p\in(1, \infty)$ and
$W\in\mathcal{A}_p$, the index $F_{\dot{a}_{p,q}^{s,\tau}(W)}$
in \cite[Theorem 3.23]{bhyy3} may be better;
see Remark \ref{rmk-compare} for the reason.
In particular, even for  the case where $\dot{A}^{s,0}_{p,q}(W)
=\dot{A}^{s}_{p,q}(W)$ with $W$ being a
matrix $\mathcal{A}_{p, \infty}$ weight, Theorem \ref{thm-T-pd} is  new.

Moreover, when $m=1$ (the scalar-valued setting) and $W\equiv1$,
Theorem \ref{thm-T-pd} improves \cite[Theorem 1.5]{syy10}
in which Sawano et al. obtained the boundedness of
pseudo-differential operators on
$\dot{A}_{p, q}^{s, \tau}$ with $\tau$
having an upper bound. In particular, when $\tau=0$,
Theorem \ref{thm-T-pd} coincides with \cite[Theorems 1.1 and 1.2]{gt99}.
\end{remark}

To prove Theorem \ref{thm-T-pd}, we need the molecular
characterization of $\dot{A}_{p,q}^{s,\upsilon}(W)$. To present it,
we first recall some symbols. For any $r\in\mathbb{R}$, let
\begin{align}\label{eq-ceil}
\begin{cases}
\lceil\!\lceil r\rceil\!\rceil:=\min\{k\in\mathbb Z:\, k>r\},\
\lceil r\rceil:=\min\{k\in\mathbb Z:\, k\geq r\},\\
\lfloor\!\lfloor r\rfloor\!\rfloor:=\max\{k\in\mathbb Z:\, k< r\},\
\lfloor r\rfloor:=\max\{k\in\mathbb Z:\, k\leq r\}
\end{cases}
\end{align}
and
\begin{align}\label{eq-r**}
r^{**}:=r-\lfloor\!\lfloor r\rfloor\!\rfloor\in(0,1].
\end{align}
For any $M\in[0,\infty)$ and $x\in\mathbb{R}^n$, let
$u_M(x):=(1+|x|)^{-M}$. The following concept of smooth molecules
was introduced in \cite[Definition 3.4]{bhyy3},
which is a slight generalization of the traditional
one as in \cite[(3.7)-(3.10)]{fj90}.

\begin{definition}\label{KLMNmole}
Let $K,M\in[0,\infty)$, $L,N\in\mathbb{R}$, and
$Q\in\mathcal{D}$. A function $m_Q$ on $\mathbb{R}^n$
is called a \emph{(smooth) $(K,L,M,N)$-molecule supported near $Q$}
if, for any $x,y\in\mathbb{R}^n$,
\begin{itemize}
\item[{\rm (i)}] $|m_Q(x)|\leq(u_{K})_Q(x)$;
\item[{\rm (ii)}] $\int_{\mathbb R^n} m_Q(x)x^\gamma\,dx=0
\text{ if }\gamma\in\mathbb{Z}_+^n$ and $|\gamma|\leq L$;
\item[{\rm (iii)}] $|\partial^\gamma m_Q(x)|
\leq[\ell(Q)]^{-|\gamma|}(u_{M})_Q(x)\text{ if }
\gamma\in\mathbb{Z}_+^n$ and $|\gamma|<N$;
\item[{\rm (iv)}]
$$\left|\partial^\gamma m_Q(x)-\partial^\gamma m_Q(y)\right|
\leq\left[\ell(Q)\right]^{-|\gamma|}\left[\frac{|x-y|}{\ell(Q)}\right]^{N^{**}}
\sup_{|z|\leq|x-y|}(u_{M})_Q(x+z)$$
if $\gamma\in\mathbb{Z}_+^n$ and $|\gamma|=\lfloor\!\lfloor N\rfloor\!\rfloor$,
where $\lfloor\!\lfloor N\rfloor\!\rfloor$ and $N^{**}$
are as, respectively, in \eqref{eq-ceil} and \eqref{eq-r**}
and $(u_{M})_Q$ is as in \eqref{eq-phi_Q} with $\varphi$
replaced by $u_{M}$.
\end{itemize}
For brevity, we also call $m_Q$ a $(K, L, M, N)$-molecule.
\end{definition}

The following concepts were introduced in \cite[Definition 5.13]{byyz24}
when establishing the molecular characterization of
$\dot{A}_{p,q}^{s,\upsilon}(W)$.

\begin{definition}\label{def-anasynmole}
Let $(A, a)\in\{(B, b), (F, f)\}$, $s\in\mathbb{R}$,
$p\in(0,\infty)$, $q\in(0, \infty]$, and $W\in \mathcal{A}_{p,\infty}$.
Suppose that $\delta_1,\delta_2$, and $\omega$ satisfy \eqref{eq-delta1<0},
$\upsilon\in\mathcal{G}(\delta_1, \delta_2; \omega)$,
and $D_{\dot{a}_{p,q}^{s,\upsilon}(W)},
E_{\dot{a}_{p,q}^{s,\upsilon}(W)}$, and
$F_{\dot{a}_{p,q}^{s,\upsilon}(W)}$ are
as in Theorem \ref{a(W)adopebound}.
We call $m:=\{m_Q\}_{Q\in\mathcal{D}}$ a
\emph{family of analysis molecules}
for $\dot{A}_{p,q}^{s,\upsilon}(W)$ if there exist
\begin{align*}
K_m>D_{\dot{a}_{p,q}^{s,\upsilon}(W)}\vee
\left[E_{\dot{a}_{p,q}^{s,\upsilon}(W)}+\frac{n}{2}\right],
\ L_m \geq E_{\dot{a}_{p,q}^{s,\upsilon}(W)}-\frac{n}{2},
\ M_m>D_{\dot{a}_{p,q}^{s,\upsilon}(W)},
\text{ and } N_m>F_{\dot{a}_{p,q}^{s,\upsilon}(W)}-\frac{n}{2}	
\end{align*}
such that, for any $Q\in\mathcal{D}$,
$m_Q$ is a $(K_m, L_m, M_m, N_m)$-molecule.
We call $g:=\{g_Q\}_{Q\in\mathcal{D}}$ a
\emph{family of synthesis molecules}
for $\dot{A}_{p,q}^{s,\upsilon}(W)$ if there exist
\begin{align}\label{eq-syn-mole}
K_g>D_{\dot{a}_{p,q}^{s,\upsilon}(W)}\vee
\left[F_{\dot{a}_{p,q}^{s,\upsilon}(W)}+\frac{n}{2}\right],
\ L_g\geq F_{\dot{a}_{p,q}^{s,\upsilon}(W)}-\frac{n}{2},
\ M_g>D_{\dot{a}_{p,q}^{s,\upsilon}(W)},
\text{ and } N_g>E_{\dot{a}_{p,q}^{s,\upsilon}(W)}-\frac{n}{2}
\end{align}
such that, for any $Q\in\mathcal{D}$,
$g_Q$ is a $(K_g, L_g, M_g, N_g)$-molecule.
In particular, for any $Q\in\mathcal{D}$, $m_Q$ (resp. $g_Q$)
is called an \emph{analysis} (resp. a \emph{synthesis}) \emph{molecule}
for $\dot{A}_{p,q}^{s,\upsilon}(W)$.
\end{definition}

\begin{remark}\label{rmk-mole-phi}
In Definition \ref{def-anasynmole},
for any $\varphi\in\mathcal{S}_{\infty}$,
it is obvious that there exists a positive
constant $C$ such that $\{C\varphi_{Q}\}_{Q\in\mathcal{D}}$
is a family of both analysis and synthesis molecules
for $\dot{A}_{p,q}^{s,\upsilon}(W)$; we omit the details.
\end{remark}

In what follows, for any $f\in\mathcal{S}'_{\infty}$
and $\varphi\in\mathcal{S}_{\infty}$, let
$\langle f, \varphi\rangle:=f(\overline{\varphi})$,
where $f(\cdot)$ denotes the dual action.
The following result was established in \cite[Lemma 5.16]{byyz24},
which gives the reasonable action of elements
in $\dot{A}_{p,q}^{s,\upsilon}(W)$ on analysis molecules
for $\dot{A}_{p,q}^{s,\upsilon}(W)$.

\begin{lemma}
Let $A\in\{B, F\}$, $s\in\mathbb{R}$, $p\in(0, \infty)$,
$q\in(0,\infty]$, and $W\in \mathcal{A}_{p, \infty}$.
Assume that $\delta_1,\delta_2$, and $\omega$ satisfy \eqref{eq-delta1<0},
$\upsilon\in\mathcal{G}(\delta_1, \delta_2; \omega)$, and
$\varphi, \psi\in\mathcal{S}$ satisfy \eqref{cond1} and \eqref{cond3}.
If $\{m_Q\}_{Q\in\mathcal{D}}$ is a family of analysis molecules
for $\dot{A}_{p,q}^{s,\upsilon}(W)$, then, for
any $\vec{f}\in\dot{A}_{p,q}^{s,\upsilon}(W)$ and $Q\in\mathcal{D}$,
\begin{align}\label{eq-f-m}
\left\langle\vec{f}, m_Q\right\rangle_*
:=\sum_{R\in\mathcal{D}}\left\langle \psi_R,
m_Q\right\rangle\left\langle\vec{f},
\varphi_R\right\rangle	
\end{align}
converges absolutely and its value is independent of
the choice of $\varphi$ and $\psi$.
\end{lemma}

We now present the molecular characterization
of $\dot{A}_{p,q}^{s,\upsilon}(W)$  obtained
in \cite[Theorem 5.17]{byyz24}.

\begin{lemma}\label{lem-moledecomp}
Let $(A, a)\in\{(B, b), (F, f)\}$, $s\in\mathbb{R}$,
$p\in(0, \infty)$, $q\in(0,\infty]$,
and $W\in \mathcal{A}_{p,\infty}$. Suppose that $\delta_1,\delta_2$,
and $\omega$ satisfy \eqref{eq-delta1<0} and
$\upsilon\in\mathcal{G}(\delta_1, \delta_2; \omega)$.
Then the following statements hold.
\begin{itemize}
\item[{\rm (i)}]
If $\{m_Q\}_{Q\in\mathcal{D}}$ is a family of analysis molecules
for $\dot{A}^{s,\upsilon}_{p,q}(W)$,
then there exists a positive constant $C$ such that,
for any $\vec{f}\in\dot{A}^{s,\upsilon}_{p,q}(W)$,
$\|\{\langle\vec{f},m_Q\rangle_*\}_{Q\in\mathcal{D}}
\|_{\dot{a}^{s,\upsilon}_{p,q}(W)}
\leq C\|\vec{f}\|_{\dot{A}^{s,\upsilon}_{p,q}(W)}$,
where $\langle\cdot, \cdot\rangle_*$ is as in \eqref{eq-f-m}.
\item[{\rm (ii)}]
If $\{g_Q\}_{Q\in\mathcal{D}}$ is a family of synthesis molecules
for $\dot{A}^{s,\upsilon}_{p,q}(W)$,
then, for any $\vec{t}:=\{\vec{t}_Q\}_{Q\in\mathcal{D}}
\in\dot{a}^{s,\upsilon}_{p,q}(W)$,
$\vec{f}:=\sum_{Q\in\mathcal{D}}\vec{t}_Q
g_Q\in(\mathcal{S}'_{\infty})^m$
and there exists a positive constant $C$,
independent of $\vec{t}$, such that
$\|\vec{f}\|_{\dot{A}^{s,\upsilon}_{p,q}(W)}
\leq C\|\vec{t}\|_{\dot{a}^{s,\upsilon}_{p,q}(W)}$.
\end{itemize}
\end{lemma}

The following lemma guarantees that
pseudo-differential operators
map smooth functions to synthesis molecules
for $\dot{A}^{s,\upsilon}_{p,q}(W)$,
which is crucial for proving Theorem \ref{thm-T-pd}.

\begin{lemma}\label{lem-psi-mole}
Let $\eta, M, N\in\mathbb{Z}_+$, $\psi\in\mathcal{S}$
satisfy \eqref{cond1}, $\theta\in\dot{S}_{1,1}^\eta$,
and $T_{\theta}$ be the pseudo-differential operator with
symbol $\theta$. Then, for any $Q\in\mathcal{D}$,
$|Q|^{\frac{\eta}{n}}T_{\theta}(\psi_Q)$
is a constant multiple of an $(M,-1,M,N)$-molecule,
where the constant is independent of $Q$.	
\end{lemma}
\begin{proof}
Notice that there is no need for $(M,-1,M,N)$-molecules to satisfy
Definition \ref{KLMNmole}(ii). Thus,
to prove the present lemma, it suffices to verify the decay properties of
$|Q|^{\frac{\eta}{n}}T_{\theta}(\psi_Q)$ for any $Q\in\mathcal{D}$.
Let $\gamma\in\mathbb{Z}_{+}^n$ with $|\gamma|\leq N$.
Applying the argument used in \cite[pp.\,265--266]{gt99},
we obtain, for any $Q\in\mathcal{D}$ and $x\in\mathbb{R}^n$,
\begin{align*}
\left|\left[\partial^\gamma T_{\theta}\left(\psi_Q\right)\right](x)\right|
\lesssim|Q|^{-\frac{\eta}{n}-\frac{1}{2}-\frac{|\gamma|}{n}}
\left(1+[\ell(Q)]^{-1}\left|x-x_Q\right|\right)^{-M},
\end{align*}
which, together with Definition \ref{KLMNmole},
further implies that, for any $Q\in\mathcal{D}$,
$|Q|^{\frac{\eta}{n}}T(\psi_Q)$ is a constant multiple
of an $(M,-1,M,N)$-molecule, where the constant is independent of $Q$.	
This finishes the proof of Lemma \ref{lem-psi-mole}.
\end{proof}

To prove Theorem \ref{thm-T-pd}, we also need
the following Calder\'on reproducing formulae,
which is exactly \cite[Lemma 2.1]{yy10}.

\begin{lemma}\label{lem-Cdreproform}
Let $\varphi,\psi\in\mathcal{S}$ satisfy \eqref{cond3}
such that both $\operatorname{supp}\widehat{\varphi}$
and $\operatorname{supp}\widehat{\psi}$
are compact and bounded away from the origin.
Then, for any $f\in\mathcal{S}_\infty$,
\begin{align}\label{eq-Cdreproform}
f=\sum_{j\in\mathbb{Z}}2^{-jn}\sum_{k\in\mathbb{Z}^n}
\left(\widetilde{\varphi}_j*f\right)\left(2^{-j}k\right)
\psi_j\left(\cdot-2^{-j}k\right)
=\sum_{Q\in\mathcal{D}}\left\langle f,\varphi_Q\right\rangle\psi_Q
\end{align}
in $\mathcal{S}_\infty$,
where $\widetilde{\varphi}(x):=\overline{\varphi(-x)}$
for any $x\in\mathbb{R}^n$. Moreover,
for any $f\in\mathcal{S}_\infty'$,
\eqref{eq-Cdreproform} also holds in $\mathcal{S}_\infty'$.
\end{lemma}

Assume that $\varphi, \psi\in\mathcal{S}$
satisfy \eqref{cond1}.
The \emph{$\varphi$-transform} $S_{\varphi}$ is
defined by setting, for any $\vec{f}:=(f_1,\dots,f_m)^T
\in(\mathcal{S}'_{\infty})^m$, $S_{\varphi}\vec{f}
:=\{\langle \vec{f},\varphi_Q\rangle\}_{Q\in\mathcal{D}}
:=\{(\langle f_1, \varphi_Q\rangle,\dots,\langle f_m,
\varphi_Q\rangle)^T\}_{Q\in\mathcal{D}}$. Moreover,
the \emph{inverse $\varphi$-transform} $T_{\psi}$
is defined by setting, for any sequence $\vec{t}
:=\{\vec{t}_Q\}_{Q\in\mathcal{D}}$ in ${\mathbb{C}}^m$,
$T_{\psi}\vec{t}:=\sum_{Q\in\mathcal{D}}\vec{t}_Q\psi_{Q}$ if
this summation converges in $(\mathcal{S}'_{\infty})^m$.
We  recall the following $\varphi$-transform
characterization of $\dot{A}_{p,q}^{s, \upsilon}(W)$
established in \cite[Theorem 2.5(i)]{byyz24}.

\begin{lemma}\label{lem-phitransMWBTL}
Let $(A, a)\in\{(B, b), (F, f)\}$, $s\in\mathbb{R}$,
$p\in(0, \infty)$, $q\in(0,\infty]$,
and $W\in\mathcal{A}_{p,\infty}$. Suppose that
$\delta_1,\delta_2$, and $\omega$ satisfy \eqref{eq-delta1<0},
$\upsilon\in\mathcal{G}(\delta_1, \delta_2; \omega)$, and
$\varphi, \psi\in\mathcal{S}$ satisfy \eqref{cond1}.
Then the maps $S_{\varphi}:\,\dot{A}_{p,q}^{s,\upsilon}(W)
\to\dot{a}_{p,q}^{s,\upsilon}(W)$
and $T_{\psi}:\,\dot{a}_{p,q}^{s,\upsilon}(W)\to
\dot{A}_{p,q}^{s,\upsilon}(W)$ are bounded.
Moreover, if $\varphi, \psi$ further satisfy \eqref{cond3},
then $T_\psi \circ S_{\varphi}$ is the identity
on $\dot{A}_{p,q}^{s,\upsilon}(W)$.
\end{lemma}

Based on the above five lemmas,
we are able to prove Theorem \ref{thm-T-pd} as follows.

\begin{proof}[Proof of Theorem \ref{thm-T-pd}]
We first prove (i). Applying Lemma \ref{lem-phitransMWBTL} and
the definitions of $\|\cdot\|_{\dot{a}_{p, q}^{s, \upsilon}(W)}$
and $S_{\varphi}$,
we conclude that, for any $\vec{f}\in\dot{A}_{p, q}^{s+\eta, \upsilon}(W)$,
\begin{align}\label{eq-f-1}
\left\|\left\{\left\langle \vec{f},
\varphi_Q\right\rangle|Q|^{-\frac{\eta}{n}}
\right\}_{Q\in\mathcal{D}}
\right\|_{\dot{a}_{p, q}^{s, \upsilon}(W)}
&=\left\|\left\{\left\langle \vec{f},
\varphi_Q\right\rangle\right\}_{Q\in\mathcal{D}}
\right\|_{\dot{a}_{p, q}^{s+\eta, \upsilon}(W)}\\
&=\left\|S_{\varphi}\vec{f}
\right\|_{\dot{a}_{p, q}^{s+\eta, \upsilon}(W)}
\lesssim\left\|\vec{f}\right\|_{\dot{A}_{p, q}^{s+\eta, \upsilon}(W)}.\nonumber
\end{align}
Since $F_{\dot{a}_{p,q}^{s,\upsilon}(W)}<\frac{n}{2}$,
there is no need for synthesis molecules for $\dot{A}_{p,q}^{s,\upsilon}(W)$
to satisfy Definition \ref{KLMNmole}(ii).
By this and Lemma \ref{lem-psi-mole} with $M$
and $N$ being sufficiently large, we find that
there exists a positive constant $C$ such that
$\{C|Q|^{\frac{\eta}{n}}T_{\theta}(\psi_Q)\}_{Q\in\mathcal{D}}$ is a
family of synthesis molecules for $\dot{A}_{p,q}^{s,\upsilon}(W)$.
For any $\vec{f}\in\dot{A}_{p, q}^{s+\eta, \upsilon}(W)$,
from this, \eqref{eq-f-1}, and Lemma \ref{lem-moledecomp}(ii)
with $\{\vec{t}_Q\}_{Q\in\mathcal{D}}$ and
$\{g_Q\}_{Q\in\mathcal{D}}$ replaced, respectively,
by $\{\langle \vec{f},\varphi_Q\rangle|Q|^{-\frac{\eta}{n}}
\}_{Q\in\mathcal{D}}$ and
$\{|Q|^{\frac{\eta}{n}}T_{\theta}(\psi_Q)\}_{Q\in\mathcal{D}}$,
we infer that
\begin{align*}
\widetilde{T}_{\theta}\left(\vec{f}\right)=\sum_{Q\in\mathcal{D}}
\left\langle \vec{f},\varphi_Q\right\rangle|Q|^{-\frac{\eta}{n}}
|Q|^{\frac{\eta}{n}}T_{\theta}\left(\psi_Q\right)
\end{align*}
converges in $(\mathcal{S}'_{\infty})^m$ and
\begin{align*}
\left\|\widetilde{T}_{\theta}\left(\vec{f}\right)
\right\|_{\dot{A}_{p, q}^{s, \upsilon}(W)}
\lesssim\left\|\left\{\left\langle \vec{f},
\varphi_Q\right\rangle|Q|^{-\frac{\eta}{n}}
\right\}_{Q\in\mathcal{D}}\right\|
_{\dot{a}_{p, q}^{s, \upsilon}(W)}
\lesssim\left\|\vec{f}\right\|_{\dot{A}_{p, q}^{s+\eta, \upsilon}(W)},
\end{align*}
which further implies that $\widetilde{T}_{\theta}$ is bounded from
$\dot{A}_{p, q}^{s+\eta, \upsilon}(W)$
to $\dot{A}_{p, q}^{s, \upsilon}(W)$.
To verify $\widetilde{T}_{\theta}=T_{\theta}$ on
$(\mathcal{S}_{\infty})^m$,
using Lemmas \ref{lem-T-pd-pre} and \ref{lem-Cdreproform},
we conclude that, for any $\vec{f}
\in(\mathcal{S}_{\infty})^m$,
\begin{align*}
\widetilde{T}_{\theta}\left(\vec{f}\right)=
\sum_{Q\in\mathcal{D}}\left\langle\vec{f},
\varphi_Q\right\rangle T_{\theta}\left(\psi_Q\right)
\end{align*}
also converges in $(\mathcal{S})^m$
and equals $T_{\theta}(\vec{f})$.
This finishes the proof of (i).

Next, we show (ii). Notice that,
under the assumptions of Theorem \ref{thm-T-pd}(ii),
for any $Q\in\mathcal{D}$, $|Q|^{\frac{\eta}{n}}T_{\theta}(\psi_Q)$
satisfies Definition \ref{KLMNmole}(ii) for any $\gamma\in\mathbb{Z}_{+}^n$
with $|\gamma|\leq F_{\dot{a}_{p,q}^{s,\upsilon}(W)}-\frac{n}{2}$.
From this and Lemma \ref{lem-psi-mole} with $M$
and $N$ being sufficiently large, it follows that
there exists a positive constant $C$ such that
$\{C|Q|^{\frac{\eta}{n}}T_{\theta}(\psi_Q)\}_{Q\in\mathcal{D}}$
is a family of synthesis molecules for $\dot{A}_{p,q}^{s,\upsilon}(W)$.
Based on this, to prove (ii), it suffices to repeat the same argument
as that used in the proof of (i); we omit the details.
This finishes the proof of (ii) and hence Theorem \ref{thm-T-pd}.
\end{proof}

\subsection{Trace and Extension Operators}\label{s-2-2}

The main tool involved in establishing the boundedness
of the trace and the extension operators on
$\dot{A}_{p,q}^{s,\upsilon}(W)$ is
the wavelet characterization of $\dot{A}_{p,q}^{s,\upsilon}(W)$.
To give the wavelet characterization of $\dot{A}_{p,q}^{s,\upsilon}(W)$,
we first recall the concept of Daubechies wavelets
(see, for example, \cite{dau88} and
\cite[Sections 3.8 and 3.9]{mey92}).
In what follows, for any $k\in\mathbb{N}$,
let $C^{k}$ denote the set of all
$k$ times continuously differentiable functions on $\mathbb{R}^n$.

\begin{definition}\label{def-Dauwav}
Let $k\in\mathbb{N}$. A family of real-valued
functions $\{\theta^{(\lambda)}\}_{\lambda=1}^{2^n-1}$
in $C^{k}$ with bounded supports
are called the \emph{Daubechies wavelets of class}
$C^{k}$ if $\{\theta^{(\lambda)}_Q :\,\lambda\in\{1,\ldots,2^n-1\},
\ Q\in\mathcal{D}\}$ is an orthonormal basis of $L^2$.
\end{definition}

Let $k\in\mathbb{N}$ and $\{\theta^{(\lambda)}\}_{\lambda=1}^{2^n-1}$
be the Daubechies wavelets of class $C^{k}$.
From \cite[Corollary 5.5.2]{dau92}, it follows that,
for any $\lambda\in\{1,\ldots,2^n-1\}$ and
$\gamma\in\mathbb{Z}_+^n$ with $|\gamma|\leq k$,
\begin{align}\label{eq-int=0}
\int_{\mathbb{R}^n}\theta^{(\lambda)}(x)x^{\gamma}\,dx=0.
\end{align}

The next lemma gives a method to generate the Daubechies wavelets
as in Definition \ref{def-Dauwav} (see, for example, \cite{dau88}).

\begin{lemma}\label{def-Dauwavv2}
For any $k\in\mathbb{N}$, there exist two real-valued
functions $\varphi,\psi\in C^k(\mathbb{R})$
with bounded supports such that
$\{\theta^{(\lambda)}_Q :\ Q\in\mathcal{D},\
\lambda:=(\lambda_1,\dots,\lambda_n)\in\Lambda_n:=
\{0, 1\}^n\setminus\{\mathbf{0}\}\}$
is an orthonormal basis of $L^2$, where,
for any $\lambda:=(\lambda_1,\dots,\lambda_n)\in\Lambda_n$
and $x:=(x_1,\dots,x_n)\in\mathbb{R}^{n}$,
\begin{align}\label{eq-theta}
\theta^{(\lambda)}(x):=\prod_{i=1}^{n}\phi^{(\lambda_i)}(x_i)
\end{align}
with $\phi^{(0)}:=\varphi$ and $\phi^{(1)}:=\psi$.
In other words, $\{\theta^{(\lambda)}\}_{\lambda\in\Lambda_n}$
are the Daubechies wavelets of class $C^{k}$ as
in Definition \ref{def-Dauwav}.
\end{lemma}

\begin{remark}\label{rmk-Dauwavv2}
Let $\varphi$ be as in Lemma \ref{def-Dauwavv2}.
Using \cite[(4.1) and (4.2)]{dau88}, we conclude that
there exists a sequence $\{h_j\}_{j\in\mathbb{Z}}$ in $\mathbb{C}$
such that, for any $t\in\mathbb{R}$,
\begin{align*}
\varphi(t)=\sqrt{2}\sum_{j\in\mathbb{Z}}h_j\varphi(2t-j).
\end{align*}
This, together with the continuity of $\varphi$,
further implies that there exists
$k_0\in\mathbb{Z}$ such that $\varphi(-k_0)\neq0$.
\end{remark}

Based on Daubechies wavelets,
the following wavelet characterization
of $\dot{A}^{s,\upsilon}_{p,q}(W)$
was established in \cite[Theorem 5.20]{byyz24}.

\begin{lemma}\label{lem-Dauwav decomp}
Let $(A, a)\in\{(B, b), (F, f)\}$, $s\in\mathbb{R}$,
$p\in(0,\infty)$, $q\in(0, \infty]$, and $W\in \mathcal{A}_{p,\infty}$.
Assume that $\delta_1,\delta_2$, and $\omega$ satisfy \eqref{eq-delta1<0},
$\upsilon\in\mathcal{G}(\delta_1, \delta_2; \omega)$,
and $k\in\mathbb{N}$ satisfies
\begin{align}\label{eq-k}
k>\max\left\{E_{\dot{a}_{p,q}^{s,\upsilon}(W)}-\frac{n}{2},
F_{\dot{a}_{p,q}^{s,\upsilon}(W)}-\frac{n}{2}\right\},
\end{align}
where $E_{\dot{a}_{p,q}^{s,\upsilon}(W)}$
and $F_{\dot{a}_{p,q}^{s,\upsilon}(W)}$ are as in
Theorem \ref{a(W)adopebound}.
If $\{\theta^{(\lambda)}\}_{\lambda=1}^{2^n-1}$ are
the Daubechies wavelets of class $C^{k}$,
then, for any $\vec{f}\in\dot{A}^{s,\upsilon}_{p,q}(W)$,
\begin{align*}
\vec{f}=\sum_{\lambda=1}^{2^n-1}\sum_{Q\in\mathcal{D}}
\left\langle\vec{f},\theta^{(\lambda)}_Q\right\rangle_*\theta^{(\lambda)}_Q
\end{align*}
in $(\mathcal{S}'_{\infty})^m$ and
$$
\left\|\vec{f}\right\|_{\dot{A}^{s,\upsilon}_{p,q}(W)}
\sim\sum_{\lambda=1}^{2^n-1}\left\|\left\{\left\langle
\vec{f},\theta^{(\lambda)}_Q \right\rangle_*\right\}
_{Q\in\mathcal{D}}\right\|_{\dot{a}^{s,\upsilon}_{p,q}(W)},
$$
where the positive equivalence constants are independent
of $\vec{f}$ and $\langle\cdot,\cdot \rangle_*$ is as in \eqref{eq-f-m}.
\end{lemma}

Following the approach as in \cite{bhyy3, fra24, fr08},
we next use Daubechies wavelets to establish
the boundedness of the trace and the extension operators on
$\dot{A}_{p,q}^{s,\upsilon}(W)$. Suppose $f$ is a function
defined on $\mathbb{R}^{n+1}$. For any $x'\in\mathbb{R}^{n}$, let
\begin{align}\label{defTr}
\operatorname{Tr}f(x'):=f(x',0).
\end{align}
For any functions $g$ and $h$ defined, respectively,
on $\mathbb{R}^{n}$ and $\mathbb{R}^{m}$
and any $z:=(x,y)\in\mathbb{R}^{n+m}$, let
\begin{align}\label{eq-g-h}
(g\otimes h)(z):=g(x)h(y).
\end{align}
Let $n\in\mathbb{N}$. In this subsection,
let $\mathcal{D}(\mathbb{R}^n)$
be the set of all dyadic cubes in $\mathbb{R}^n$
and, for any $j\in\mathbb{Z}$, let
$\mathcal{D}_j(\mathbb{R}^n):=\{Q\in\mathcal{D}
(\mathbb{R}^n):\ \ell(Q)=2^{-j}\}$.
Let $k\in\mathbb{N}$ and $\varphi,\psi\in
C^k(\mathbb{R})$ be as in Lemma \ref{def-Dauwavv2}.
For any $\lambda'\in\Lambda_{n}$,
let $\theta^{(\lambda')}$ be as in \eqref{eq-theta}.
For any $Q'\in\mathcal{D}(\mathbb{R}^n)$ and $i\in\mathbb{Z}$,
let $P(Q',i):=Q'\times[i\ell(Q'),(i+1)\ell(Q'))$.
For any $\lambda'\in\Lambda_{n}$,
$Q'\in\mathcal{D}(\mathbb{R}^n)$,
and $x:=(x',x_{n+1})\in\mathbb{R}^{n+1}$, let
\begin{align}\label{defExt}
\operatorname{Ext}\theta^{(\lambda')}_{Q'}(x)
:=&\frac{[\ell(Q')]^{\frac{1}{2}}}{\varphi(-k_0)}
\left[\theta^{(\lambda')}\otimes\varphi\right]_{P(Q',k_0)}(x)
=\frac{[\ell(Q')]^{\frac12}}{\varphi(-k_0)}
\theta^{(\lambda',0)}_{P(Q',k_0)}(x)\\
=&\frac{1}{\varphi(-k_0)}\theta^{(\lambda')}_{Q'}(x')
\varphi\left(\frac{x_{n+1}}{\ell(Q')}-k_0\right),\notag
\end{align}
where $k_0$ is as in Remark \ref{rmk-Dauwavv2}.
From \eqref{defTr} and \eqref{defExt}, we infer that,
for any $\lambda'\in\Lambda_{n}$,
$Q'\in\mathcal{D}(\mathbb{R}^n)$, and $x'\in\mathbb{R}^n$,
\begin{align}\label{TrExt}
\operatorname{Tr}\circ\operatorname{Ext}
\left(\theta^{(\lambda')}_{Q'}\right)(x')
=\frac{[\ell(Q')]^{\frac12}}{\varphi(-k_0)}
\operatorname{Tr}\theta^{(\lambda',0)}_{P(Q',k_0)}(x')
=\theta^{(\lambda')}_{Q'}(x').
\end{align}

Let ${\upsilon}^{(n+1)}$ be a positive function
defined on $\mathcal{D}(\mathbb{R}^{n+1})$.
For any $Q'\in\mathcal{D}(\mathbb{R}^n)$, let
$${\upsilon}^{(n)}(Q'):={\upsilon}^{(n+1)}
\left(P\left(Q',0\right)\right).$$
We say that ${\upsilon}^{(n)}$ is the
\emph{restriction} of ${\upsilon}^{(n+1)}$
on $\mathcal{D}(\mathbb{R}^n)$.
Assume that $\delta_2\in[0, \infty)$, $\delta_1\in(-\infty,\delta_2]$,
$\omega\in[0,(n+1)(\delta_2-\delta_1)]$, and
$\upsilon^{(n+1)}$ is a $(\delta_1, \delta_2; \omega)$-order
growth function on $\mathcal{D}(\mathbb{R}^{n+1})$.
By the definition of growth functions, we
find that ${\upsilon}^{(n)}$, the restriction of
${\upsilon}^{(n+1)}$ on $\mathcal{D}(\mathbb{R}^n)$,
is a $(\frac{n+1}{n}\delta_1, \frac{n+1}{n}\delta_2; \omega)$-order
growth function on $\mathcal{D}(\mathbb{R}^n)$.

We now give the first main result of this subsection,
which establishes the boundedness of the trace operator
on generalized matrix-weighted Besov-type spaces.

\begin{theorem}\label{thm-Tr-B}
Let $p,\gamma\in(0,\infty)$, $q\in(0,\infty]$,
$W\in\mathcal{A}_{p,\infty}(\mathbb{R}^{n+1},\mathbb{C}^m)$,
and $V\in \mathcal{A}_{p,\infty}(\mathbb{R}^{n},\mathbb{C}^m)$.
Suppose that $\delta_2\in[0, \infty)$, $\delta_1\in[0, \delta_2]$,
$\omega\in[0,(n+1)(\delta_2-\delta_1)]$,
$\upsilon^{(n+1)}$ is a $(\delta_1, \delta_2; \omega)$-order
growth function on $\mathcal{D}(\mathbb{R}^{n+1})$,
and ${\upsilon}^{(n)}$ is the restriction of ${\upsilon}^{(n+1)}$
on $\mathcal{D}(\mathbb{R}^n)$. If $s\in(\frac{\gamma}{p}+
E+\frac{d_{p,\infty}^{\mathrm{upper}}(V)}{p}, \infty)$ with
\begin{align}\label{eq-E-B}
E:=\begin{cases}
\displaystyle n\left(\frac{1}{p}-\delta_1\right)
&\displaystyle\text{if}\,\,\delta_1>\frac{1}{p}\text{\ or\ if\ }
\delta_1=\frac{1}{p}\text{ and }q=\infty,\\
\displaystyle n\left(\frac{1}{p}-1\right)_+&\text{otherwise}
\end{cases}
\end{align}
and $d_{p,\infty}^{\mathrm{upper}}(V)$ being as in \eqref{eq-upp-dim},
then the trace operator $\operatorname{Tr}$
in \eqref{defTr} can be extended to a
bounded linear operator from $\dot{B}^{s,\upsilon^{(n+1)}}_{p,q}(W)$
to $\dot{B}^{s-\frac{\gamma}{p},\upsilon^{(n)}}_{p,q}(V)$
if and only if there exists a positive constant $C$ such that,
for any $Q'\in\mathcal{D}(\mathbb{R}^n)$ and $\vec z\in\mathbb{C}^m$,
\begin{equation}\label{eq-tr-B-B}
\int_{Q'}\left|V^{\frac{1}{p}}(x)\vec z\right|^p\,dx
\leq C 2^{j_{Q'}\gamma}\int_{P(Q',0)}\left|W^{\frac{1}{p}}(x)\vec z\right|^p\,dx.
\end{equation}
\end{theorem}

Before presenting the proof of Theorem \ref{thm-Tr-B},
we need first to give two lemmas. To this end,
we recall the concept of generalized averaging
sequence spaces introduced in \cite[Definition 3.3]{byyz24}.

\begin{definition}\label{averBTLseq}
Let $s\in\mathbb{R}$, $p, q\in(0,\infty]$,
and $\mathbb{A}:=\{A_Q\}_{Q\in\mathcal{D}(\mathbb{R}^n)}$
be a sequence of positive definite matrices.
Assume that $\delta_1,\delta_2$, and $\omega$ satisfy
\eqref{eq-delta1<0} and $\upsilon\in\mathcal{G}
(\delta_1, \delta_2; \omega)$. The
\emph{generalized averaging
Besov-type sequence space}
$\dot{b}_{p,q}^{s,\upsilon}(\mathbb{A})$ and,
when $p\in(0, \infty)$, the
\emph{generalized averaging
Triebel--Lizorkin-type sequence space}
$\dot{f}_{p,q}^{s,\upsilon}(\mathbb{A})$
are respectively defined to be the sets of
all sequences $\vec{t}:=\{\vec{t}_Q\}
_{Q\in\mathcal{D}(\mathbb{R}^n)}$ in
${\mathbb{C}}^m$ such that
\begin{align*}	
\left\|\vec{t}\right\|
_{\dot{a}_{p,q}^{s,\upsilon}(\mathbb{A})}	
:=\left\|\left\{2^{js}\left|\mathbb{A}_j
\vec{t}_j\right|\right\}
_{j\in\mathbb{Z}}\right\|_{L\dot{A}
_{p, q}^{\upsilon}}<\infty,
\end{align*}
where $(A, a)\in\{(B, b), (F, f)\}$,
$\|\cdot\|_{L\dot{A}_{p, q}^{\upsilon}}$
is as in \eqref{LA_nu}, and, for any $j\in\mathbb{Z}$,
\begin{align*}
\mathbb{A}_j:=\sum_{Q\in\mathcal{D}_j(\mathbb{R}^n)}A_Q\mathbf{1}_Q
\end{align*}
and $\vec{t}_j$ is as in \eqref{veclambda_j}.
\end{definition}

\begin{remark}\label{rmk-a(A)-a}
Let all the symbols be the same as in Definition \ref{averBTLseq}.
If $m=1$ (the scalar-valued setting) and
$\mathbb{A}:=\{1\}_{Q\in\mathcal{D}(\mathbb{R}^n)}$, we
denote the space $\dot{a}_{p,q}^{s,\upsilon}(\mathbb{A})$
by $\dot{a}_{p,q}^{s,\upsilon}(\mathbb{R}^n)$. Moreover,
it is easy to find that, for any sequence
$\vec{t}:=\{\vec{t}_Q\}_{Q\in\mathcal{D}(\mathbb{R}^n)}$
in ${\mathbb{C}}^m$, $\|\vec{t}\|
_{\dot{a}_{p,q}^{s,\upsilon}(\mathbb{A})}=
\|\{|A_Q\vec{t}_Q|\}_{Q\in\mathcal{D}(\mathbb{R}^n)}\|
_{\dot{a}_{p,q}^{s,\upsilon}(\mathbb{R}^n)}$.	
\end{remark}

The following lemma obtained in \cite[Corollary 3.15]{byyz24}
gives the equivalence between $\dot{a}_{p,q}^{s,\upsilon}(W)$
and $\dot{a}_{p,q}^{s,\upsilon}(\mathbb{A})$,
which is useful in the following proofs.

\begin{lemma}\label{a(A)=a(W)}
Let $a\in\{b, f\}$, $s\in\mathbb{R}$, $p\in(0, \infty)$, and $q\in(0,\infty]$.
Suppose that $\delta_1,\delta_2$, and $\omega$ satisfy \eqref{eq-delta1<0}
and $\upsilon\in\mathcal{G}(\delta_1, \delta_2; \omega)$.
If $W\in\mathcal{A}_{p,\infty}$ and $\mathbb{A}$
is a sequence of reducing operators of order $p$ for $W$,
then, for any sequence $\vec{t}:=\{\vec{t}_Q\}
_{Q\in\mathcal{D}(\mathbb{R}^n)}$ in ${\mathbb{C}}^m$,
\begin{align*}
\left\|\vec{t}\right\|_{\dot{a}_{p,q}^{s,\upsilon}(\mathbb{A})}
\sim\left\|\vec{t}\right\|_{\dot{a}_{p,q}^{s,\upsilon}(W)},
\end{align*}
where the positive equivalence
constants are independent of $\vec{t}$.
\end{lemma}

The next lemma follows from an observation of dyadic
cubes in $\mathbb{R}^n$; we omit the details.

\begin{lemma}\label{lem-P_R}
Let $i\in\mathbb{Z}$. For any $Q'\in\mathcal{D}(\mathbb{R}^n)$,
there exists $Q\in\mathcal{D}(\mathbb{R}^{n+1})$ such that,
for any $R\in\mathcal{D}(\mathbb{R}^n)$ with $R\subset Q'$,
$P(R, i)\subset Q$ and $\ell(Q')\leq\ell(Q)\lesssim\ell(Q')$,
where the implicit positive constant is
independent of $Q'$ but may depend on $i$.
\end{lemma}

Now, we give the proof of Theorem \ref{thm-Tr-B}.

\begin{proof}[Proof of Theorem \ref{thm-Tr-B}]
To prove the present theorem, let $k\in\mathbb{N}$
satisfy \eqref{eq-k} in Lemma \ref{lem-Dauwav decomp}
on both $\dot{B}^{s,\upsilon^{(n+1)}}_{p,q}(W)$
and $\dot{B}^{s-\frac{\gamma}{p},\upsilon^{(n)}}_{p,q}(V)$ and let
$\varphi,\psi\in C^k(\mathbb{R})$ be as in Lemma \ref{def-Dauwavv2}.
For any $\lambda'\in\Lambda_{n}$ and $\lambda\in\Lambda_{n+1}$,
let $\theta^{(\lambda')}$ and $\theta^{(\lambda)}$ be as in \eqref{eq-theta}.

We first show the necessity. To this end, let
$\lambda'_1:=(1,\dots, 1)\in\Lambda_{n}$
and $\lambda_1:=(\lambda'_1, 1)\in\Lambda_{n+1}$.
For any $Q'\in\mathcal{D}(\mathbb{R}^n)$ and $\vec{z}\in\mathbb{C}^m$,
we define the sequence $\vec{t}:=\{\vec{t}_Q\}_{Q\in\mathcal{D}
(\mathbb{R}^{n+1})}$ in ${\mathbb{C}}^m$ by setting,
for any $Q\in\mathcal{D}(\mathbb{R}^{n+1})$,
\begin{align}\label{eq-nece-t}
\vec{t}_Q:=\begin{cases}
\vec{z}
&\text{if }Q=P(Q', 0),\\
\mathbf{0}&\text{otherwise}.
\end{cases}
\end{align}
Observe that the assumption $\delta_1\in[0, \delta_2]$
guarantees that there exists a positive constant $C$ such that,
for any $Q,P\in\mathcal{D}(\mathbb{R}^{n+1})$ with $Q\subset P$,
$\upsilon^{(n+1)}(Q)\leq C\upsilon^{(n+1)}(P)$.
From this, \eqref{eq-nece-t}, the definition of
$\|\cdot\|_{\dot{b}^{s,\upsilon^{(n+1)}}_{p,q}(W)}$,
and the assumption that ${\upsilon}^{(n)}$
is the restriction of ${\upsilon}^{(n+1)}$, we infer that,
for any $Q'\in\mathcal{D}(\mathbb{R}^n)$
and $\vec z\in\mathbb{C}^m$,
\begin{align}\label{eq-t-z}
\left\|\vec{t}\right\|
_{\dot{b}^{s,\upsilon^{(n+1)}}_{p,q}(W)}
&=\sup_{P\in\mathcal{D}(\mathbb{R}^{n+1}),
P\supset P(Q', 0)}\frac{2^{j_{Q'}(s+\frac{n+1}{2})}}
{\upsilon^{(n+1)}(P)}\left[\int_{P(Q', 0)}\left|W^{\frac{1}{p}}(x)
\vec{z}\right|^p\, dx\right]^{\frac{1}{p}}\\
&\sim\frac{2^{j_{Q'}(s +\frac{n+1}{2})}}{\upsilon^{(n+1)}
(P(Q', 0))}\left[\int_{P(Q', 0)}
\left|W^{\frac{1}{p}}(x)\vec{z}\right|^p\, dx\right]^{\frac{1}{p}}\nonumber\\
&=\frac{2^{j_{Q'}(s +\frac{n+1}{2})}}{\upsilon^{(n)}(Q')}\left[\int_{P(Q', 0)}
\left|W^{\frac{1}{p}}(x)\vec{z}\right|^p\, dx\right]^{\frac{1}{p}}.\nonumber
\end{align}
If $\operatorname{Tr}$ is bounded from
$\dot{B}^{s,\upsilon^{(n+1)}}_{p,q}(W)$
to $\dot{B}^{s-\frac{\gamma}{p},\upsilon^{(n)}}_{p,q}(V)$,
then, by the definition of $\|\cdot\|_{\dot{b}^{s-\frac{\gamma}{p},
\upsilon^{(n)}}_{p,q}(V)}$, \eqref{eq-nece-t}, \eqref{defTr},
\eqref{eq-t-z}, and Lemma \ref{lem-Dauwav decomp},
we find that, for any $Q'\in\mathcal{D}(\mathbb{R}^n)$
and $\vec z\in\mathbb{C}^m$,
\begin{align*}
&\frac{2^{j_{Q'}(s-\frac{\gamma}{p}+
\frac{n+1}{2})}}{\upsilon^{(n)}(Q')}
\left[\int_{Q'}\left|V^{\frac{1}{p}}(x)
\vec{z}\right|^p\, dx\right]^{\frac{1}{p}}\\
&\quad\leq\left\|\left\{2^{\frac{j_{R}}{2}}
\vec{t}_{P(R,0)}\right\}_{R\in\mathcal{D}(\mathbb{R}^n)}\right\|
_{\dot{b}^{s-\frac{\gamma}{p},\upsilon^{(n)}}_{p,q}(V)}
\sim\left\|\left[\ell\left(Q'\right)\right]^{-\frac{1}{2}}
\theta^{(\lambda'_1)}_{Q'}\vec{z}\right\|
_{\dot{B}^{s-\frac{\gamma}{p},\upsilon^{(n)}}_{p,q}(V)}\\
&\quad=\left\|\operatorname{Tr}
\left(\theta_{P(Q',0)}^{(\lambda_1)}\right)\vec{z}
\right\|_{\dot{B}^{s-\frac{\gamma}{p},\upsilon^{(n)}}_{p,q}(V)}
\lesssim\left\|\theta_{P(Q',0)}^{(\lambda_1)}\vec{z}
\right\|_{\dot{B}^{s,\upsilon^{(n+1)}}
_{p,q}(W)}\sim\left\|\vec{t}\right\|
_{\dot{b}^{s,\upsilon^{(n+1)}}_{p,q}(W)}\\
&\quad\sim\frac{2^{j_{Q'}(s +\frac{n+1}{2})}}{\upsilon^{(n)}(Q')}
\left[\int_{P(Q', 0)}\left|W^{\frac{1}{p}}(x)\vec{z}
\right|^p\, dx\right]^{\frac{1}{p}},
\end{align*}
which, together with the assumption
$\upsilon^{(n)}(Q')\in(0,\infty)$ for any
$Q'\in\mathcal{D}(\mathbb{R}^n)$,
further implies that \eqref{eq-tr-B-B} holds and hence
completes the proof of the necessity.

Next, we prove the sufficiency.
From Lemma \ref{lem-Dauwav decomp}, it follows that,
for any $\vec{f}\in\dot{B}_{p, q}^{s, \upsilon^{(n+1)}}(W)$,
$$\vec{f}=\sum_{\lambda\in\Lambda_{n+1}}
\sum_{Q\in\mathcal{D}(\mathbb{R}^{n+1})}\left\langle\vec{f},
\theta_Q^{(\lambda)}\right\rangle_*\theta_Q^{(\lambda)}$$
in $[\mathcal{S}'_{\infty}(\mathbb{R}^{n+1})]^m$,
where $\langle\cdot,\cdot \rangle_*$ is as in \eqref{eq-f-m}.
Thus, for any $\vec{f}\in\dot{B}_{p, q}^{s, \upsilon^{(n+1)}}(W)$,
we define the action of $\operatorname{Tr}$ on $\vec{f}$ by setting
\begin{align}\label{eq-def-Trf}
\operatorname{Tr}\vec{f}:&=\sum_{\lambda\in \Lambda_{n+1}}
\sum_{Q\in\mathcal{D}(\mathbb{R}^{n+1})}\left\langle\vec{f},
\theta_Q^{(\lambda)}\right\rangle_*\operatorname{Tr}
\theta_Q^{(\lambda)}.
\end{align}
To complete the proof of the sufficiency,
it suffices to show that, for any
$\vec{f}\in\dot{B}_{p, q}^{s, \upsilon^{(n+1)}}(W)$,
$\operatorname{Tr}\vec{f}$ is well-defined in
$[\mathcal{S}'_{\infty}(\mathbb{R}^{n})]^m$
and $\operatorname{Tr}$ is bounded from
$\dot{B}_{p, q}^{s, \upsilon^{(n+1)}}(W)$
to $\dot{B}_{p, q}^{s-\frac{\gamma}{p},\upsilon^{(n)}}(V)$.
For any given $\vec{f}\in\dot{B}_{p, q}^{s, \upsilon^{(n+1)}}(W)$
and for any $\lambda\in\Lambda_{n+1}$, let
\begin{align}\label{eq-t-lambda}
\vec{t}^{(\lambda)}:=\left\{\vec{t}_Q^{(\lambda)}
\right\}_{Q\in\mathcal{D}(\mathbb{R}^{n+1})}
:=\left\{[\ell(Q)]^{-\frac{1}{2}}
\left\langle\vec{f}, \theta_Q^{(\lambda)}\right\rangle_*
\right\}_{Q\in\mathcal{D}(\mathbb{R}^{n+1})}.
\end{align}
Since, for any $\lambda\in \Lambda_{n+1}$,
the support of $\theta^{(\lambda)}$ is bounded,
then there exists $M\in\mathbb{N}$ such that,
for any $\lambda\in \Lambda_{n+1}$,
$\operatorname{supp}\theta^{(\lambda)}
\subset B(\mathbf{0}, M)$. Applying this, \eqref{defTr},
and \eqref{eq-phi_Q}, we conclude that,
for any $\lambda\in \Lambda_{n+1}$,
$i\in\mathbb{Z}$ with $|i|>M$, $Q'\in\mathcal{D}(\mathbb{R}^n)$,
$Q:=P(Q', i)\in\mathcal{D}(\mathbb{R}^{n+1})$,
and $x'\in\mathbb{R}^n$,
$$\operatorname{Tr}\theta_Q^{(\lambda)}(x')
=\theta_Q^{(\lambda)}\left(x', 0\right)=
\left[\ell\left(Q'\right)\right]^{-\frac{n+1}{2}}
\theta^{(\lambda)}\left(\frac{x'-x_{Q'}}{\ell(Q')},-i\right)=0,$$
which, together with \eqref{eq-def-Trf}, \eqref{eq-t-lambda},
\eqref{defTr}, and the fact that $\mathcal{D}(\mathbb{R}^{n+1})=\{P(Q',i):
\,i\in\mathbb{Z},\,Q'\in\mathcal{D}(\mathbb{R}^n)\}$, further implies that, for any
$\vec{f}\in\dot{B}_{p, q}^{s, \upsilon^{(n+1)}}(W)$,
\begin{align}\label{eq-Trf}
\operatorname{Tr} \vec{f}
&=\sum_{\lambda\in \Lambda_{n+1}}
\sum_{i\in\mathbb{Z}}
\sum_{Q'\in\mathcal{D}(\mathbb{R}^n)}\left\langle\vec{f},
\theta_{P(Q',i)}^{(\lambda)}\right\rangle_*
\operatorname{Tr}\theta_{P(Q',i)}^{(\lambda)}\\
&=\sum_{\lambda\in \Lambda_{n+1}}\sum_{i=-M}^M
\sum_{Q'\in\mathcal{D}(\mathbb{R}^n)}
\vec{t}_{P(Q',i)}^{(\lambda)}
\left[\ell\left(Q'\right)\right]^{\frac{1}{2}}
\theta_{P(Q',i)}^{(\lambda)}\left(\cdot, 0\right).\nonumber
\end{align}

To prove $\operatorname{Tr} \vec{f}$ is well-defined
on $[\mathcal{S}'_{\infty}(\mathbb{R}^{n})]^m$,
we first show that, for any $\lambda\in\Lambda_{n+1}$
and $i\in\mathbb{Z}$ with $|i|\leq M$,
$\{\vec{t}_{P(Q', i)}^{(\lambda)}\}
_{Q'\in\mathcal{D}(\mathbb{R}^n)}\in
\dot{b}_{p, q}^{s-\frac{\gamma}{p}, \upsilon^{(n)}}(V)$.
For this purpose, let $\mathbb{A}(W):=\{A_{Q, W}\}_{Q\in\mathcal{D}(\mathbb{R}^{n+1})}$
and $\mathbb{A}(V):=\{A_{Q', V}\}_{Q'\in\mathcal{D}(\mathbb{R}^n)}$
be sequences of reducing operators of order $p$,
respectively, for $W$ and $V$.
Let $\beta_1 \in \llbracket d_{p, \infty}^{\text {lower}}(W),\infty)$
and $\beta_2\in \llbracket d_{p, \infty}^{\text {upper }}(W),\infty)$,
where $d_{p, \infty}^{\text {lower}}(W)$
and $d_{p, \infty}^{\text {upper }}(W)$ are as, respectively,
in \eqref{eq-low-dim} and \eqref{eq-upp-dim}.
By Definition \ref{def-red-ope}, \eqref{eq-tr-B-B},
and Lemma \ref{growEST}, we find that,
for any $\lambda\in\Lambda_{n+1}$,
$i\in\mathbb{Z}$ with $|i|\leq M$, and
$Q'\in\mathcal{D}(\mathbb{R}^n)$,
\begin{align}\label{eq-A_n-A_n+1}
\left|A_{Q', V}\vec{t}_{P(Q', i)}^{(\lambda)}\right|
&\sim2^{j_{Q'}\frac{n}{p}}\left[\int_{Q'}
\left|V^{\frac{1}{p}}(x)\vec{t}_{P(Q', i)}^{(\lambda)}\right|^p\,dx\right]^{\frac{1}{p}}\\
&\lesssim 2^{j_{Q'}(\frac{n}{p}+\frac{\gamma}{p})}
\left[\int_{P(Q',0)}\left|W^{\frac{1}{p}}(x)
\vec{t}_{P(Q', i)}^{(\lambda)}\right|^p\,dx\right]^{\frac{1}{p}}\nonumber\\
&\sim2^{j_{Q'}(\frac{\gamma}{p}-\frac{1}{p})}
\left|A_{P(Q', 0), W}\vec{t}_{P(Q', i)}^{(\lambda)}\right|\nonumber\\ &\lesssim(1+|i|)^{\beta_1+\beta_2}2^{j_{Q'}(\frac{\gamma}{p}-\frac{1}{p})}
\left|A_{P(Q', i), W}\vec{t}_{P(Q', i)}^{(\lambda)}\right|\nonumber\\
&\leq\left(1+M\right)^{\beta_1+\beta_2}2^{j_{Q'}(\frac{\gamma}{p}-\frac{1}{p})}
\left|A_{P(Q', i), W}\vec{t}_{P(Q', i)}^{(\lambda)}\right|\nonumber\\
&\sim2^{j_{Q'}(\frac{\gamma}{p}-\frac{1}{p})}\left|A_{P(Q', i), W}
\vec{t}_{P(Q', i)}^{(\lambda)}\right|\nonumber.
\end{align}
For any $i\in\mathbb{Z}$ with $|i|\leq M$ and
for any $R\in\mathcal{D}(\mathbb{R}^n)$,
let $R(i)\in\mathcal{D}(\mathbb{R}^{n+1})$ be
as in Lemma \ref{lem-P_R} with $Q'$ and $Q$
replaced, respectively, by $R$ and $R(i)$.
From the growth condition of $\upsilon^{(n+1)}$,
Lemma \ref{lem-P_R}, and the assumption that $\upsilon^{(n)}$
is the restriction of $\upsilon^{(n+1)}$,
we deduce that, for any $i\in\mathbb{Z}$ with $|i|\leq M$ and
for any $R\in\mathcal{D}(\mathbb{R}^n)$,
\begin{align}\label{eq-vRi}
\upsilon^{(n)}(R)=
\frac{\upsilon^{(n+1)}(P(R,0))}{\upsilon^{(n+1)}(P(R,i))}
\frac{\upsilon^{(n+1)}(P(R,i))}{\upsilon^{(n+1)}(R(i))}
\upsilon^{(n+1)}(R(i))\sim\upsilon^{(n+1)}(R(i)).
\end{align}
This, combined with the definitions of
$\|\cdot\|_{\dot{b}_{p, q}^{s-\frac{\gamma}{p},
\upsilon^{(n)}}(\mathbb{R}^n)}$ and
$\|\cdot\|_{\dot{b}_{p, q}^{s, \upsilon^{(n+1)}}
(\mathbb{A}(W))}$, Lemma \ref{lem-P_R}, and
\eqref{eq-t-lambda}, further implies that,
for any $\lambda\in\Lambda_{n+1}$
and $i\in\mathbb{Z}$ with $|i|\leq M$,
\begin{align}\label{eq-t-f-1}
&\left\|\left\{2^{j_{Q'}
(\frac{\gamma}{p}-\frac{1}{p})}
\left|A_{P(Q', i), W}\vec{t}
_{P(Q', i)}^{(\lambda)}\right|\right\}
_{Q'\in\mathcal{D}(\mathbb{R}^n)}\right\|
_{\dot{b}_{p, q}^{s-\frac{\gamma}{p},
\upsilon^{(n)}}(\mathbb{R}^n)}\\
&\quad=\sup_{R\in\mathcal{D}(\mathbb{R}^n)}
\frac{1}{\upsilon^{(n)}(R)}
\left\{\sum_{j=j_R}^{\infty}\left[\sum_{
\genfrac{}{}{0pt}{}{Q'\in\mathcal{D}_j(\mathbb{R}^n)}{Q'\subset R}}
2^{j(s+\frac{n}{2})p}\left|A_{P(Q', i), W}
\vec{t}_{P(Q', i)}^{(\lambda)}\right|^p
|Q'|^{\frac{n+1}{n}}\right]^{\frac{q}{p}}
\right\}^{\frac{1}{q}}\nonumber\\
&\quad\lesssim\sup_{R\in\mathcal{D}(\mathbb{R}^n)}
\frac{1}{\upsilon^{(n+1)}(R(i))}\left\{\sum_{j=j_{R(i)}}^{\infty}
\left[\sum_{\genfrac{}{}{0pt}{}{Q\in\mathcal{D}
_j(\mathbb{R}^{n+1})}{Q\subset R(i)}}
2^{j(s+\frac{n+1}{2}-\frac{1}{2})p}
\left|A_{Q, W}\vec{t}_{Q}^{(\lambda)}
\right|^p|Q|\right]^{\frac{q}{p}}\right\}^{\frac{1}{q}}\nonumber\\
&\quad\leq\sup_{P\in\mathcal{D}(\mathbb{R}^{n+1})}
\frac{1}{\upsilon^{(n+1)}(P)}\left\{\sum_{j=j_P}^{\infty}
\left[\sum_{\genfrac{}{}{0pt}{}{Q\in\mathcal{D}_j
(\mathbb{R}^{n+1})}{Q\subset P}}2^{j(s+\frac{n+1}{2}-\frac{1}{2})p}
\left|A_{Q, W}\vec{t}_{Q}^{(\lambda)}
\right|^p|Q|\right]^{\frac{q}{p}}\right\}^{\frac{1}{q}}\nonumber\\
&\quad=\left\|\left\{[\ell(Q)]^{\frac{1}{2}}\vec{t}_Q^{(\lambda)}
\right\}_{Q\in\mathcal{D}(\mathbb{R}^{n+1})}\right\|
_{\dot{b}_{p, q}^{s, \upsilon^{(n+1)}}
(\mathbb{A}(W))}=\left\|\left\{\left\langle\vec{f},
\theta_Q^{(\lambda)}\right\rangle_*\right\}
_{Q\in\mathcal{D}(\mathbb{R}^{n+1})}\right\|
_{\dot{b}_{p, q}^{s, \upsilon^{(n+1)}}
(\mathbb{A}(W))}.\notag
\end{align}
Using \eqref{eq-A_n-A_n+1}, \eqref{eq-t-f-1},
Remark \ref{rmk-a(A)-a}, and Lemmas \ref{a(A)=a(W)} and
\ref{lem-Dauwav decomp}, we conclude that, for any $\lambda\in\Lambda_{n+1}$
and $i\in\mathbb{Z}$ with $|i|\leq M$,

\begin{align}\label{eq-t-f}
\left\|\left\{\vec{t}_{P(Q', i)}^{(\lambda)}\right\}
_{Q'\in \mathcal{D}(\mathbb{R}^n)}\right\|
_{\dot{b}_{p, q}^{s-\frac{\gamma}{p}, \upsilon^{(n)}}(V)}
&\sim\left\|\left\{\vec{t}_{P(Q', i)}^{(\lambda)}
\right\}_{Q'\in \mathcal{D}(\mathbb{R}^n)}
\right\|_{\dot{b}_{p, q}^{s-\frac{\gamma}{p},
\upsilon^{(n)}}(\mathbb{A}(V))}\\
&=\left\|\left\{\left|A_{Q', V}\vec{t}
_{P(Q', i)}^{(\lambda)}\right|
\right\}_{Q'\in \mathcal{D}(\mathbb{R}^n)}\right\|
_{\dot{b}_{p, q}^{s-\frac{\gamma}{p},
\upsilon^{(n)}}(\mathbb{R}^n)}\nonumber\\
&\lesssim\left\|\left\{2^{j_{Q'}
(\frac{\gamma}{p}-\frac{1}{p})}
\left|A_{P(Q', i), W}\vec{t}
_{P(Q', i)}^{(\lambda)}\right|\right\}
_{Q'\in\mathcal{D}(\mathbb{R}^n)}\right\|
_{\dot{b}_{p, q}^{s-\frac{\gamma}{p},
\upsilon^{(n)}}(\mathbb{R}^n)}\notag\\
&\lesssim\left\|\left\{\left\langle\vec{f},
\theta_Q^{(\lambda)}\right\rangle_*\right\}
_{Q\in\mathcal{D}(\mathbb{R}^{n+1})}\right\|
_{\dot{b}_{p, q}^{s, \upsilon^{(n+1)}}
(\mathbb{A}(W))}\notag\\
&\sim\left\|\left\{\left\langle\vec{f},
\theta_Q^{(\lambda)}\right\rangle_*\right\}
_{Q\in\mathcal{D}(\mathbb{R}^{n+1})}\right\|
_{\dot{b}_{p, q}^{s, \upsilon^{(n+1)}}(W)}
\lesssim\left\|\vec{f}\right\|
_{\dot{B}^{s,\upsilon^{(n+1)}}_{p,q}(W)}\notag
\end{align}
and hence
$\{\vec{t}_{P(Q', i)}^{(\lambda)}\}_{Q'\in\mathcal{D}(\mathbb{R}^n)}
\in\dot{b}_{p, q}^{s-\frac{\gamma}{p}, \upsilon^{(n)}}(V)$.

Notice that, under the assumptions of the present theorem,
particularly $s\in(\frac{\gamma}{p}+E
+\frac{d_{p,\infty}^{\mathrm{upper}}(V)}{p},\infty)$,
there is no need for synthesis molecules for
$\dot{B}_{p, q}^{s-\frac{\gamma}{p},
\upsilon^{(n)}}(V)$ to satisfy Definition \ref{KLMNmole}(ii).
Moreover, for any $\lambda\in \Lambda_{n+1}$,
$\theta^{(\lambda)}\left(\cdot, 0\right)$ has compact support.
By these observations, Definition \ref{KLMNmole},
and the choice of $k$ at the beginning of the present proof,
it is easy to verify that, for any $\lambda\in \Lambda_{n+1}$
and $i\in\mathbb{Z}$ with $|i|\leq M$, there
exists a positive constant $C$ such that
$\{C[\ell(Q')]^{\frac{1}{2}}\theta_{P(Q',i)}^{(\lambda)}
(\cdot, 0)\}_{Q'\in\mathcal{D}(\mathbb{R}^n)}$
is a family of synthesis molecules for
$\dot{B}_{p, q}^{s-\frac{\gamma}{p}, \upsilon^{(n)}}(V)$.
From this, \eqref{eq-t-f}, and Lemma \ref{lem-moledecomp}(ii),
it follows that, for any $\lambda\in\Lambda_{n+1}$
and $i\in\mathbb{Z}$ with $|i|\leq M$,
$$
\sum_{Q'\in\mathcal{D}(\mathbb{R}^n)}\vec{t}
_{P(Q',i)}^{(\lambda)}[\ell(Q')]^{\frac{1}{2}}
\theta_{P(Q',i)}^{(\lambda)}\left(\cdot, 0\right)
$$
converges in $[\mathcal{S}'_{\infty}(\mathbb{R}^{n})]^m$ and
\begin{align*}
\left\|\sum_{Q'\in\mathcal{D}(\mathbb{R}^n)}\vec{t}
_{P(Q',i)}^{(\lambda)}[\ell(Q')]^{\frac{1}{2}}
\theta_{P(Q',i)}^{(\lambda)}\left(\cdot, 0\right)
\right\|_{\dot{B}^{s-\frac{\gamma}{p},\upsilon^{(n)}}_{p,q}(V)}
&\lesssim\left\|\left\{\vec{t}_{P(Q', i)}^{(\lambda)}\right\}
_{Q'\in\mathcal{D}(\mathbb{R}^n)}\right\|_{\dot{b}_{p, q}^{s-\frac{\gamma}{p},
\upsilon^{(n)}}(V)}\\
&\lesssim\left\|\left\{\left\langle\vec{f},
\theta_Q^{(\lambda)}\right\rangle_*\right\}_{Q\in\mathcal{D}(\mathbb{R}^{n+1})}
\right\|_{\dot{b}_{p, q}^{s, \upsilon^{(n+1)}}(W)}.	
\end{align*}
Moreover, applying this, \eqref{eq-Trf},
the quasi-triangle inequality of
$\|\cdot\|_{\dot{B}^{s-\frac{\gamma}{p},\upsilon^{(n)}}_{p,q}(V)}$,
and Lemma \ref{lem-Dauwav decomp}, we obtain,
for any $\vec{f}\in\dot{B}_{p, q}^{s, \upsilon^{(n+1)}}(W)$,
$\operatorname{Tr}\vec{f}$ is well-defined in
$[\mathcal{S}'_{\infty}(\mathbb{R}^{n})]^m$ and
\begin{align}\label{eq-Trf-f-B}
\left\|\operatorname{Tr}\vec{f}\right\|
_{\dot{B}^{s-\frac{\gamma}{p},\upsilon^{(n)}}_{p,q}(V)}
&\lesssim\sum_{\lambda\in \Lambda_{n+1}}
\sum_{i=-M}^M\left\|\sum_{Q'\in\mathcal{D}(\mathbb{R}^n)}
\vec{t}_{P(Q',i)}^{(\lambda)}[\ell(Q')]^{\frac{1}{2}} \theta_{P(Q',i)}^{(\lambda)}\left(\cdot, 0\right)\right\|_{\dot{B}^{s-\frac{\gamma}{p},
\upsilon^{(n)}}_{p,q}(V)}\\
&\lesssim\sum_{\lambda\in \Lambda_{n+1}}
\left\|\left\{\left\langle\vec{f}, \theta_Q^{(\lambda)}
\right\rangle_*\right\}_{Q\in\mathcal{D}(\mathbb{R}^{n+1})}
\right\|_{\dot{b}^{s,\upsilon^{(n+1)}}_{p,q}(W)}
\sim\left\|\vec{f}\right\|_{\dot{B}^{s,\upsilon^{(n+1)}}_{p,q}(W)}.	\nonumber
\end{align}
Using \eqref{eq-Trf-f-B}, we also conclude that
$\operatorname{Tr}$ is bounded from
$\dot{B}_{p, q}^{s, \upsilon^{(n+1)}}(W)$
to $\dot{B}_{p, q}^{s-\frac{\gamma}{p},\upsilon^{(n)}}(V)$.
This finishes the proof of the sufficiency
and hence Theorem \ref{thm-Tr-B}.
\end{proof}

\begin{remark}\label{rmk-Tr-B}
By checking the proof of Theorem \ref{thm-Tr-B},
we find that the assumption $\delta_1\in[0, \delta_2]$
is only used in the proof of necessity to
estimate $\|\vec{t}\|_{\dot{a}_{p,q}^{s, \upsilon}(W)}$
for any $\vec{t}$ as in \eqref{eq-nece-t}.
Moreover, one can also verify that
the proof of the sufficiency of Theorem \ref{thm-Tr-B}
is also valid for the case where $\delta_1\in(-\infty, \delta_2]$;
we omit the details.
\end{remark}

Next, we establish the boundedness of the extension operator
on generalized matrix-weighted Besov-type spaces.

\begin{theorem}\label{thm-Ext-B}
Let $s\in\mathbb{R}$, $p, \gamma\in(0,\infty)$, $q\in(0,\infty]$,
$W\in\mathcal{A}_{p,\infty}(\mathbb{R}^{n+1},\mathbb{C}^m)$,
and $V\in \mathcal{A}_{p,\infty}(\mathbb{R}^{n},\mathbb{C}^m)$.
Assume that $\delta_2\in[0, \infty)$, $\delta_1\in[0, \delta_2]$,
$\omega\in[0,(n+1)(\delta_2-\delta_1)]$,
$\upsilon^{(n+1)}$ is a $(\delta_1, \delta_2; \omega)$-order
growth function on $\mathcal{D}(\mathbb{R}^{n+1})$,
and ${\upsilon}^{(n)}$ is the restriction of
$\upsilon^{(n+1)}$ on $\mathcal{D}(\mathbb{R}^n)$.
If there exists a positive constant $C$ such that,
for any $Q'\in\mathcal{D}(\mathbb{R}^n)$ and $\vec{z}\in\mathbb{C}^m$,
\begin{align}\label{eq-Ext-B-B}
2^{j_{Q'}\gamma}\int_{P(Q', 0)}\left|W^{\frac{1}{p}}(x)
\vec z\right|^p\,dx\leq C \int_{Q'}\left|V^{\frac{1}{p}}(x)
\vec z\right|^p\,dx,
\end{align}
then the extension operator $\operatorname{Ext}$
in \eqref{defExt} can be extended
to a bounded linear operator from
$\dot{B}^{s-\frac{\gamma}{p},\upsilon^{(n)}}_{p,q}(V)$
to $ \dot{B}^{s,\upsilon^{(n+1)}}_{p,q}(W)$.
Furthermore, if $s\in(\frac{\gamma}{p}+E+
\frac{d_{p,\infty}^{\mathrm{upper}}(V)}{p},\infty)$
and \eqref{eq-tr-B-B} holds, where $E$ and $d_{p,\infty}^{\mathrm{upper}}(V)$
are as, respectively, in \eqref{eq-E-B} and \eqref{eq-upp-dim},
then $\operatorname{Tr}\circ\operatorname{Ext}$
is the identity on $\dot{B}^{s-\frac{\gamma}{p},\upsilon^{(n)}}_{p,q}(V)$.
\end{theorem}
\begin{proof}
To prove the present theorem, let $k\in\mathbb{N}$ satisfy \eqref{eq-k}
in Lemma \ref{lem-Dauwav decomp} on both $\dot{B}^{s,\upsilon^{(n+1)}}_{p,q}(W)$
and $\dot{B}^{s-\frac{\gamma}{p},\upsilon^{(n)}}_{p,q}(V)$
and $\varphi,\psi\in C^k(\mathbb{R})$
be as in Lemma \ref{def-Dauwavv2}.
For any $\lambda'\in\Lambda_{n}$
and $\lambda\in\Lambda_{n+1}$, let $\theta^{(\lambda')}$
and $\theta^{(\lambda)}$ be as in \eqref{eq-theta}.
Using Lemma \ref{lem-Dauwav decomp}, we conclude that,
for any $\vec{f}\in\dot{B}^{s-\frac{\gamma}{p}, \upsilon^{(n)}}_{p,q}(V)$,
\begin{align*}
\vec{f}=\sum_{\lambda'\in\Lambda_{n}}
\sum_{Q'\in\mathcal{D}(\mathbb{R}^n)}
\left\langle\vec{f},\theta^{(\lambda')}_{Q'}
\right\rangle_*\theta^{(\lambda')}_{Q'}
\end{align*}
in $[\mathcal{S}_\infty'(\mathbb{R}^{n})]^m$,
where $\langle\cdot,\cdot \rangle_*$ is as in \eqref{eq-f-m}.
Thus, for any $\vec{f}\in\dot{B}^{s-\frac{\gamma}{p},
\upsilon^{(n)}}_{p,q}(V)$, we define $\operatorname{Ext}\vec{f}$
by setting
\begin{align}\label{eq-Ext-f}
\operatorname{Ext}\vec{f}
:=\sum_{\lambda'\in\Lambda_{n}}\sum_{Q'\in\mathcal{D}(\mathbb{R}^n)}
\left\langle\vec{f},\theta^{(\lambda')}_{Q'}\right\rangle_*
\operatorname{Ext}\theta^{(\lambda')}_{Q'}.
\end{align}

Next, we claim that, for any $\vec{f}\in
\dot{B}^{s-\frac{\gamma}{p},\upsilon^{(n)}}_{p,q}(V)$,
$\operatorname{Ext}\vec{f}$ is well-defined in
$[\mathcal{S}_\infty'(\mathbb{R}^{n+1})]^m$
and $\operatorname{Ext}$ is bounded from
$\dot{B}^{s-\frac{\gamma}{p},\upsilon^{(n)}}_{p,q}(V)$
to $\dot{B}^{s,\upsilon^{(n+1)}}_{p,q}(W)$.
Let $k_0$ be as in Remark \ref{rmk-Dauwavv2}.
For any given $\vec{f}\in\dot{B}^{s-\frac{\gamma}{p},
\upsilon^{(n)}}_{p,q}(V)$ and for any $\lambda'\in\Lambda_{n}$,
we define the sequence $\vec{t}^{(\lambda')}:=
\{\vec{t}^{(\lambda')}_Q\}_{Q\in\mathcal{D}(\mathbb{R}^{n+1})}$
by setting, for any $Q\in\mathcal{D}(\mathbb{R}^{n+1})$,
\begin{equation}\label{eq-t}
\vec{t}^{(\lambda')}_Q:=
\begin{cases}
\left[\ell\left(Q'\right)\right]^{\frac{1}{2}}
\left\langle\vec{f},\theta^{(\lambda')}_{Q'}\right\rangle_*
&\text{if }Q=P(Q',k_0)\text{ for some }Q'\in\mathcal{D}(\mathbb{R}^n),\\
\mathbf{0}&\text{otherwise}.
\end{cases}
\end{equation}
By this, \eqref{defExt}, and \eqref{eq-Ext-f},
we find that, for any $\lambda'\in\Lambda_{n}$,
$Q'\in\mathcal{D}(\mathbb{R}^n)$, and $x\in\mathbb{R}^{n+1}$,
$$
\left\langle\vec{f},\theta^{(\lambda')}_{Q'}\right\rangle_*
\left[\operatorname{Ext}\theta^{(\lambda')}_{Q'}\right](x)
=\frac{1}{\varphi(-k_0)}\vec{t}^{(\lambda')}_{P(Q',k_0)}
\left[\theta^{(\lambda')}\otimes\varphi\right]_{P(Q',k_0)}(x)
$$
and hence
\begin{align}\label{eq-Ext-f-x}
\operatorname{Ext}\vec{f}&=\sum_{\lambda'\in\Lambda_{n}}
\sum_{Q'\in\mathcal{D}(\mathbb{R}^n)}\frac{1}{\varphi(-k_0)}
\vec{t}^{(\lambda')}_{P(Q',k_0)}
\left[\theta^{(\lambda')}\otimes\varphi\right]_{P(Q',k_0)}\\
&=\sum_{\lambda'\in\Lambda_{n}}\sum_{Q\in\mathcal{D}(\mathbb{R}^{n+1})}
\frac{1}{\varphi(-k_0)}\vec{t}^{(\lambda')}_{Q}
\left[\theta^{(\lambda')}\otimes\varphi\right]_{Q}.\nonumber
\end{align}
We now prove that, for any $\lambda'\in\Lambda_{n}$,
$\vec{t}^{(\lambda')}\in\dot{b}^{s,\upsilon^{(n+1)}}_{p,q}(W)$.
To this end, let $\mathbb{A}(W):=\{A_{Q, W}\}_{Q\in\mathcal{D}(\mathbb{R}^{n+1})}$
and $\mathbb{A}(V):=\{A_{Q', V}\}_{Q'\in\mathcal{D}(\mathbb{R}^n)}$
be sequences of reducing operators of order $p$, respectively,
for $W$ and $V$. Let $\beta_1 \in \llbracket
d_{p, \infty}^{\text {lower}}(W),\infty)$ and $\beta_2\in \llbracket
d_{p, \infty}^{\text {upper }}(W),\infty)$,
where $d_{p, \infty}^{\text {lower}}(W)$ and
$d_{p, \infty}^{\text {upper }}(W)$ are as, respectively,
in \eqref{eq-low-dim} and \eqref{eq-upp-dim}.
Applying Definition \ref{def-red-ope}, Lemma \ref{growEST},
and \eqref{eq-Ext-B-B}, we obtain,
for any $\lambda'\in\Lambda_{n}$ and $Q'\in\mathcal{D}(\mathbb{R}^n)$,
\begin{align}\label{eq-AV-AW}
\left|A_{P(Q',k_0), W}\vec{t}^{(\lambda')}_{P(Q',k_0)}\right|
&\lesssim(1+|k_0|)^{\beta_1+\beta_2}\left|A_{P(Q',0), W}\vec{t}^{(\lambda')}_{P(Q',k_0)}\right|\\
&\sim(1+|k_0|)^{\beta_1+\beta_2}2^{j_{Q'}\frac{n+1}{p}}
\left[\int_{P(Q',0)}\left|W^{\frac{1}{p}}(x)
\vec{t}^{(\lambda')}_{P(Q',k_0)}
\right|^p\,dx\right]^{\frac{1}{p}}\nonumber\\
&\lesssim(1+|k_0|)^{\beta_1+\beta_2}
2^{j_{Q'}(\frac{n+1}{p}-\frac{\gamma}{p})}
\left[\int_{Q'}\left|V^{\frac{1}{p}}(x)
\vec{t}^{(\lambda')}_{P(Q',k_0)}\right|^p\,dx
\right]^{\frac{1}{p}}\nonumber\\
&\sim2^{-j_{Q'}(\frac{\gamma}{p}-\frac{1}{p})}	
\left|A_{Q', V}\vec{t}_{P(Q', k_0)}^{(\lambda')}\right|\notag.
\end{align}
For any $P:=Q_{j,k}\in\mathcal{D}(\mathbb{R}^{n+1})$
with $j\in\mathbb{Z}$ and $k:=(k',k_{n+1})\in\mathbb{Z}^{n+1}$,
let
\begin{align}\label{eq-I(P)}
\mathcal{I}(P):=Q_{j,k'}\in\mathcal{D}(\mathbb{R}^n).
\end{align}
By the growth condition of $\upsilon^{(n+1)}$
and the assumption that $\upsilon^{(n)}$
is the restriction of $\upsilon^{(n+1)}$, we find that,
for any $P:=Q_{j,k}\in\mathcal{D}(\mathbb{R}^{n+1})$
with $j\in\mathbb{Z}$ and $k:=(k',k_{n+1})\in\mathbb{Z}^{n+1}$,
\begin{align}\label{eq-vIP}
\frac{\upsilon^{(n)}(\mathcal{I}(P))}{\upsilon^{(n+1)}(P)}
=\frac{\upsilon^{(n+1)}(Q_{j, (k',0)})}{\upsilon^{(n+1)}(P)}
=\frac{\upsilon^{(n+1)}(Q_{j, (k',0)})}
{\upsilon^{(n+1)}(Q_{j,(k',k_{n+1})})}
\sim(1+|k_{n+1}|)^{\omega}.
\end{align}
Observe that, for any $P:=Q_{j,k}\in\mathcal{D}(\mathbb{R}^{n+1})$
with $j\in\mathbb{Z}$, $k:=(k',k_{n+1})\in\mathbb{Z}^{n+1}$,
and $|k_{n+1}|\in(|k_0|+1, \infty)$, using the basic property
of dyadic cubes in $\mathbb{R}^n$, we conclude that,
for any $Q'\subset \mathcal{I}(P)$,
\begin{align}\label{eq-P}
P(Q',k_0)\cap P=\emptyset.
\end{align}
From this, \eqref{eq-vIP}, \eqref{eq-I(P)},
\eqref{eq-AV-AW}, \eqref{eq-t}, the definitions of $\|\cdot\|_{\dot{b}^{s,\upsilon^{(n+1)}}_{p,q}(\mathbb{R}^{n+1})}$
and $\|\cdot\|_{\dot{b}^{s-\frac{\gamma}{p},
\upsilon^{(n)}}_{p,q}(\mathbb{R}^n)}$, and Remark \ref{rmk-a(A)-a},
it follows that, for any $\lambda'\in\Lambda_{n}$,
\begin{align}\label{eq-t-f-2}
&\left\|\left\{\left|A_{Q, W}\vec{t}^{(\lambda')}_Q\right|
\right\}_{Q\in\mathcal{D}(\mathbb{R}^{n+1})}
\right\|_{\dot{b}^{s,\upsilon^{(n+1)}}_{p,q}(\mathbb{R}^{n+1})}\\
&\quad=\sup_{P\in\mathcal{D}(\mathbb{R}^{n+1})}
\frac{1}{\upsilon^{(n+1)}(P)}
\left\{\sum_{j=j_P}^{\infty}\left[\sum_{
\genfrac{}{}{0pt}{}{Q\in\mathcal{D}_j(\mathbb{R}^{n+1})}{Q\subset P}}
2^{j(s+\frac{n+1}{2})p}\left|A_{Q, W}
\vec{t}_{Q}^{(\lambda')}\right|^p
|Q|\right]^{\frac{q}{p}}
\right\}^{\frac{1}{q}}\nonumber\\
&\quad\lesssim\sup_{P\in\mathcal{D}(\mathbb{R}^{n+1})}
\frac{1}{\upsilon^{(n)}(\mathcal{I}(P))}\left\{\sum_{j=j_P}^{\infty}
\left[\sum_{\genfrac{}{}{0pt}{}{Q'\in\mathcal{D}_j
(\mathbb{R}^{n})}{Q'\subset \mathcal{I}(P)}}2^{j(s+\frac{n+1}{2})p}
\left|A_{P(Q',k_0), W}\vec{t}_{P(Q', k_0)}^{(\lambda')}
\right|^p|Q'|^{\frac{n+1}{n}}\right]^{\frac{q}{p}}
\right\}^{\frac{1}{q}}\nonumber\\
&\quad\lesssim\sup_{P\in\mathcal{D}(\mathbb{R}^{n+1})}
\frac{1}{\upsilon^{(n)}(\mathcal{I}(P))}\left\{\sum_{j=j_P}^{\infty}
\left[\sum_{\genfrac{}{}{0pt}{}{Q'\in\mathcal{D}_j
(\mathbb{R}^{n})}{Q'\subset \mathcal{I}(P)}}2^{j(s-\frac{\gamma}{p}+\frac{n+1}{2})p}
\left|A_{Q', V}\vec{t}_{P(Q', k_0)}^{(\lambda')}
\right|^p|Q'|\right]^{\frac{q}{p}}
\right\}^{\frac{1}{q}}\nonumber\\
&\quad=\sup_{R\in\mathcal{D}(\mathbb{R}^{n})}
\frac{1}{\upsilon^{(n)}(R)}\left\{\sum_{j=j_R}^{\infty}
\left[\sum_{\genfrac{}{}{0pt}{}{Q'\in\mathcal{D}_j
(\mathbb{R}^{n})}{Q'\subset R}}2^{j(s-\frac{\gamma}{p}+\frac{n+1}{2})p}
\left|A_{Q', V}\vec{t}_{P(Q', k_0)}^{(\lambda')}
\right|^p|Q'|\right]^{\frac{q}{p}}\right\}^{\frac{1}{q}}\nonumber\\
&\quad=\left\|\left\{\left|A_{Q', V}
\left\langle\vec{f},\theta^{(\lambda')}_{Q'}\right\rangle_*
\right|\right\}_{Q'\in\mathcal{D}(\mathbb{R}^n)}
\right\|_{\dot{b}^{s-\frac{\gamma}{p},
\upsilon^{(n)}}_{p,q}(\mathbb{R}^n)}\notag\\
&\quad=\left\|\left\{\left\langle\vec{f},\theta^{(\lambda')}_{Q'}
\right\rangle_*\right\}_{Q'\in\mathcal{D}(\mathbb{R}^n)}\right\|
_{\dot{b}^{s-\frac{\gamma}{p},\upsilon^{(n)}}_{p,q}
(\mathbb{A}(V))}.\notag
\end{align}
By Lemmas \ref{a(A)=a(W)} and \ref{lem-Dauwav decomp},
\eqref{eq-t-f-2}, and Remark \ref{rmk-a(A)-a},
we find that, for any $\lambda'\in\Lambda_{n}$,
\begin{align}\label{eq-t-f-x}
\left\|\vec{t}^{(\lambda')}\right\|
_{\dot{b}^{s,\upsilon^{(n+1)}}_{p,q}(W)}
&\sim\left\|\vec{t}^{(\lambda')}\right\|
_{\dot{b}^{s,\upsilon^{(n+1)}}_{p,q}(\mathbb{A}(W))}
=\left\|\left\{\left|A_{Q, W}\vec{t}^{(\lambda')}_Q\right|
\right\}_{Q\in\mathcal{D}(\mathbb{R}^{n+1})}
\right\|_{\dot{b}^{s,\upsilon^{(n+1)}}_{p,q}(\mathbb{R}^{n+1})}\\
&\lesssim\left\|\left\{\left\langle\vec{f},\theta^{(\lambda')}_{Q'}
\right\rangle_*\right\}_{Q'\in\mathcal{D}(\mathbb{R}^n)}\right\|
_{\dot{b}^{s-\frac{\gamma}{p},\upsilon^{(n)}}_{p,q}
(\mathbb{A}(V))}\nonumber\\
&\sim\left\|\left\{\left\langle\vec{f},
\theta^{(\lambda')}_{Q'}\right\rangle_*
\right\}_{Q'\in\mathcal{D}(\mathbb{R}^n)}
\right\|_{\dot{b}^{s-\frac{\gamma}{p},\upsilon^{(n)}}
_{p,q}(V)}\lesssim\left\|\vec{f}
\right\|_{\dot{B}^{s-\frac{\gamma}{p},\upsilon^{(n)}}_{p,q}(V)}\notag
\end{align}
and hence
\begin{equation}\label{eq-t-b}
\vec{t}^{(\lambda')}\in\dot{b}^{s,\upsilon^{(n+1)}}_{p,q}(W).
\end{equation}
Notice that, for any $\lambda'\in\Lambda_{n}$,
$\theta^{(\lambda')}\otimes\varphi$ has bounded support.
From this, Definition \ref{KLMNmole}, \eqref{eq-int=0},
and the choice of $k$ at the beginning of the present proof,
it follows that there exists a positive
constant $C$ such that, for any $\lambda'\in\Lambda_{n}$,
$\{C[\theta^{(\lambda')}\otimes
\varphi]_Q\}_{Q\in\mathcal{D}(\mathbb{R}^{n+1})}$
is a family of synthesis molecules
for $\dot{B}^{s,\upsilon^{(n+1)}}_{p,q}(W)$.
Applying this, Lemma \ref{lem-moledecomp}(ii),
and \eqref{eq-t-b}, we obtain,
for any $\lambda'\in\Lambda_{n}$,
\begin{align}\label{eq-conver}
\sum_{Q\in\mathcal{D}(\mathbb{R}^{n+1})}\vec{t}^{(\lambda')}_{Q}
\left[\theta^{(\lambda')}\otimes\varphi\right]_{Q}
\end{align}
converges in $[\mathcal{S}_\infty'(\mathbb{R}^{n+1})]^m$ and
\begin{align}\label{eq-t-B-B}
\left\|\sum_{Q\in\mathcal{D}(\mathbb{R}^{n+1})}\vec{t}^{(\lambda')}_Q
\left[\theta^{(\lambda')}\otimes \varphi\right]_Q\right\|
_{\dot{B}^{s,\upsilon^{(n+1)}}_{p,q}(W)}\lesssim
\left\|\vec{t}^{(\lambda')}\right\|
_{\dot{b}^{s,\upsilon^{(n+1)}}_{p,q}(W)}.
\end{align}
Using \eqref{eq-Ext-f-x}, \eqref{eq-conver},
\eqref{eq-t-B-B}, \eqref{eq-t-f-x}, the quasi-triangle
inequality of $\|\cdot\|_{\dot{B}^{s,\upsilon^{(n+1)}}_{p,q}(W)}$,
and Lemma \ref{lem-Dauwav decomp},
we conclude that, for any $\vec{f}\in\dot{B}^{s-\frac{\gamma}{p},
\upsilon^{(n)}}_{p,q}(V)$, $\operatorname{Ext}\vec{f}$
is well-defined in $[\mathcal{S}_\infty'(\mathbb{R}^{n+1})]^m$ and
\begin{align*}
\left\|\operatorname{Ext}\vec{f}\right\|
_{\dot{B}^{s,\upsilon^{(n+1)}}_{p,q}(W)}
&\lesssim\sum_{\lambda'\in\Lambda_{n}}
\left\|\sum_{Q\in\mathcal{D}(\mathbb{R}^{n+1})}\vec{t}^{(\lambda')}_Q
\left[\theta^{(\lambda')}\otimes \varphi\right]_Q\right\|
_{\dot{B}^{s,\upsilon^{(n+1)}}_{p,q}(W)}
\lesssim\sum_{\lambda'\in\Lambda_{n}}
\left\|\vec{t}^{(\lambda')}\right\|
_{\dot{b}^{s,\upsilon^{(n+1)}}_{p,q}(W)}\\
&\lesssim\sum_{\lambda'\in\Lambda_{n}}
\left\|\left\{\left\langle\vec{f},
\theta^{(\lambda')}_{Q'}\right\rangle_*
\right\}_{Q'\in\mathcal{D}(\mathbb{R}^n)}
\right\|_{\dot{b}^{s-\frac{\gamma}{p},
\upsilon^{(n)}}_{p,q}(V)}\sim\left\|\vec{f}\right\|
_{\dot{B}^{s-\frac{\gamma}{p},\upsilon^{(n)}}_{p,q}(V)},
\end{align*}
which further implies that $\operatorname{Ext}$ is bounded from
$\dot{B}^{s-\frac{\gamma}{p},\upsilon^{(n)}}_{p,q}(V)$
to $\dot{B}^{s,\upsilon^{(n+1)}}_{p,q}(W)$ and hence
completes the proof of the above claim.

To finish the proof, it remains to show that,
if $s\in(\frac{\gamma}{p}+E+\frac{d_{p,
\infty}^{\mathrm{upper}}(V)}{p},\infty)$ and
\eqref{eq-tr-B-B} holds, then $\operatorname{Tr}\circ
\operatorname{Ext}$ is the identity on
$\dot{B}^{s-\frac{\gamma}{p},\upsilon^{(n)}}_{p,q}(V)$.
From \eqref{eq-Ext-f}, it follows that,
for any $\vec{f}\in\dot{B}^{s-\frac{\gamma}{p},
\upsilon^{(n)}}_{p,q}(V)$,
\begin{align}\label{eq-Tr-Ext-f}
\left(\operatorname{Tr}\circ\operatorname{Ext}\right)\left(\vec{f}\right)
=\operatorname{Tr}\left(\sum_{\lambda'\in\Lambda_{n}}
\sum_{Q'\in\mathcal{D}(\mathbb{R}^n)}\left\langle\vec{f},
\theta^{(\lambda')}_{Q'}\right\rangle_*
\operatorname{Ext}\theta^{(\lambda')}_{Q'}\right).
\end{align}
Applying \eqref{defExt} and \eqref{eq-theta}, we obtain
$$
\left\{\frac{\varphi(-k_0)}{[\ell(Q')]^{\frac12}}
\operatorname{Ext}\theta^{(\lambda')}_{Q'}
\right\}_{\lambda'\in\Lambda_{n}, Q'\in\mathcal{D}(\mathbb{R}^n)}
\subset\left\{\theta^{(\lambda)}_Q\right\}
_{\lambda\in\Lambda_{n+1}, Q\in\mathcal{D}(\mathbb{R}^{n+1})}.
$$
By this, \eqref{eq-Tr-Ext-f}, \eqref{eq-def-Trf}, \eqref{TrExt},
and Lemma \ref{lem-Dauwav decomp}, we find that,
for any $\vec{f}\in\dot{B}^{s-\frac{\gamma}{p},
\upsilon^{(n)}}_{p,q}(V)$,
\begin{align*}
\left(\operatorname{Tr}\circ\operatorname{Ext}\right)\left(\vec{f}\right)
&=\sum_{\lambda'\in\Lambda_{n}}\sum_{Q'\in\mathcal{D}(\mathbb{R}^n)}
\left\langle\vec{f},\theta^{(\lambda')}_{Q'}\right\rangle_*
\left(\operatorname{Tr}\circ\operatorname{Ext}\right)
\theta^{(\lambda')}_{Q'}\\
&=\sum_{\lambda'\in\Lambda_{n}}\sum_{Q'\in\mathcal{D}(\mathbb{R}^n)}
\left\langle\vec{f},\theta^{(\lambda')}_{Q'}
\right\rangle_*\theta^{(\lambda')}_{Q'}=\vec{f}\notag,
\end{align*}
which further implies that
$\operatorname{Tr}\circ\operatorname{Ext}$
is the identity on $\dot{B}^{s-\frac{\gamma}{p},
\upsilon^{(n)}}_{p,q}(V)$.
This finishes the proof of Theorem \ref{thm-Ext-B}.
\end{proof}

\begin{remark}
Let all the symbols be the same as in Theorems \ref{thm-Tr-B} and
\ref{thm-Ext-B}. Let $\gamma=1$, $\tau\in[0, \infty)$, and,
for any $Q\in\mathcal{D}(\mathbb{R}^{n+1})$, let
$\upsilon^{(n+1)}(Q):=|Q|^{\tau}$. Observe that
the restriction $\upsilon^{(n)}$ of $\upsilon^{(n+1)}$
on $\mathcal{D}(\mathbb{R}^{n})$ satisfies that, for
any $Q\in\mathcal{D}(\mathbb{R}^{n})$, $\upsilon^{(n)}(Q)=|Q|^{\frac{n+1}{n}\tau}$.
In this case, the spaces $\dot{B}^{s,\upsilon^{(n+1)}}_{p,q}(W)$ and
$\dot{B}^{s-\frac{\gamma}{p},\upsilon^{(n)}}_{p,q}(V)$
in Theorems \ref{thm-Tr-B} and \ref{thm-Ext-B} are respectively
$\dot{B}^{s,\tau}_{p,q}(W)$ and
$\dot{B}^{s-\frac{1}{p},\frac{n+1}{n}\tau}_{p,q}(V)$.
On the one hand, Theorems \ref{thm-Tr-B} and
\ref{thm-Ext-B} are the homogeneous variants, respectively, of
\cite[Theorems 6.3 and 6.5]{bhyy5} in which Bu et al.
obtained the corresponding results for inhomogeneous spaces $B^{s,\tau}_{p,q}(W)$ and $B^{s-\frac{1}{p},\frac{n+1}{n}\tau}_{p,q}(V)$;
on the other hand, Bu et al. \cite[Theorems 5.6 and 5.10]{bhyy3}
also established the corresponding results in
Theorems \ref{thm-Tr-B} and \ref{thm-Ext-B} for
$\dot{B}^{s,\tau}_{p,q}(W)$ and
$\dot{B}^{s-\frac{1}{p},\frac{n+1}{n}\tau}_{p,q}(V)$
with $W, V$ being matrix $\mathcal{A}_p$ weights.
However, Theorems \ref{thm-Tr-B} and
\ref{thm-Ext-B} only respectively coincide with
\cite[Theorems 5.6 and 5.10]{bhyy3} for the case where $p\in(0, 1]$.
When $p\in(1, \infty)$, the range of the index $s$
in \cite[Theorems 5.6 and 5.10]{bhyy3} may be better than
that in Theorems \ref{thm-Tr-B} and \ref{thm-Ext-B};
see Remark \ref{rmk-compare} for the reason.
Furthermore, the spaces $\dot{B}^{s,0}_{p,q}(W)$
and $\dot{B}^{s-\frac{1}{p},0}_{p,q}(V)$ respectively
reduce to classical spaces $\dot{B}^{s}_{p,q}(W)$
and $\dot{B}^{s-\frac{1}{p}}_{p,q}(V)$
with $W, V$ being matrix $\mathcal{A}_{p,\infty}$ weights.
Even for these spaces,  Theorems \ref{thm-Tr-B} and
\ref{thm-Ext-B} are also new.

It is worth pointing out that
Frazier and Roudenko \cite[Theorems 1.2 and 1.3]{fr08}
also obtained the boundedness of the trace and the
extension operators for $\dot{B}^{s}_{p,q}(W)$ and
$\dot{B}^{s-\frac{1}{p}}_{p,q}(V)$ with
$W, V$ being ``$p$-admissible'' matrix weights, where $p\in (0,\infty)$
(see, for example, \cite[p.\,182]{fr08} and \cite[p.\,277]{rou03}).
We now claim that, for any $p\in(0,n]$, the matrix $\mathcal{A}_{p,\infty}$
class and the matrix $p$-admissible class can not contain each other.
Therefore, Theorems \ref{thm-Tr-B} and \ref{thm-Ext-B}
and \cite[Theorems 1.2 and 1.3]{fr08} have their own domains of applicability.
To see this, recall that the doubling exponent must be greater than
$n$ (see \cite[Proposition 2.10(iii)]{hs14}). As a result, in this case,
the matrix $p$-admissible class includes matrix $\mathcal{A}_p$ weights
and diagonal matrix weights with the doubling property.
Notice that the matrix weight $W_{\alpha,\beta}$ in \cite[Example 2.4]{byyz25}
with $\beta\in(n(p-1)_+,\infty)$ and $\alpha\in[\frac{|\beta|}{p},\infty)$
is exactly a matrix $\mathcal{A}_{p,\infty}$ weight but not
a matrix $p$-admissible weight. On the other hand, let $w$ be a
scalar doubling weight but not an $A_{\infty}$ weight (see \cite[(2.4)]{fm74} for such an example)
and $W:=wI_m$, where $I_m$ is the identity matrix of the order $m$.
From \cite[Lemma 3.2(ii)]{bhyy4}, it follows that $W$ is not a matrix $\mathcal{A}_{p,\infty}$
weight. However, $W$ is clearly a $p$-admissible matrix weight.
These two examples show the above claim.

In particular, when $m=1$ (the scalar-valued setting) and $W, V\equiv1$,
Theorem \ref{thm-Tr-B} improves \cite[Theorem 1.3]{syy10}
in which Sawano et al. established the boundedness of the trace
operator for $\dot{B}^{s,\tau}_{p,q}(\mathbb{R}^{n+1})$
with $\tau$ having an upper bound. Moreover,
we also refer to \cite{fra24} for a recent result
on the boundedness of the trace and the extension operators
on weighted mixed-norm Besov spaces.
\end{remark}

Next, we establish the boundedness of the trace operator on
generalized matrix-weighted Triebel--Lizorkin-type spaces.

\begin{theorem}\label{thm-Tr-F}
Let $p, \gamma\in(0,\infty)$, $q\in(0,\infty]$,
$W\in \mathcal{A}_{p,\infty}(\mathbb{R}^{n+1},\mathbb{C}^m)$,
and $V\in \mathcal{A}_{p,\infty}(\mathbb{R}^{n},\mathbb{C}^m)$.
Suppose that $\delta_2\in[0, \infty)$, $\delta_1\in[0, \delta_2]$,
$\omega\in[0,(n+1)(\delta_2-\delta_1)]$,
$\upsilon^{(n+1)}$ is a $(\delta_1, \delta_2; \omega)$-order
growth function on $\mathcal{D}(\mathbb{R}^{n+1})$,
and ${\upsilon}^{(n)}$ is the restriction of
${\upsilon}^{(n+1)}$ on $\mathcal{D}(\mathbb{R}^n)$.
If $s\in(\frac{\gamma}{p}+E+\frac{d_{p,\infty}
^{\mathrm{upper}}(V)}{p},\infty)$ with
\begin{align}\label{eq-E-F}
E:=\begin{cases}
\displaystyle n\left(\frac{1}{p}-\delta_1\right)
&\displaystyle\text{if }\delta_1>\frac{1}{p},\\
\displaystyle n\left(\frac{1}{p}-1\right)_+
&\text{otherwise}
\end{cases}
\end{align}
and $d_{p,\infty}^{\mathrm{upper}}(V)$ being as in \eqref{eq-upp-dim},
then the trace operator $\operatorname{Tr}$ in \eqref{defTr}
can be extended to a bounded linear operator from
$\dot{F}^{s,\upsilon^{(n+1)}}_{p,q}(W)$
to $\dot{B}^{s-\frac{\gamma}{p},\upsilon^{(n)}}_{p,p}(V)$
if and only if \eqref{eq-tr-B-B} holds.
\end{theorem}
\begin{proof}
To prove the present theorem, let $k\in\mathbb{N}$
satisfy \eqref{eq-k} in Lemma \ref{lem-Dauwav decomp}
on both $\dot{F}^{s,\upsilon^{(n+1)}}_{p,q}(W)$
and $\dot{B}^{s-\frac{\gamma}{p},\upsilon^{(n)}}_{p,p}(V)$ and
let $\varphi,\psi\in C^k(\mathbb{R})$ be as in
Lemma \ref{def-Dauwavv2}. For any $\lambda\in\Lambda_{n+1}$,
let $\theta^{(\lambda)}$ be as in \eqref{eq-theta}.
The proof of the necessity is similar to that of the necessity
of Theorem \ref{thm-Tr-B}; we omit the details.

Next, we show the sufficiency.
By Lemma \ref{lem-Dauwav decomp}, we find that,
for any $\vec{f}\in\dot{F}^{s,\upsilon^{(n+1)}}_{p,q}(W)$,
\begin{align}\label{eq-expan-f}
\vec{f}=\sum_{\lambda\in\Lambda_{n+1}}
\sum_{Q\in\mathcal{D}(\mathbb{R}^{n+1})}\left\langle
\vec{f},\theta^{(\lambda)}_Q\right\rangle_*\theta^{(\lambda)}_Q
\end{align}
in $[\mathcal{S}_\infty'(\mathbb{R}^{n+1})]^m$,
where $\langle\cdot,\cdot \rangle_*$ is as in \eqref{eq-f-m}.
Based on the expansion \eqref{eq-expan-f},
for any $\vec{f}\in\dot{F}^{s,\upsilon^{(n+1)}}_{p,q}(W)$,
we define the action of $\operatorname{Tr}$ on $\vec{f}$ by
setting
\begin{align*}
\operatorname{Tr}\vec{f}:=\sum_{\lambda\in\Lambda_{n+1}}
\sum_{Q \in \mathcal{D}(\mathbb{R}^{n+1})}\left\langle\vec{f}, \theta_Q^{(\lambda)}\right\rangle_*\operatorname{Tr} \theta_Q^{(\lambda)},
\end{align*}
which, combined with \eqref{defTr} and \eqref{eq-Trf},
further implies that
\begin{align}\label{eq-Trf-F}
\operatorname{Tr}\vec{f}
=\sum_{\lambda\in \Lambda_{n+1}}\sum_{i=-M}^M
\sum_{Q'\in\mathcal{D}(\mathbb{R}^n)}
\vec{t}_{P(Q',i)}^{(\lambda)}
\left[\ell\left(Q'\right)\right]^{\frac{1}{2}}
\theta_{P(Q',i)}^{(\lambda)}\left(\cdot, 0\right),
\end{align}
where $M\in\mathbb{N}$ and $\{\vec{t}^{(\lambda)}_Q\}
_{Q\in\mathcal{D}(\mathbb{R}^{n+1})}$ are as,
respectively, in \eqref{eq-Trf} and \eqref{eq-t-lambda}.

We now prove that, for any $\vec{f}\in\dot{F}^{s,\upsilon^{(n+1)}}_{p,q}(W)$,
$\operatorname{Tr}\vec{f}$ is well-defined
in $[\mathcal{S}'_{\infty}(\mathbb{R}^{n})]^m$ and
$\operatorname{Tr}$ is bounded from
$\dot{F}^{s,\upsilon^{(n+1)}}_{p,q}(W)$
to $\dot{B}^{s-\frac{\gamma}{p},\upsilon^{(n)}}_{p,p}(V)$.
To this end, let $\mathbb{A}(W):=\{A_{Q, W}\}_{Q\in\mathcal{D}(\mathbb{R}^{n+1})}$
and $\mathbb{A}(V):=\{A_{Q', V}\}_{Q'\in\mathcal{D}(\mathbb{R}^n)}$
be sequences of reducing operators of order $p$, respectively,
for $W$ and $V$. From Lemma \ref{a(A)=a(W)}, Remark \ref{rmk-a(A)-a},
and \eqref{eq-A_n-A_n+1}, we deduce that,
for any $\lambda\in\Lambda_{n+1}$ and $i\in\mathbb{Z}$ with $|i|\leq M$,
\begin{align}\label{eq-t-x}
&\left\|\left\{\vec{t}^{(\lambda)}_{P(Q',i)}
\right\}_{Q'\in\mathcal{D}(\mathbb{R}^n)}
\right\|_{\dot{b}^{s-\frac{\gamma}{p},\upsilon^{(n)}}_{p,p}(V)}\\
&\quad\sim\left\|\left\{\vec{t}^{(\lambda)}_{P(Q',i)}
\right\}_{Q'\in\mathcal{D}(\mathbb{R}^n)}
\right\|_{\dot{b}^{s-\frac{\gamma}{p},\upsilon^{(n)}}
_{p,p}(\mathbb{A}(V))}
=\left\|\left\{\left|A_{Q', V}\vec{t}^{(\lambda)}
_{P(Q',i)}\right|\right\}_{Q'\in\mathcal{D}(\mathbb{R}^n)}
\right\|_{\dot{b}^{s-\frac{\gamma}{p},\upsilon^{(n)}}
_{p,p}(\mathbb{R}^{n})}\notag\\
&\quad\lesssim\left\|\left\{2^{j_{Q'}(\frac{\gamma}{p}
-\frac{1}{p})}\left|A_{P(Q',i), W}
\vec{t}^{(\lambda)}_{P(Q',i)}\right|\right\}
_{Q'\in\mathcal{D}(\mathbb{R}^n)}\right\|_{\dot{b}^{s-\frac{\gamma}{p},
\upsilon^{(n)}}_{p,p}(\mathbb{R}^{n})}.\nonumber
\end{align}
For any $i\in\mathbb{Z}$ with $|i|\leq M$ and
for any $R\in\mathcal{D}(\mathbb{R}^n)$,
let $R(i)\in\mathcal{D}(\mathbb{R}^{n+1})$ be
as in Lemma \ref{lem-P_R} with $Q'$ and $Q$
replaced, respectively, by $R$ and $R(i)$.
For any $Q'\in\mathcal{D}(\mathbb{R}^n)$ and $i\in\mathbb{Z}$, let
\begin{equation}\label{eq-E_PQ}
E_{P(Q',i)}:=Q'\times\left[\ell(Q')\left(i+\frac13\right),
\ell(Q')\left(i+\frac23\right)\right).
\end{equation}
Observe that, for any $i\in\mathbb{Z}$,
$\{E_{P(Q',i)}\}_{Q'\in\mathcal{D}(\mathbb{R}^n)}$
is a collection of pairwise disjoint sets.
This, combined with the definition of
$\|\cdot\|_{\dot{b}^{s-\frac{\gamma}{p},
\upsilon^{(n)}}_{p,p}(\mathbb{R}^n)}$,
Lemma \ref{lem-P_R}, and \eqref{eq-vRi},
further implies that,
for any $\lambda\in\Lambda_{n+1}$ and
$i\in\mathbb{Z}$ with $|i|\leq M$,
\begin{align}\label{eq-x}
&\left\|\left\{2^{j_{Q'}(\frac{\gamma}{p}
-\frac{1}{p})}\left|A_{P(Q',i), W}
\vec{t}^{(\lambda)}_{P(Q',i)}\right|\right\}
_{Q'\in\mathcal{D}(\mathbb{R}^n)}\right\|_{\dot{b}^{s-\frac{\gamma}{p},
\upsilon^{(n)}}_{p,p}(\mathbb{R}^n)}\\
&\quad=\sup_{R\in\mathcal{D}(\mathbb{R}^n)}\frac{1}{\upsilon^{(n)}(R)}
\left[\sum_{j=j_R}^\infty\sum_{\genfrac{}{}{0pt}{}{Q'\in\mathcal{D}_{j}
(\mathbb{R}^n)}{Q'\subset R}}
2^{j(s+\frac{n}{2})p}\left|A_{P(Q',i), W}
\vec{t}_{P(Q',i)}^{(\lambda)}\right|^p
|Q'|^{\frac{n+1}{n}}\right]^{\frac{1}{p}}\notag\\
&\quad\sim\sup_{R\in\mathcal{D}(\mathbb{R}^n)}
\frac{1}{\upsilon^{(n+1)}(R(i))}\left\{\sum_{j=j_R}^\infty\int_{R(i)}
\sum_{\genfrac{}{}{0pt}{}{Q'\in\mathcal{D}_{j}(\mathbb{R}^n)}{Q'\subset R}}
\left[|P(Q',i)|^{-\frac{s}{n+1}-\frac{1}{2}}\right.\right.\notag\\
&\quad\quad\times\left.\left|A_{P(Q',i), W}\left\langle\vec{f},
\theta^{(\lambda)}_{P(Q',i)}\right\rangle_*
\right|\mathbf{1}_{E_{P(Q',i)}}(x)\right]^p\,dx\Bigg\}^{\frac{1}{p}}\notag,
\end{align}
where the last equivalence follows from the fact that
$|E_{P(Q',i)}|=\frac{1}{3}|P(Q',i)|$ for any $i\in\mathbb{Z}$
and $Q'\in\mathcal{D}(\mathbb{R}^n)$.
Applying this again, \eqref{eq-t-x}, \eqref{eq-x},
the disjointness of $\{E_{P(Q',i)}\}_{Q'\in\mathcal{D}(\mathbb{R}^n)}$
for any $i\in\mathbb{Z}$, the definition of $\|\cdot\|_{\dot{f}^{s,\upsilon^{(n+1)}}_{p,q}(\mathbb{A}(W))}$,
and Lemmas \ref{lem-P_R}, \ref{a(A)=a(W)}, and \ref{lem-Dauwav decomp},
we obtain, for any $\lambda\in\Lambda_{n+1}$ and
$i\in\mathbb{Z}$ with $|i|\leq M$,
\begin{align}\label{eq-t-y}
&\left\|\left\{\vec{t}^{(\lambda)}_{P(Q',i)}
\right\}_{Q'\in\mathcal{D}(\mathbb{R}^n)}
\right\|_{\dot{b}^{s-\frac{\gamma}{p},\upsilon^{(n)}}_{p,p}(V)}\\
&\quad\lesssim\sup_{R\in\mathcal{D}(\mathbb{R}^n)}
\frac{1}{\upsilon^{(n+1)}(R(i))}
\left(\int_{R(i)}\Bigg\{\sum_{
\genfrac{}{}{0pt}{}{Q'\in\mathcal{D}
(\mathbb{R}^n)}{Q'\subset R}}
\left[|P(Q',i)|^{-\frac{s}{n+1}-\frac{1}{2}}
\right.\right.\notag\\
&\quad\quad\left.\left.\times\left|A_{P(Q',i), W}
\left\langle\vec{f},\theta^{(\lambda)}_{P(Q',i)}
\right\rangle_*\right|\mathbf{1}_{E_{P(Q',i)}}(x)
\right]^q\Bigg\}^{\frac{p}{q}}\,dx\right)^{\frac{1}{p}}\notag\\
&\quad\lesssim\sup_{R\in\mathcal{D}(\mathbb{R}^n)}
\frac{1}{\upsilon^{(n+1)}(R(i))}
\left(\int_{R(i)}\left\{\sum_{
\genfrac{}{}{0pt}{}{Q\in\mathcal{D}
(\mathbb{R}^{n+1})}{Q\subset R(i)}}
\left[|Q|^{-\frac{s}{n+1}}
\left|A_{Q, W}
\left\langle\vec{f},\theta^{(\lambda)}_{Q}
\right\rangle_*\right|\widetilde{\mathbf{1}}_{Q}(x)
\right]^q\right\}^{\frac{p}{q}}\,dx\right)^{\frac{1}{p}}\notag\\
&\quad\leq\sup_{P\in\mathcal{D}(\mathbb{R}^{n+1})}\frac{1}{\upsilon^{(n+1)}(P)}
\left(\int_P\left\{\sum_{\genfrac{}{}{0pt}{}{Q\in\mathcal{D}
(\mathbb{R}^{n+1})}{Q\subset P}}\left[|Q|^{-\frac{s}{n+1}}
\left|A_{Q,W}\left\langle\vec{f},\theta^{(\lambda)}_Q\right\rangle_*
\right|\widetilde{\mathbf{1}}_Q (x)\right]^q\right\}^{\frac{p}{q}}\,dx\right)^{\frac{1}{p}}\notag\\
&\quad=\left\|\left\{\left\langle\vec{f},
\theta^{(\lambda)}_Q \right\rangle_*\right\}_{Q\in\mathcal{D}(\mathbb{R}^{n+1})}
\right\|_{\dot{f}^{s,\upsilon^{(n+1)}}_{p,q}(\mathbb{A}(W))}\nonumber\\
&\quad\sim\left\|\left\{\left\langle\vec{f},
\theta^{(\lambda)}_Q \right\rangle_*\right\}_{Q\in\mathcal{D}(\mathbb{R}^{n+1})}
\right\|_{\dot{f}^{s,\upsilon^{(n+1)}}_{p,q}(W)}
\lesssim\left\|\vec{f}\right\|
_{\dot{F}^{s,\upsilon^{(n+1)}}_{p,q}(W)}\notag
\end{align}
and hence
\begin{align}\label{eq-t-b_pp}
\left\{\vec{t}^{(\lambda)}_{P(Q',i)}\right\}_{Q'\in\mathcal{D}
(\mathbb{R}^n)}\in\dot{b}^{s-\frac{\gamma}{p},\upsilon^{(n)}}_{p,p}(V).
\end{align}
By an argument similar to that used in the proof of Theorem \ref{thm-Tr-B},
we find that there exists a positive constant $C$
such that, for any $\lambda\in\Lambda_{n+1}$ and
$i\in\mathbb{Z}$ with $|i|\leq M$,
$\{C[\ell(Q')]^{\frac{1}{2}}\theta_{P(Q',i)}^{(\lambda)}
(\cdot, 0)\}_{Q'\in\mathcal{D}(\mathbb{R}^n)}$
is a family of synthesis molecules for
$\dot{B}_{p, p}^{s-\frac{\gamma}{p}, \upsilon^{(n)}}(V)$.
Using this, Lemma \ref{lem-moledecomp}(ii), and \eqref{eq-t-b_pp},
we conclude that, for any $\lambda\in\Lambda_{n+1}$
and $i\in\mathbb{Z}$ with $|i|\leq M$,
\begin{align*}
\sum_{Q'\in\mathcal{D}(\mathbb{R}^n)}
\vec{t}_{P(Q',i)}^{(\lambda)}
\left[\ell\left(Q'\right)\right]^{\frac{1}{2}}
\theta_{P(Q',i)}^{(\lambda)}\left(\cdot, 0\right)
\end{align*}
converges in $[\mathcal{S}'_{\infty}(\mathbb{R}^{n})]^m$ and
\begin{align*}
\left\|\sum_{Q'\in\mathcal{D}(\mathbb{R}^n)}
\vec{t}^{(\lambda)}_{P(Q',i)}[\ell(Q')]^{\frac12}
\theta^{(\lambda)}_{P(Q',i)}(\cdot,0)\right\|
_{\dot{B}^{s-\frac{\gamma}{p},\upsilon^{(n)}}_{p,p}(V)}
\lesssim\left\|\left\{\vec{t}^{(\lambda)}_{P(Q',i)}\right\}
_{Q'\in\mathcal{D}(\mathbb{R}^n)}\right\|_{\dot{b}^{s-\frac{\gamma}{p},
\upsilon^{(n)}}_{p,p}(V)}.
\end{align*}
From this, the quasi-triangle inequality of
$\|\cdot\|_{\dot{B}^{s-\frac{\gamma}{p},\upsilon^{(n)}}_{p,p}(V)}$,
\eqref{eq-t-y}, \eqref{eq-Trf-F}, and Lemma \ref{lem-Dauwav decomp},
we infer that, for any $\vec{f}\in\dot{F}^{s,\upsilon^{(n+1)}}_{p,q}(W)$,
$\operatorname{Tr}\vec{f}$ is well-defined in
$[\mathcal{S}'_{\infty}(\mathbb{R}^{n})]^m$ and
\begin{align}\label{eq-Trf-f-F}
\left\|\operatorname{Tr}\vec{f}\right\|_{\dot{B}^{s-\frac{\gamma}{p},
\upsilon^{(n)}}_{p,p}(V)}
&\lesssim\sum_{\lambda\in\Lambda_{n+1}}\sum_{i=-M}^M
\left\|\sum_{Q'\in\mathcal{D}(\mathbb{R}^n)}
\vec{t}^{(\lambda)}_{P(Q',i)}[\ell(Q')]^{\frac12}
\theta^{(\lambda)}_{P(Q',i)}(\cdot,0)\right\|
_{\dot{B}^{s-\frac{\gamma}{p},\upsilon^{(n)}}_{p,p}(V)}\\
&\lesssim\sum_{\lambda\in\Lambda_{n+1}}\sum_{i=-M}^M
\left\|\left\{\vec{t}^{(\lambda)}_{P(Q',i)}\right\}
_{Q'\in\mathcal{D}(\mathbb{R}^n)}\right\|_{\dot{b}^{s-\frac{\gamma}{p},
\upsilon^{(n)}}_{p,p}(V)}\nonumber\\
&\lesssim\sum_{\lambda\in\Lambda_{n+1}}
\left\|\left\{\left\langle\vec{f},\theta^{(\lambda)}_{Q}
\right\rangle_*\right\}_{Q\in\mathcal{D}(\mathbb{R}^{n+1})}
\right\|_{\dot{f}^{s,\upsilon^{(n+1)}}
_{p,q}(W)}\sim\left\|\vec{f}\right\|
_{\dot{F}^{s,\upsilon^{(n+1)}}_{p,q}(W)}.\nonumber
\end{align}
Applying \eqref{eq-Trf-f-F}, we also obtain
the trace operator $\operatorname{Tr}$ is bounded from
$\dot{F}^{s,\upsilon^{(n+1)}}_{p,q}(W)$ to $\dot{B}^{s-\frac{\gamma}{p},\upsilon^{(n)}}_{p,p}(V)$.
This finishes the proof of the
sufficiency and hence Theorem \ref{thm-Tr-F}.
\end{proof}

\begin{remark}
By Remark \ref{rmk-Tr-B},
we find that the assumption $\delta_1\in[0, \delta_2]$
is only used in the proof of necessity.
Moreover, one can also verify that
the proof of the sufficiency of Theorem \ref{thm-Tr-F}
is also valid for any $\delta_1\in(-\infty, \delta_2]$;
we omit the details.
\end{remark}

In order to obtain the boundedness of
the extension operator on generalized
matrix-weighted Triebel--Lizorkin-type spaces,
we need   the following lemma,
which is precisely \cite[Lemma 3.10]{byyz24}.

\begin{lemma}\label{lem-aLA-equi}
Let $(A, a)\in\{(B, b), (F, f)\}$, $s\in\mathbb{R}$,
$p, q\in(0,\infty]$ ($p<\infty$ if $a=f$), and
$\upsilon$ be a positive function defined on $\mathcal{D}(\mathbb{R}^n)$.
If $\epsilon\in(0,1]$ and $\{E_Q\}_{Q\in\mathcal{D}
(\mathbb{R}^n)}$ is a sequence of measurable sets with $E_Q \subset Q$
and $|E_Q|\geq\epsilon|Q|$ for any $Q\in\mathcal{D}
(\mathbb{R}^n)$, then, for any sequence
$t:=\{t_Q\}_{Q\in\mathcal{D}(\mathbb{R}^n)}$ in $\mathbb{C}$,
\begin{align*}
\left\|t\right\|_{\dot{a}_{p,q}^{s,\upsilon}}
\sim\left\|\left\{2^{js}\sum_{Q\in\mathcal{D}_j(\mathbb{R}^n)}
t_Q\widetilde{\mathbf{1}}_{E_Q}\right\}
_{j\in\mathbb{Z}}\right\|_{L\dot{A}_{p, q}^{\upsilon}},
\end{align*}
where $\widetilde{\mathbf{1}}_{E_Q}
:=|E_Q|^{-\frac{1}{2}}\mathbf{1}_{E_Q}$,
$\|\cdot\|_{L\dot{A}_{p, q}^{\upsilon}}$ is
as in \eqref{LA_nu} and
the positive equivalence
constants are independent of $t$.
\end{lemma}

\begin{theorem}\label{thm-Ext-F}
Let $s\in\mathbb{R}$,  $p, \gamma\in(0,\infty)$, $q\in(0,\infty]$,
$W\in \mathcal{A}_{p, \infty}(\mathbb{R}^{n+1},\mathbb{C}^m)$,
and $V\in \mathcal{A}_{p, \infty}(\mathbb{R}^{n},\mathbb{C}^m)$.
Assume that $\delta_2\in[0, \infty)$, $\delta_1\in[0, \delta_2]$,
$\omega\in[0,(n+1)(\delta_2-\delta_1)]$,
$\upsilon^{(n+1)}$ is a $(\delta_1, \delta_2; \omega)$-order
growth function on $\mathcal{D}(\mathbb{R}^{n+1})$,
and ${\upsilon}^{(n)}$ is the restriction of
$\upsilon^{(n+1)}$ on $\mathcal{D}(\mathbb{R}^n)$.
If \eqref{eq-Ext-B-B} holds, then the extension operator
$\operatorname{Ext}$ in \eqref{defExt} can be extended
to a bounded linear operator from
$\dot{B}^{s-\frac{\gamma}{p},\upsilon^{(n)}}_{p,p}(V)$
to $\dot{F}^{s,\upsilon^{(n+1)}}_{p,q}(W)$.
Furthermore, if $s\in(\frac{\gamma}{p}+E+
\frac{d_{p,\infty}^{\mathrm{upper}}(V)}{p},\infty)$
and \eqref{eq-tr-B-B} holds,
where $E$ and $d_{p,\infty}^{\mathrm{upper}}(V)$ are as,
respectively, in \eqref{eq-E-F} and \eqref{eq-upp-dim},
then $\operatorname{Tr}\circ\operatorname{Ext}$ is the identity on
$\dot{B}^{s-\frac{\gamma}{p},\upsilon^{(n)}}_{p,p}(V)$.
\begin{proof}
To prove the present theorem, let $k\in\mathbb{N}$
satisfy \eqref{eq-k} in Lemma \ref{lem-Dauwav decomp}
on both $\dot{F}^{s,\upsilon^{(n+1)}}_{p,q}(W)$
and $\dot{B}^{s-\frac{\gamma}{p},\upsilon^{(n)}}_{p,p}(V)$ and
$\varphi,\psi\in C^k(\mathbb{R})$ be as in Lemma \ref{def-Dauwavv2}.
For any $\lambda'\in\Lambda_{n}$ and $\lambda\in\Lambda_{n+1}$,
let $\theta^{(\lambda')}$ and $\theta^{(\lambda)}$ be as in \eqref{eq-theta}.
Using Lemma \ref{lem-Dauwav decomp}, we conclude that,
for any $\vec{f}\in\dot{B}^{s-\frac{\gamma}{p},
\upsilon^{(n)}}_{p,p}(V)$,
$$
\vec{f}=\sum_{\lambda'\in\Lambda_{n}}\sum_{Q'\in\mathcal{D}(\mathbb{R}^n)}
\left\langle\vec{f},\theta^{(\lambda')}_{Q'}\right\rangle_*
\theta^{(\lambda')}_{Q'}
$$
in $[\mathcal{S}_\infty'(\mathbb{R}^{n})]^m$,
where $\langle\cdot,\cdot \rangle_*$ is as in \eqref{eq-f-m}.
For any $\vec{f}\in\dot{B}^{s-\frac{\gamma}{p},
\upsilon^{(n)}}_{p,p}(V)$,
let $\operatorname{Ext}\vec{f}$ be as in \eqref{eq-Ext-f}.

Next, we show that, for any $\vec{f}\in\dot{B}^{s-\frac{\gamma}{p},
\upsilon^{(n)}}_{p,p}(V)$, $\operatorname{Ext}\vec{f}$
is well-defined in $[\mathcal{S}_\infty'(\mathbb{R}^{n+1})]^m$
and $\operatorname{Ext}$ is bounded from
$\dot{B}^{s-\frac{\gamma}{p},\upsilon^{(n)}}_{p,p}(V)$
to $\dot{F}^{s,\upsilon^{(n+1)}}_{p,q}(W)$.
To this end, for any given $\vec{f}\in\dot{B}^{s-\frac{\gamma}{p},
\upsilon^{(n)}}_{p,p}(V)$ and for any $\lambda'\in\Lambda_{n}$,
let $\vec{t}^{(\lambda')}:=\{\vec{t}^{(\lambda')}_Q\}
_{Q\in\mathcal{D}(\mathbb{R}^{n+1})}$ be as in \eqref{eq-t}.
Let $k_0$ be as in Remark \ref{rmk-Dauwavv2}.
For any $Q'\in\mathcal{D}(\mathbb{R}^n)$,
let $E_{P(Q', k_0)}$ be as in \eqref{eq-E_PQ}.
Moreover, let $\mathbb{A}(W):=\{A_{Q, W}\}_{Q\in\mathcal{D}(\mathbb{R}^{n+1})}$
and $\mathbb{A}(V):=\{A_{Q', V}\}_{Q'\in\mathcal{D}(\mathbb{R}^n)}$
be sequences of reducing operators of order $p$,
respectively, for $W$ and $V$.

From \eqref{eq-AV-AW} and the
definition of $\|\cdot\|_{L\dot F_{p,q}^{\upsilon^{(n+1)}}}$,
it follows that, for any $\lambda'\in\Lambda_{n}$,
\begin{align*}
&\left\|\left\{2^{j(s+\frac{n+1}{2})}
\sum_{Q'\in\mathcal{D}_j(\mathbb{R}^n)}\left|A_{P(Q', k_0),W}
\vec{t}^{(\lambda')}_{P(Q', k_0)}\right|\mathbf{1}_{E_{P(Q', k_0)}}
\right\}_{j\in\mathbb Z}\right\|_{L\dot F_{p,q}^{\upsilon^{(n+1)}}}\\
&\quad\lesssim\left\|\left\{2^{j(s+\frac{n+1}{2})}
\sum_{Q'\in\mathcal{D}_j(\mathbb{R}^n)}
2^{-j(\frac{\gamma}{p}-\frac{1}{p})}	
\left|A_{Q', V}\vec{t}_{P(Q', k_0)}^{(\lambda)}
\right|\mathbf{1}_{E_{P(Q', k_0)}}
\right\}_{j\in\mathbb Z}\right\|_{L\dot F_{p,q}^{\upsilon^{(n+1)}}}\nonumber\\
&\quad=\sup_{P\in\mathcal{D}(\mathbb{R}^{n+1})}\frac{1}{\upsilon^{(n+1)}(P)}
\left(\int_P\Bigg\{\sum_{\genfrac{}{}{0pt}{}{Q'\in\mathcal{D}(\mathbb{R}^n)}{P(Q', k_0)
\subset P}}\left[|Q'|^{-\frac{s}{n}-\frac{1}{2}
+\frac{1}{p}(\frac{\gamma}{n}-\frac{1}{n})}\right.\right.\nonumber\\
&\quad\quad\quad\quad\times\left.\left.\left|A_{Q', V}
\left\langle\vec{f},\theta^{(\lambda')}_{Q'}
\right\rangle_*\right|\mathbf{1}_{E_{P(Q', k_0)}}(x)\right]^q
\Bigg\}^{\frac{p}{q}}\,dx\right)^{\frac{1}{p}}=:\Omega,
\end{align*}
which, together with the definition of $\|\cdot\|_{\dot{b}^{s-\frac{\gamma}{p},\upsilon^{(n)}}
_{p,p}(\mathbb{R}^n)}$, the construction of $\mathcal{I}(P)$ for any $P\in\mathcal{D}(\mathbb{R}^{n+1})$ [see \eqref{eq-I(P)}], the disjointness of $\{E_{P(Q',k_0)}\}_{Q'\in\mathcal{D}(\mathbb{R}^n)}$,
\eqref{eq-t}, \eqref{eq-vIP}, \eqref{eq-P},
and Remark \ref{rmk-a(A)-a}, further implies that
\begin{align}\label{eq-a}
\Omega&\lesssim\sup_{P\in\mathcal{D}(\mathbb{R}^{n+1})}
\frac{1}{\upsilon^{(n)}(\mathcal{I}(P))}
\left[\sum_{\genfrac{}{}{0pt}{}{Q'\in\mathcal{D}(\mathbb{R}^n)}{Q'\subset \mathcal{I}(P)}}
|Q'|^{-(\frac{s}{n}+\frac{1}{2})p+\frac{\gamma}{n}+1}
\left|A_{Q',V}\left\langle\vec{f},\theta^{(\lambda')}_{Q'}
\right\rangle_*\right|^p \right]^{\frac{1}{p}}\\
&=\sup_{R\in\mathcal{D}(\mathbb{R}^n)}\frac{1}{\upsilon^{(n)}(R)}
\left\{\int_R \sum_{\genfrac{}{}{0pt}{}{Q'\in\mathcal{D}(\mathbb{R}^n)}{Q'\subset R}}
\left[|Q'|^{-\frac{1}{n}(s-\frac{\gamma}{p})}
\left|A_{Q', V}\left\langle\vec{f},\theta^{(\lambda')}_{Q'}
\right\rangle_*\right|\widetilde{\mathbf{1}}_{Q'}(x)
\right]^p\,dx\right\}^{\frac{1}{p}}\notag\\
&=\left\|\left\{\left|A_{Q', V}\left\langle\vec{f},
\theta^{(\lambda')}_{Q'} \right\rangle_*\right|
\right\}_{Q'\in\mathcal{D}(\mathbb{R}^n)}
\right\|_{\dot{b}^{s-\frac{\gamma}{p},\upsilon^{(n)}}
_{p,p}(\mathbb{R}^n)}\notag=\left\|\left\{\left\langle\vec{f},
\theta^{(\lambda')}_{Q'}\right\rangle_*\right\}
_{Q'\in\mathcal{D}(\mathbb{R}^n)}\right\|_{\dot{b}^{s-\frac{\gamma}{p},
\upsilon^{(n)}}_{p,p}(\mathbb{A}(V))}\notag.
\end{align}
For any $Q\in\mathcal{D}(\mathbb{R}^{n+1})$,
let $E_Q$ be as in \eqref{eq-E_PQ}.
Applying this, the definition of $\|\cdot\|_{\dot{f}^{s,
\upsilon^{(n+1)}}_{p,q}(\mathbb{A}(W))}$, \eqref{eq-a}, \eqref{eq-t},
and Lemmas \ref{a(A)=a(W)}, \ref{lem-aLA-equi},
and \ref{lem-Dauwav decomp} with the fact that
$|E_Q|=\frac{1}{3}|Q|$
for any $Q\in\mathcal{D}(\mathbb{R}^{n+1})$, we obtain,
for any $\lambda'\in\Lambda_{n}$,
\begin{align}\label{eq-b}
&\left\|\left\{\vec{t}^{(\lambda')}_Q\right\}
_{Q\in\mathcal{D}(\mathbb{R}^{n+1})}\right\|
_{\dot{f}^{s,\upsilon^{(n+1)}}_{p,q}(W)}\\
&\quad\sim\left\|\left\{\vec{t}^{(\lambda')}_Q\right\}
_{Q\in\mathcal{D}(\mathbb{R}^{n+1})}\right\|
_{\dot{f}^{s,\upsilon^{(n+1)}}_{p,q}(\mathbb{A}(W))}
\sim\left\|\left\{2^{js}\sum_{Q\in\mathcal{D}_j
(\mathbb{R}^{n+1})}\left|A_{Q, W}
\vec{t}^{(\lambda')}_{Q}\right|\widetilde{\mathbf{1}}_{E_{Q}}
\right\}_{j\in\mathbb Z}\right\|_{L\dot F_{p,q}^{\upsilon^{(n+1)}}}\nonumber\\
&\quad=\left\|\left\{2^{js}\sum_{Q'\in\mathcal{D}_j
(\mathbb{R}^n)}\left|A_{P(Q', k_0),W}
\vec{t}^{(\lambda')}_{P(Q', k_0)}\right|
\widetilde{\mathbf{1}}_{E_{P(Q', k_0)}}
\right\}_{j\in\mathbb Z}\right\|_{L\dot{F}_{p,q}^{\upsilon^{(n+1)}}}\nonumber\\
&\quad\lesssim\left\|\left\{\left\langle\vec{f},
\theta^{(\lambda')}_{Q'}\right\rangle_*\right\}
_{Q'\in\mathcal{D}(\mathbb{R}^n)}\right\|_{\dot{b}^{s-\frac{\gamma}{p},
\upsilon^{(n)}}_{p,p}(\mathbb{A}(V))}
\sim\left\|\left\{\left\langle\vec{f},\theta^{(\lambda')}_{Q'} \right\rangle_*\right\}_{Q'\in\mathcal{D}(\mathbb{R}^n)}\right\|
_{\dot{b}^{s-\frac{\gamma}{p},\upsilon^{(n)}}_{p,p}(V)}\nonumber\\
&\quad\lesssim\left\|\vec{f}\right\|_{\dot{B}^{s-\frac{\gamma}{p},
\upsilon^{(n)}}_{p,p}(V)}\nonumber
\end{align}
and hence
\begin{align}\label{eq-t-f-xx}
\left\{\vec{t}^{(\lambda')}_Q\right\}_{Q\in\mathcal{D}(\mathbb{R}^{n+1})}
\in\dot{f}^{s,\upsilon^{(n+1)}}_{p,q}(W).
\end{align}
By an argument similar to that used in the proof of Theorem \ref{thm-Ext-B},
we find that there exists a positive constant $C$ such
that, for any $\lambda'\in\Lambda_{n}$,
$\{C[\theta^{(\lambda')}\otimes\varphi]_Q\}
_{Q\in\mathcal{D}(\mathbb{R}^{n+1})}$
is a family of synthesis molecules for
$\dot{F}^{s,\upsilon^{(n+1)}}_{p,q}(W)$.
Using this, \eqref{eq-t-f-xx}, and Lemma \ref{lem-moledecomp}(ii),
we conclude that, for any $\lambda'\in\Lambda_{n}$,
\begin{align*}
\sum_{Q\in\mathcal{D}(\mathbb{R}^{n+1})}\vec{t}^{(\lambda')}_Q
\left[\theta^{(\lambda')}\otimes \varphi \right]_Q
\end{align*}
converges in $[\mathcal{S}_\infty'(\mathbb{R}^{n+1})]^m$ and
\begin{align*}
\left\|\sum_{Q\in\mathcal{D}(\mathbb{R}^{n+1})}\vec{t}^{(\lambda')}_Q
\left[\theta^{(\lambda')}\otimes \varphi \right]_Q \right\|
_{\dot{F}^{s,\upsilon^{(n+1)}}_{p,q}(W)}
\lesssim\left\|\left\{\vec{t}^{(\lambda')}_Q\right\}
_{Q\in\mathcal{D}(\mathbb{R}^{n+1})}\right\|
_{\dot{f}^{s,\upsilon^{(n+1)}}_{p,q}(W)}\nonumber.
\end{align*}
From this, the quasi-triangle inequality of
$\|\cdot\|_{\dot{F}^{s,\upsilon^{(n+1)}}_{p,q}(W)}$,
\eqref{eq-Ext-f}, \eqref{eq-b}, and Lemma \ref{lem-Dauwav decomp},
it follows that, for any $\vec{f}\in\dot{B}^{s-\frac{\gamma}{p},
\upsilon^{(n)}}_{p,p}(V)$, $\operatorname{Ext}\vec{f}$ is well-defined
in $[\mathcal{S}_\infty'(\mathbb{R}^{n+1})]^m$ and
\begin{align*}
\left\|\operatorname{Ext}\vec{f}\right\|
_{\dot{F}^{s,\upsilon^{(n+1)}}_{p,q}(W)}
&\lesssim\sum_{\lambda'\in\Lambda_{n}}
\left\|\sum_{Q\in\mathcal{D}(\mathbb{R}^{n+1})}\vec{t}^{(\lambda')}_Q
\left[\theta^{(\lambda')}\otimes \varphi \right]_Q \right\|
_{\dot{F}^{s,\upsilon^{(n+1)}}_{p,q}(W)}\\
&\lesssim\sum_{\lambda'\in\Lambda_{n}}
\left\|\left\{\vec{t}^{(\lambda')}_Q\right\}
_{Q\in\mathcal{D}(\mathbb{R}^{n+1})}\right\|
_{\dot{f}^{s,\upsilon^{(n+1)}}_{p,q}(W)}\nonumber\\
&\lesssim\sum_{\lambda'\in\Lambda_{n}}\left
\|\left\{\left\langle\vec{f},\theta^{(\lambda')}_{Q'}
\right\rangle_*\right\}_{Q'\in\mathcal{D}(\mathbb{R}^n)}\right\|
_{\dot{b}^{s-\frac{\gamma}{p},\upsilon^{(n)}}_{p,p}
(V)}\sim\left\|\vec{f}\right\|
_{\dot{B}^{s-\frac{\gamma}{p},\upsilon^{(n)}}_{p,p}(V)},\nonumber
\end{align*}
which further implies that
$\operatorname{Ext}$ is bounded from
$\dot{B}^{s-\frac{\gamma}{p},\upsilon^{(n)}}_{p,p}(V)$
to $\dot{F}^{s,\upsilon^{(n+1)}}_{p,q}(W)$.

If $s\in(\frac{\gamma}{p}+E+\frac{d_{p,
\infty}^{\mathrm{upper}}(V)}{p},\infty)$ and
\eqref{eq-tr-B-B} holds, to prove $\operatorname{Tr}
\circ\operatorname{Ext}$ is the identity on
$\dot{B}^{s-\frac{\gamma}{p},\upsilon^{(n)}}_{p,p}(V)$,
it suffices to repeat an argument similar to that used in the proof
of Theorem \ref{thm-Ext-B} with $\vec{f}\in\dot{B}^{s-\frac{\gamma}{p},
\upsilon^{(n)}}_{p,q}(V)$ replaced by $\dot{B}^{s-\frac{\gamma}{p},\upsilon^{(n)}}_{p,p}(V)$;
we omit the details. This finishes the proof of Theorem \ref{thm-Ext-F}.
\end{proof}
\end{theorem}

\begin{remark}
Let all the symbols be the same as in Theorems \ref{thm-Tr-F} and
\ref{thm-Ext-F}. Let $\gamma=1$, $\tau\in[0, \infty)$, and,
for any $Q\in\mathcal{D}(\mathbb{R}^{n+1})$, let
$\upsilon^{(n+1)}(Q):=|Q|^{\tau}$. Observe that
the restriction $\upsilon^{(n)}$ of
$\upsilon^{(n+1)}$ on $\mathcal{D}(\mathbb{R}^{n})$ satisfies, for
any $Q\in\mathcal{D}(\mathbb{R}^{n})$, $\upsilon^{(n)}(Q)=|Q|^{\frac{n+1}{n}\tau}$.
In this case, the spaces $\dot{F}^{s,\upsilon^{(n+1)}}_{p,q}(W)$ and
$\dot{B}^{s-\frac{\gamma}{p},\upsilon^{(n)}}_{p,p}(V)$
respectively reduce to $\dot{F}^{s,\tau}_{p,q}(W)$
and $\dot{B}^{s-\frac{1}{p},\frac{n+1}{n}\tau}_{p,p}(V)$.
On the one hand, Theorems \ref{thm-Tr-F} and
\ref{thm-Ext-F} are homogeneous variants, respectively,
of \cite[Theorems 6.3 and 6.5]{bhyy5} in which Bu et al.
obtained the corresponding results in
Theorems \ref{thm-Tr-F} and
\ref{thm-Ext-F} for inhomogeneous spaces $F^{s,\tau}_{p,q}(W)$ and
$B^{s-\frac{1}{p},\frac{n+1}{n}\tau}_{p,p}(V)$;
on the other hand, Bu et al. \cite[Theorems 5.13 and 5.16]{bhyy3}
also established the corresponding results in
Theorems \ref{thm-Tr-F} and \ref{thm-Ext-F} for
$\dot{F}^{s,\tau}_{p,q}(W)$ and $\dot{B}^{s-\frac{1}{p},\frac{n+1}{n}\tau}_{p,p}(V)$
with $W, V$ being matrix $\mathcal{A}_p$ weights.
However, Theorems \ref{thm-Tr-F} and \ref{thm-Ext-F}
only respectively coincide with \cite[Theorems 5.13 and 5.16]{bhyy3}
for the case where $p\in(0, 1]$.
When $p\in(1, \infty)$, the range of the index $s$
in \cite[Theorems 5.13 and 5.16]{bhyy3} may be better than that
in Theorems \ref{thm-Tr-F} and \ref{thm-Ext-F};
see Remark \ref{rmk-compare} for the reason.
Notice that   $\dot{F}^{s,0}_{p,q}(W)$
and $\dot{B}^{s-\frac{1}{p},0}_{p,p}(V)$ respectively
reduce to   $\dot{F}^{s}_{p,q}(W)$
and $\dot{B}^{s-\frac{1}{p}}_{p,p}(V)$
with $W, V$ being matrix $\mathcal{A}_{p,\infty}$ weights.
Even for these spaces,
Theorems \ref{thm-Tr-F} and \ref{thm-Ext-F} are also new.

Furthermore, when $m=1$ (the scalar-valued setting)
and $W, V\equiv1$, Theorem \ref{thm-Tr-B} contains
\cite[Theorem 1.4]{syy10} in which Sawano et al. obtained
the corresponding result for
$\dot{F}^{s,\tau}_{p,q}(\mathbb{R}^{n+1})$ with $\tau$
having an upper bound.  In particular, when $\tau=0$,
Theorem \ref{thm-Tr-F} is exactly the classical result
\cite[Theorem 5.1]{jaw77}.
\end{remark}

\subsection{Calder\'{o}n--Zygmund Operators}\label{s-2-3}

The goal of this subsection is to establish the
boundedness of Calder\'{o}n--Zygmund operators
on $\dot{A}^{s,\upsilon}_{p,q}(W)$. To this end,
we first recall the concept of Calder\'on--Zygmund operators.
Let $T$ be a continuous linear operator from
$\mathcal{S}(\mathbb{R}^n)$ to $\mathcal{S}'(\mathbb{R}^n)$.
Using the well-known Schwartz kernel theorem,
we conclude that there exists a \emph{distribution kernel}
$\mathcal{K}\in\mathcal{S}'(\mathbb R^n\times\mathbb R^n)$
such that, for any $\varphi,\psi\in\mathcal{S}(\mathbb R^n)$,
$$
\langle T\varphi,\psi\rangle
=\langle\mathcal K,\varphi\otimes\psi\rangle,
$$
where $\varphi\otimes\psi$ is as in \eqref{eq-g-h}.
The following definition of Calder\'on--Zygmund operators
is precisely \cite[Definition 6.7]{bhyy3}.

\begin{definition}\label{def-CZO}
Let $T$ be a continuous linear operator from
$\mathcal{S}(\mathbb R^n)$ to $\mathcal{S}'(\mathbb R^n)$
with distribution kernel $\mathcal{K}$. For any $l\in(0,\infty)$,
we say that $T$ is a \emph{Calder\'on--Zygmund operator} of order $l$,
denoted by $T\in\operatorname{CZO}(l)$, if the restriction of
$\mathcal{K}$ on $\{(x,y)\in\mathbb{R}^n\times\mathbb{R}^n:\ x\neq y\}$
agrees with a continuous function $K$ whose
all derivatives  exist as
continuous functions and there exists a
positive constant $C$ such that
\begin{itemize}
\item[{\rm (i)}] for any $\gamma\in\mathbb{Z}_+^n$ with
$|\gamma|\leq\lfloor\!\lfloor l\rfloor\!\rfloor$
and for any $x,y\in\mathbb{R}^n$ with $x\neq y$,
\begin{align*}
\left|\partial_x^{\gamma}K(x,y)\right|
\leq C\frac{1}{|x-y|^{n+|\gamma|}};
\end{align*}
\item[{\rm (ii)}] for any $\gamma\in\mathbb{Z}_+^n$
with $|\gamma|=\lfloor\!\lfloor l\rfloor\!\rfloor$
and for any $x,y,h\in\mathbb{R}^n$ with $|h|<\frac12|x-y|$,
\begin{equation*}
\left|\partial_x^{\gamma}K(x,y)
-\partial_x^\gamma K(x+h,y)\right|
\leq C\frac{|h|^{l^{**}}}{|x-y|^{n+l}},
\end{equation*}
\end{itemize}
where $\lfloor\!\lfloor l\rfloor\!\rfloor $ and $l^{**}$
are as, respectively, in \eqref{eq-ceil} and \eqref{eq-r**}.
For any $l\in(-\infty, 0]$, the condition
$T\in\operatorname{CZO}(l)$ is void.
\end{definition}

The following type of Calder\'on--Zygmund operators
was introduced in \cite[Definition 6.11]{bhyy3}.

\begin{definition}\label{def-CZK}
Let $ E,F\in\mathbb{R}$ and $T$ be a continuous linear operator
from  $\mathcal{S}(\mathbb R^n)$ to $\mathcal{S}'(\mathbb R^n)$
with distribution kernel $\mathcal{K}$.
We call $T\in\operatorname{CZK}^0(E;F)$
if the restriction of $\mathcal K$ on
$\{(x,y)\in\mathbb R^n\times\mathbb R^n:\ x\neq y\}$
agrees with a continuous function $K$
whose all derivatives below exist as
continuous functions and there exists a
positive constant $C$ such that
\begin{itemize}
\item[{\rm (i)}] for any $\alpha\in\mathbb{Z}_+^n$
with $|\alpha|\leq\lfloor\!\lfloor E\rfloor\!\rfloor_+$
and for any $x,y\in\mathbb R^n$ with $x\neq y$,
\begin{align*}
|\partial_x^\alpha K(x,y)|
\leq C\frac{1}{|x-y|^{n+|\alpha|}};
\end{align*}
\item[{\rm (ii)}] for any $\alpha\in\mathbb{Z}_+^n$ with
$|\alpha|=\lfloor\!\lfloor E\rfloor\!\rfloor$
and for any $x,y,h\in\mathbb R^n$ with $|h|<\frac12|x-y|$,
\begin{align*}
|\partial_x^\alpha K(x,y)-\partial_x^\alpha K(x+h,y)|
\leq C\frac{|h|^{E^{**}}}{|x-y|^{n+E}};
\end{align*}
\item[{\rm (iii)}] for any $\alpha,\beta\in\mathbb{Z}_+^n$
with $|\alpha|\leq\lfloor\!\lfloor E\rfloor\!\rfloor_+$
and $|\beta|=\lfloor\!\lfloor F-|\alpha|\rfloor\!\rfloor$
and for any $x,y,h\in\mathbb R^n$ with $|h|<\frac12|x-y|$,
\begin{align*}
\left|\partial_x^\alpha\partial_y^\beta K(x,y)
-\partial_x^\alpha\partial_y^\beta K(x,y+h)\right|
\leq C\frac{|h|^{(F-|\alpha|)^{**}}}{|x-y|^{n+F}},
\end{align*}
where $\lfloor\!\lfloor E\rfloor\!\rfloor$ and $E^{**}$
are as, respectively, in \eqref{eq-ceil} and \eqref{eq-r**}.
\end{itemize}
Moreover, we say that $T\in\operatorname{CZK}^1(E;F)$
if $T\in\operatorname{CZK}^0(E;F)$ and there  exists
a positive constant $L$ such that, for
any $\alpha,\beta\in\mathbb{Z}_+^n$ with
$|\alpha|=\lfloor\!\lfloor E\rfloor\!\rfloor$
and $|\beta|=\lfloor\!\lfloor F-E\rfloor\!\rfloor$
and for any $x,y,u,v\in\mathbb R^n$ with $|u|+|v|<\frac12|x-y|$,
\begin{align*}
\begin{split}
&\left|\partial_x^\alpha\partial_y^\beta\mathcal K(x,y)
-\partial_x^\alpha\partial_y^\beta\mathcal K(x+u,y)
-\partial_x^\alpha\partial_y^\beta\mathcal K(x,y+v)
+\partial_x^\alpha\partial_y^\beta\mathcal K(x+u,y+v)\right|\\
&\quad\leq L\frac{|u|^{E^{**}}|v|^{(F-E)^{**}}}{|x-y|^{n+F}}.
\end{split}
\end{align*}
\end{definition}
For further discussions on the relation between Calder\'on--Zygmund operators
in Definitions \ref{def-CZO} and \ref{def-CZK},
we refer to \cite[Remark 6.13]{bhyy3}.
To further discuss the domain of operators in Definition \ref{def-CZO},
we give some concepts and symbols.
Let $\mathscr{D}:=C_{\mathrm{c}}^{\infty}$ be the
set of all infinitely differentiable functions on $\mathbb{R}^n$
with compact support, equipped with the usual inductive limit topology
(see, for example, \cite[Definition 2.3.1]{gra14a}).
Recall that a sequence of functions
$\{\varphi_j\}_{j\in\mathbb{N}}$ in $\mathscr{D}$
is called a \emph{uniform approximation to the identity}
(see \cite[(2.2.10) and (2.2.11)]{tor91}) if
\begin{itemize}
\item[{\rm (i)}] for any compact set $K\subset\mathbb{R}^n$,
there exists $j_K\in\mathbb{N}$ such that,
for any $j>j_K$ and $x\in K$, $\varphi_j(x)=1$;
\item[{\rm (ii)}] $\sup_{j\in\mathbb{N}}
\|\varphi_j\|_{L^\infty}<\infty$.
\end{itemize}
As in \cite[p.\,38]{tor91}, for any $k\in\mathbb{Z}_+$, let
\begin{align}\label{eq-Ok}
\mathscr{O}^{k}
&:=\left\{f\in C^\infty:\
\text{there exists a positive constant }C\text{ such that, }\right.\\
&\qquad\qquad\quad\left.\text{ for any }x\in\mathbb{R}^n,\
|f(x)|\leq C(1+|x|)^{k}\right\}\notag
\end{align}
and
\begin{align*}
\mathscr{D}_{k}
:=\left\{g\in\mathscr{D}:\
\int_{\mathbb R^n} g(x)x^\gamma\,dx=0\text{ for any }
\gamma\in\mathbb{Z}_+^n\text{ with }|\gamma|\leq
k \right\}.
\end{align*}
The following lemma gives a method to define the action of
Calder\'on--Zygmund operators on functions of
the polynomial growth as in \eqref{eq-Ok},
which is a special case of \cite[Lemma 2.2.12]{tor91}.

\begin{lemma}
Let $l\in(0,\infty)$, $T\in\operatorname{CZO}(l)$,
and $\varphi:=\{\varphi_j\}_{j\in\mathbb{N}}$
be a uniform approximation to the identity.
Then, for any $f\in\mathscr{O}^{\lfloor\!\lfloor l\rfloor\!\rfloor}$
and $g\in\mathscr{D}_{\lfloor\!\lfloor l\rfloor\!\rfloor}$,
the limit
\begin{equation}\label{eq-Tfg}
\langle T(f),g\rangle:=\lim_{j\to\infty}\langle T(\varphi_jf),g\rangle
\end{equation}
exists and its value is independent of the choice of $\varphi$.
In particular, for any $\gamma\in\mathbb{Z}_+^n$
with $|\gamma|\leq\lfloor\!\lfloor l\rfloor\!\rfloor$,
\eqref{eq-Tfg} gives a way to define $T(x^\gamma):
\mathscr{D}_{\lfloor\!\lfloor l\rfloor\!\rfloor}\to\mathbb{C}$.
\end{lemma}

Suppose that $T$ is a continuous linear operator
from $\mathcal{S}$ to $\mathcal{S}'$.
The operator $T^*$ is defined by setting,
for any $\varphi,\psi\in\mathcal{S}$,
$$\langle T^*\psi,\varphi\rangle:=\langle T\varphi,\psi\rangle.$$
We say that $T$ has the \emph{weak boundedness property},
denoted by $T\in\mathbf{WBP}$, if,
for any bounded subset $\mathscr{K}\subset\mathscr{D}$,
there exists a positive constant $C=C(\mathscr{K})$ such that,
for any $\varphi,\eta\in\mathscr{K}$, $h\in\mathbb R^n$,
and $r\in(0,\infty)$,
\begin{align*}
\left|\left\langle T \left[\varphi\left(\frac{\cdot-h}{r}\right)\right],
\eta \left(\frac{\cdot-h}{r}\right)\right\rangle\right|\leq Cr^n.
\end{align*}

Based on Definition \ref{def-CZK}, Bu et al.
\cite[Definition 6.17]{bhyy3} also introduced
the following type of Calder\'on--Zygmund operators.

\begin{definition}\label{def-CZOK}
Let $\sigma\in\{0,1\}$, $ E,F,G,H\in\mathbb{R}$, and
$T$ be a continuous linear operator from
$\mathcal{S}$ to $\mathcal{S}'$.
We say that $T\in\operatorname{CZO}^\sigma(E,F,G,H)$ if
\begin{itemize}
\item[{\rm (i)}] $T\in\mathbf{WBP}$;
\item[{\rm (ii)}] $T\in\operatorname{CZK}^\sigma(E;F)$;
\item[{\rm (iii)}] there exists $\epsilon\in(0, \infty)$ such that
$T\in\operatorname{CZO}(\lfloor G\rfloor+\epsilon)$ and
$T^*\in\operatorname{CZO}(\lfloor H\rfloor+\epsilon)$;
\item[{\rm (iv)}] for any $\gamma\in\mathbb{Z}_+^n$
with $|\gamma|\leq \lfloor G\rfloor$
and for any $\theta\in\mathbb{Z}_+^n$ with
$|\theta|\leq \lfloor H\rfloor$,
$T(x^\gamma)=0$ and $T^*(x^\theta)=0$.
\end{itemize}
\end{definition}

By the standard arguments used in \cite{bhyy3,fj90,ftw88,tor91},
to obtain the boundedness of Calder\'on--Zygmund operators
on $\dot{A}^{s,\upsilon}_{p,q}(W)$, it suffices to
prove that Calder\'on--Zygmund operators
map sufficiently regular atoms (see Definition \ref{def-atom})
to synthesis molecules for $\dot{A}^{s,\upsilon}_{p,q}(W)$.
Following this idea, we start with recalling \cite[Lemma 6.1]{bhyy3}.

\begin{lemma}\label{lem-hat-T}
Let $\varphi, \psi\in\mathcal{S}$
satisfy \eqref{cond1} and \eqref{cond3} and
$T$ be a continuous linear operator
from $\mathcal{S}_{\infty}$
to $\mathcal{S}'_{\infty}$.
Then $\widehat{T}:=\left\{\langle T \psi_R, \varphi_Q
\rangle\right\}_{Q, R\in\mathcal{D}}$ maps
$$S_{\varphi}\left[\mathcal{S}_{\infty}\left(\mathbb{R}^n\right)\right]:=
\left\{\left\{\left\langle f,
\varphi_Q\right\rangle\right\}_{Q\in\mathcal{D}}:
\,f\in\mathcal{S}_{\infty}\left(\mathbb{R}^n\right)\right\}$$
to
$$
S_{\varphi}\left[\mathcal{S}'_{\infty}\left(\mathbb{R}^n\right)\right]:=
\left\{\left\{\left\langle f, \varphi_Q
\right\rangle\right\}_{Q\in\mathcal{D}}:
f\in\mathcal{S}'_{\infty}\left(\mathbb{R}^n\right)\right\}
$$
and satisfies the identity $\widehat{T} \circ S_{\varphi}
=S_{\varphi} \circ T$ on $\mathcal{S}_{\infty}$.
\end{lemma}

Based on Lemma \ref{lem-hat-T}, we next establish an
extension result related to $\dot{A}^{s,\upsilon}_{p,q}(W)$.
To this end, we first present the following definition,
which is exactly in \cite[Definition 5.7]{byyz24}.

\begin{definition}
Let $a\in\{b, f\}$, $s\in\mathbb{R}$, $p\in(0,\infty)$,
$q\in(0, \infty]$, and $W\in \mathcal{A}_{p,\infty}$.
Assume that $\delta_1,\delta_2$, and $\omega$ satisfy
\eqref{eq-delta1<0} and $\upsilon\in\mathcal{G}
(\delta_1, \delta_2; \omega)$.
An infinite matrix is said to be
$\dot{a}_{p,q}^{s,\upsilon}(W)$-\emph{almost diagonal}
if it is $(D, E, F)$-almost diagonal with
$D>D_{\dot{a}_{p,q}^{s,\upsilon}(W)}$,
$E>E_{\dot{a}_{p,q}^{s,\upsilon}(W)}$, and
$F>F_{\dot{a}_{p,q}^{s,\upsilon}(W)}$,	
where $D_{\dot{a}_{p,q}^{s,\upsilon}(W)}$,
$E_{\dot{a}_{p,q}^{s,\upsilon}(W)}$,
and $F_{\dot{a}_{p,q}^{s,\upsilon}(W)}$ are as in
Theorem \ref{a(W)adopebound}.
\end{definition}

In what follows, for any quasi-normed spaces $X$ and $Y$,
let $\mathcal{L}(X, Y)$ be the set of all bounded linear
operators from $X$ to $Y$. In particular, we simply denote
$\mathcal{L}(X, X)$ by $\mathcal{L}(X)$. Suppose that $T$
is a continuous linear operator
from $\mathcal{S}_{\infty}$
to $\mathcal{S}'_{\infty}$.
For any $\vec{f}:=(f_1,\dots,f_m)^T
\in(\mathcal{S}_{\infty})^m$,
let $T\vec{f}:=(T f_1,\dots, T f_m)^T
\in(\mathcal{S}'_{\infty})^m$.
The following lemma gives an extension result on
$\dot{A}^{s,\upsilon}_{p,q}(W)$.

\begin{lemma}\label{lem-ad-T}
Let $(A, a)\in\{(B, b), (F, f)\}$, $s\in\mathbb{R}$,
$p\in(0,\infty)$, $q\in(0, \infty]$, $W\in \mathcal{A}_{p,\infty}$,
and $\upsilon\in\mathcal{G}(\delta_1, \delta_2; \omega)$ with
$\delta_1,\delta_2,\omega$ as in \eqref{eq-delta1<0}.
Assume that $\varphi, \psi\in\mathcal{S}$
satisfy \eqref{cond1} and \eqref{cond3}
and $T$ is a continuous linear operator
from $\mathcal{S}_{\infty}$ to $\mathcal{S}'_{\infty}$.
If $\widehat{T}:=\left\{\left\langle T
\psi_R, \varphi_Q\right\rangle\right\}_{Q, R\in\mathcal{D}}$
is $\dot{a}_{p, q}^{s, \upsilon}(W)$-almost diagonal,
then there exists $\widetilde{T}\in\mathcal{L}
(\dot{A}_{p, q}^{s, \upsilon}(W))$ which agrees with
$T$ on $(\mathcal{S}_{\infty})^m
\cap\dot{A}^{s,\upsilon}_{p,q}(W)$.
\end{lemma}

\begin{proof}
From Lemmas \ref{lem-Cdreproform} and \ref{lem-hat-T},
it follows that $\widetilde{T}:=T_{\psi}\circ\widehat{T}
\circ S_{\varphi}$ agrees with $T$ on
$(\mathcal{S}_{\infty})^m$.
By Lemma \ref{lem-phitransMWBTL}, we find that
$S_\varphi\in\mathcal{L}(\dot{A}_{p, q}^{s, \upsilon}(W),
\dot{a}_{p, q}^{s, \upsilon}(W))$ and
$T_\psi \in \mathcal{L}(\dot{a}_{p, q}^{s, \upsilon}(W),
\dot{A}_{p, q}^{s, \upsilon}(W))$.
Applying these and Theorem \ref{a(W)adopebound},
we conclude that $\widetilde{T}\in\mathcal{L}
(\dot{A}_{p, q}^{s, \upsilon}(W))$.
This finishes the proof of Lemma \ref{lem-ad-T}.
\end{proof}

We now recall the concept of (smooth) atoms
(see, for instance, \cite[(4.1)-(4.3)]{fj90}
and \cite[Definition 6.4]{bhyy3}).

\begin{definition}\label{def-atom}
Let $L,N\in\mathbb R$ and $Q\subset\mathbb{R}^n$
be a cube. A function $a_Q\in
C_{\mathrm{c}}^{\infty}$ is called
an \emph{$(L,N)$-atom} on $Q$ if
\begin{itemize}
\item[{\rm (i)}] $\operatorname{supp}a_Q\subset3Q$;
\item[{\rm (ii)}] for any $\gamma\in\mathbb{Z}_+^n$
with $|\gamma|\leq L$,
$\int_{\mathbb{R}^n} a_Q(x)x^\gamma\,dx=0$;
\item[{\rm (iii)}]  for any $\gamma\in\mathbb{Z}_+^n$
with $|\gamma|\leq N$ and for any $x\in\mathbb{R}^n$,
$|\partial^\gamma a_Q(x)|\leq|Q|^{-\frac12-\frac{|\gamma|}{n}}$.
\end{itemize}
\end{definition}

The next decomposition lemma follows from
the proof of \cite[Proposition 6.5]{bhyy3};
we omit the details.

\begin{lemma}\label{lem-psi-sumatom}
Let $M,L,N\in(0, \infty)$ and $\varphi,
\psi\in\mathcal{S}$
satisfy \eqref{cond1} and \eqref{cond3}.
Then there exists a positive constant $C$ such that,
for any $R\in\mathcal{D}$, there is a sequence of
$(L,N)$-atoms $\{t_P\}_{P\in\mathcal{D}, \ell(P)=\ell(R)}$ such that
\begin{align}\label{eq-decom-atom}
\psi_R =C\sum_{P\in\mathcal{D}, \ell(P)=\ell(R)}
\left(1+[\ell(R)]^{-1}\left|x_R-x_P\right|\right)^{-M} t_P
\end{align}
in $\mathcal{S}$.
\end{lemma}

The following two lemmas are respectively
\cite[Lemma 5.15(i) and Proposition 5.8]{byyz24}.

\begin{lemma}\label{lem-anasynmole}
Let $(A, a)\in\{(B, b), (F, f)\}$, $s\in\mathbb{R}$, $p\in(0, \infty)$,
$q\in(0,\infty]$, and $W\in\mathcal{A}_{p, \infty}$.
Suppose that $\delta_1,\delta_2$, and $\omega$ satisfy \eqref{eq-delta1<0},
$\upsilon\in\mathcal{G}(\delta_1, \delta_2; \omega)$,
and $\varphi, \psi\in\mathcal{S}$ satisfy \eqref{cond1} and \eqref{cond3}.
Assume that $\{m_Q\}_{Q\in\mathcal{D}}$ and
$\{g_R\}_{R \in \mathcal{D}}$ are respectively families
of analysis and synthesis molecules for $\dot{A}_{p,q}^{s,\upsilon}(W)$.
Then $\{\langle m_Q, g_R\rangle\}_{Q, R\in\mathcal{D}}$
is $\dot{a}_{p,q}^{s,\upsilon}(W)$-almost diagonal.
\end{lemma}

\begin{lemma}\label{lem-compadope}
Let $a\in\{b, f\}$, $s\in\mathbb{R}$, $p\in(0,\infty)$,
$q\in(0, \infty]$, and $W\in \mathcal{A}_{p,\infty}$.
Suppose that $\delta_1,\delta_2$, and $\omega$
satisfy \eqref{eq-delta1<0} and
$\upsilon\in\mathcal{G}(\delta_1, \delta_2; \omega)$.
If infinite matrices $U^{(1)}:=\{u^{(1)}_{QR}\}
_{Q,R\in\mathcal{D}}$ and $U^{(2)}
:=\{u^{(2)}_{QR}\}_{Q,R\in\mathcal{D}}$ are both
$\dot{a}_{p,q}^{s,\upsilon}(W)$-almost diagonal,
then the infinite matrix
$U:=U^{(1)}\circ U^{(2)}:=\{\sum_{P\in\mathcal{D}}u^{(1)}_{QP}
u^{(2)}_{PR}\}_{Q,R\in\mathcal{D}}$
is also $\dot{a}_{p,q}^{s,\upsilon}(W)$-almost diagonal.
\end{lemma}

We next establish an extension result related to
$\dot{A}_{p,q}^{s,\upsilon}(W)$, which is crucial for
proving the main result of this subsection.

\begin{lemma}\label{lem-Ext}
Let $A\in\{B, F\}$, $s\in\mathbb{R}$,
$p,L,N\in(0,\infty)$, $q\in(0,\infty]$,
and $W\in\mathcal{A}_{p,\infty}$. Assume that
$\delta_1, \delta_2,\omega$ satisfy \eqref{eq-delta1<0},
$\upsilon\in\mathcal{G}(\delta_1, \delta_2; \omega)$, and
$T$ is a continuous linear operator from
$\mathcal{S}$ to $\mathcal{S}'$.
Then the following two statements hold.
\begin{itemize}
\item[{\rm (i)}] If there exists a positive constant $C$ such that,
for any family of sufficiently regular atoms $\{t_P\}_{P\in\mathcal{D}}$,
$\{C T(t_P)\}_{P\in\mathcal{D}}$ is a family of $(K,L,M,N)$-molecules
with $K,L,M,N$ satisfying \eqref{eq-syn-mole},
then there exists an operator $\widetilde{T}
\in\mathcal{L}(\dot{A}^{s,\upsilon}_{p,q}(W))$
which agrees with $T$ on $(\mathcal{S}_{\infty})^m
\cap\dot{A}^{s,\upsilon}_{p,q}(W)$.	
\item[{\rm (ii)}] Moreover, if there exists an operator
$R\in\mathcal{L}(L^2)$
which agrees with $T$ on $\mathcal{S}$,
then the operator $\widetilde{T}$ in (i) also agrees with $T$ on
$(\mathcal{S})^m\cap\dot{A}^{s,\upsilon}_{p,q}(W)$.
\end{itemize}
\end{lemma}

\begin{proof}
To prove (i), let $\varphi, \psi\in\mathcal{S}$
satisfy \eqref{cond1} and \eqref{cond3}
and let $M\in(D_{\dot{a}_{p,q}^{s,\upsilon}(W)}, \infty)$,
where $D_{\dot{a}_{p,q}^{s,\upsilon}(W)}$ is as in
Theorem \ref{a(W)adopebound}. By Lemma \ref{lem-psi-sumatom},
we find that, for any $R\in\mathcal{D}$, there exist
a positive constant $C$ and a sequence $\{t_P\}_{P\in\mathcal{D}, \ell(P)=\ell(R)}$
of sufficiently regular atoms such that \eqref{eq-decom-atom} holds
in $\mathcal{S}$. Applying \eqref{eq-decom-atom} and the continuity
of $T$, we obtain, for any $Q, R\in\mathcal{D}$,
\begin{align}\label{eq-TpsiR-phiQ}
\left\langle T\psi_R, \varphi_Q\right\rangle
=C\sum_{P\in\mathcal{D}, \ell(P)=\ell(R)}
\left(1+[\ell(R)]^{-1}\left|x_R-x_P\right|\right)^{-M}
\left\langle T t_P, \varphi_Q\right\rangle.
\end{align}
Next, we define the matrix $\{h_{RP}\}_{R,P\in\mathcal{D}}$
by setting, for any $R,P\in\mathcal{D}$,
\begin{align*}
h_{RP}=
\begin{cases}
\left(1+[\ell(R)]^{-1}\left|x_R-x_P\right|\right)^{-M}
&\displaystyle\text{if } \ell(R)=\ell(P),\\
0  &\text{otherwise.}
\end{cases}
\end{align*}
Since $M\in(D_{\dot{a}_{p,q}^{s,\upsilon}(W)}, \infty)$,
it is easy to verify that $\{h_{RP}\}_{R,P\in\mathcal{D}}$
is $\dot{a}_{p,q}^{s,\upsilon}(W)$-almost diagonal.
Using the assumption on $T$ that,
for any family of sufficiently regular atoms
$\{t_P\}_{P\in\mathcal{D}}$, $\{C T(t_P)\}_{P\in\mathcal{D}}$
is a family of $(K,L,M,N)$-molecules
with $K,L,M,N$ satisfying \eqref{eq-syn-mole},
Remark \ref{rmk-mole-phi}, and
Lemma \ref{lem-anasynmole}, we conclude that
$\{\langle T t_P, \varphi_Q\rangle\}_{Q, P\in\mathcal{D}}$
is $\dot{a}_{p,q}^{s,\upsilon}(W)$-almost diagonal.
This, together with \eqref{eq-TpsiR-phiQ} and
Lemma \ref{lem-compadope}, further implies that
$\{\langle T\psi_R, \varphi_Q\rangle\}_{Q,R\in\mathcal{D}}$
is $\dot{a}_{p, q}^{s, \upsilon}(W)$-almost diagonal.
By this and Lemma \ref{lem-ad-T}, we find that
there exists an operator $\widetilde{T}
\in\mathcal{L}(\dot{A}^{s,\upsilon}_{p,q}(W))$
that agrees with $T$ on $(\mathcal{S}_{\infty})^m
\cap\dot{A}^{s,\upsilon}_{p,q}(W)$, which completes the proof of (i).
The proof of (ii) follows from  the same argument as that
used in the proof of \cite[Proposition 6.5(ii)]{bhyy3};
we omit the details. This finishes the proof of Lemma \ref{lem-Ext}.
\end{proof}

Based on Lemma \ref{lem-Ext}, to establish the boundedness of
Calder\'on--Zygmund operators on $\dot{A}^{s,\upsilon}_{p,q}(W)$,
it suffices to find   right assumptions on
Calder\'on--Zygmund operators ensuring  that they map sufficiently regular
atoms to synthesis molecules for $\dot{A}^{s,\upsilon}_{p,q}(W)$.
The next proposition,which is precisely \cite[Proposition 6.19]{bhyy3},
provides such  assumptions on Calder\'on--Zygmund operators
in Definition \ref{def-CZOK}.

\begin{proposition}\label{prop-EFGH-KLMN}
Let $\sigma\in\{0,1\}$, $E,F,G,H,L,N\in\mathbb{R}$,
$K,M\in[0, \infty)$, and $T\in\operatorname{CZO}^\sigma(E,F,G,H)$.
If
\begin{align*}
\sigma\geq\mathbf{1}_{(0,\infty)}(N),\
\begin{cases}
E\geq N,\\
E>\lfloor N\rfloor_+,
\end{cases}
\begin{cases}
F\geq (K\vee M)-n,\\
F>\lfloor L\rfloor,
\end{cases}
G\geq\lfloor N\rfloor_+,\ \text{and}\
H\geq\lfloor L\rfloor,
\end{align*}
then there exists a positive constant $C$ such that,
for any family of sufficiently regular atoms $\{t_P\}_{P\in\mathcal{D}}$,
$\{C T(t_P)\}_{P\in\mathcal{D}}$ is a family of $(K,L,M,N)$-molecules.
\end{proposition}

Applying Proposition \ref{prop-EFGH-KLMN},
we obtain the following result.

\begin{corollary}\label{cor-attosynmole}
Let $a\in\{b, f\}$, $s\in\mathbb{R}$,
$p\in(0,\infty)$, $q\in(0, \infty]$, and $W\in \mathcal{A}_{p,\infty}$.
Suppose that $\delta_1, \delta_2$, and $\omega$ satisfy \eqref{eq-delta1<0},
$\upsilon\in\mathcal{G}(\delta_1, \delta_2; \omega)$,
and $D_{\dot{a}_{p,q}^{s,\upsilon}(W)},
E_{\dot{a}_{p,q}^{s,\upsilon}(W)}$, and $F_{\dot{a}_{p,q}^{s,\upsilon}(W)}$
are as in Theorem \ref{a(W)adopebound}. If
\begin{align}\label{eq-CZ}
&\sigma\geq\mathbf{1}_{[0,\infty)}
\left(E_{\dot{a}_{p,q}^{s,\upsilon}(W)}-\frac{n}{2}\right),\
E>\left(E_{\dot{a}_{p,q}^{s,\upsilon}(W)}-\frac{n}{2}\right)_+,\
F>\left[D_{\dot{a}_{p,q}^{s,\upsilon}(W)}\vee
\left(F_{\dot{a}_{p,q}^{s,\upsilon}(W)}+\frac{n}{2}\right)\right]-n,\\\nonumber
&G\geq\left\lfloor
E_{\dot{a}_{p,q}^{s,\upsilon}(W)}-\frac{n}{2}\right\rfloor_+,
\text{ and }H\geq\left\lfloor F_{\dot{a}_{p,q}^{s,\upsilon}(W)}-\frac{n}{2}\right\rfloor,
\end{align}
then there exists a positive constant $C$ such that,
for any family of sufficiently regular atoms $\{t_P\}_{P\in\mathcal{D}}$,
$\{C T(t_P)\}_{P\in\mathcal{D}}$ is a family of $(K,L,M,N)$-molecules
with $K,L,M,N$ satisfying \eqref{eq-syn-mole}.
\end{corollary}
\begin{proof}
To prove the present corollary, let
$L:=F_{\dot{a}_{p,q}^{s,\upsilon}(W)}-\frac{n}{2}$,
$K\in(D_{\dot{a}_{p,q}^{s,\upsilon}(W)}\vee
(F_{\dot{a}_{p,q}^{s,\upsilon}(W)}+\frac{n}{2}), F+n]$, $M\in(D_{\dot{a}_{p,q}^{s,\upsilon}(W)}, F+n]$, and
\begin{align*}
\begin{cases}
N\in\left(\left\lfloor E_{\dot{a}_{p,q}^{s,\upsilon}(W)}
-\frac{n}{2}\right\rfloor, \left\lceil \! \left\lceil
E_{\dot{a}_{p,q}^{s,\upsilon}(W)}-\frac{n}{2}\right\rceil \! \right\rceil
\wedge E\right) &\displaystyle\text{if }
E_{\dot{a}_{p,q}^{s,\upsilon}(W)}-\frac{n}{2}\geq0,\\
N:=0\in \left(E_{\dot{a}_{p,q}^{s,\upsilon}(W)}-\frac{n}{2}, E\right)&\text{otherwise}.
\end{cases}
\end{align*}
There is no difficulty to verify that,
under the assumptions of \eqref{eq-CZ},
\begin{align*}
\sigma\geq\mathbf{1}_{(0,\infty)}(N),\
\begin{cases}
E\geq N,\\
E>\lfloor N\rfloor_+,
\end{cases}
\begin{cases}
F\geq (K\vee M)-n,\\
F>\lfloor L\rfloor,
\end{cases}
G\geq\lfloor N\rfloor_+,\ \text{and}\
H\geq\lfloor L\rfloor
\end{align*}
hold. By these and Proposition \ref{prop-EFGH-KLMN},
we find that there exists a positive constant $C$ such that,
for any family of sufficiently regular atoms $\{t_P\}_{P\in\mathcal{D}}$,
$\{C T(t_P)\}_{P\in\mathcal{D}}$ is a family of $(K,L,M,N)$-molecules
with $K,L,M,N$ satisfying \eqref{eq-syn-mole}.
This finishes the proof of Corollary \ref{cor-attosynmole}.
\end{proof}

Finally, we establish the boundedness of
Calder\'on--Zygmund operators on
$\dot{A}^{s,\upsilon}_{p,q}(W)$.

\begin{theorem}\label{thm-CZ}
Let $(A, a)\in\{(B, b), (F, f)\}$, $s\in\mathbb{R}$,
$p\in(0,\infty)$, $q\in(0, \infty]$, and $W\in \mathcal{A}_{p,\infty}$.
Assume that $\delta_1, \delta_2$, and $\omega$ satisfy
\eqref{eq-delta1<0}, $\upsilon\in\mathcal{G}(\delta_1,
\delta_2; \omega)$, and $D_{\dot{a}_{p,q}^{s,\upsilon}(W)},
E_{\dot{a}_{p,q}^{s,\upsilon}(W)}$, and $F_{\dot{a}_{p,q}^{s,\upsilon}(W)}$
are as in Theorem \ref{a(W)adopebound}. Suppose that
$T\in\operatorname{CZO}^\sigma(E,F,G,H)$ with $\sigma\in\{0,1\}$
and $E,F,G,H\in\mathbb{R}$ satisfying \eqref{eq-CZ}.
Then the following statements hold.
\begin{itemize}
\item[{\rm (i)}] There exists an operator
$\widetilde{T}\in\mathcal L(\dot{A}^{s,\upsilon}_{p,q}(W))$
which agrees with $T$ on $(\mathcal{S}_{\infty})^m
\cap\dot{A}^{s,\upsilon}_{p,q}(W)$.
\item[{\rm (ii)}] Furthermore, when $H\geq0$,
then the operator $\widetilde{T}$ in (i) also agrees with $T$
on $(\mathcal{S})^m\cap\dot{A}^{s,\upsilon}_{p,q}(W)$.	
\end{itemize}
\end{theorem}

\begin{proof}
The proof of (i) immediately follows from Lemma \ref{lem-Ext}(i)
and Proposition \ref{prop-EFGH-KLMN};
we omit the details. Next, we prove (ii).
Notice that, under the assumptions of (ii),
$E, F>0$ and $G, H\geq0$ hold.
From these and the assumption
$T\in\operatorname{CZO}^\sigma(E,F,G,H)$,
it follows that $T\in \mathbf{WPB}\cap {\rm CZO}(E)$,
$T^{*}\in {\rm CZO}(F)$, and $T(1)=T^{*}(1)=0$.
Thus, the operator $T$ also satisfies all the
assumptions of \cite[Theorem 1]{dj84}.
By \cite[Theorem 1]{dj84}, we find that $T$ can be extended
to a bounded linear operator on $L^2$.
Using this and Lemma \ref{lem-Ext}(ii), we conclude that
the operator $\widetilde{T}$ in (i) also agrees with $T$
on $(\mathcal{S})^m\cap\dot{A}^{s,\upsilon}_{p,q}(W)$.
This finishes the proof of Theorem \ref{thm-CZ}.
\end{proof}

\begin{remark}
Let all the symbols be the same as in Theorem \ref{thm-CZ}.
Let $\tau\in[0,\infty)$ and, for any $Q\in\mathcal{D}$,
let $\upsilon(Q):=|Q|^\tau$. In this case,
the space $\dot{A}_{p, q}^{s, \upsilon}(W)$
is exactly the matrix-weighted Besov--Triebel--Lizorkin-type space
$\dot{A}_{p, q}^{s,\tau}(W)$ with $W$ being a matrix
$\mathcal{A}_{p, \infty}$ weight. On the one hand,
Bu et al. \cite[Theorem 6.24]{bhyy5} obtained
the corresponding result in Theorem \ref{thm-CZ}(i)
for the inhomogeneous space $A^{s,\tau}_{p,q}(W)$;
on the other hand, Bu et al. \cite[Theorem 6.18]{bhyy3}
also established the corresponding results in Theorem \ref{thm-CZ}
for $\dot{A}^{s, \tau}_{p,q}(W)$ with
$W$ being a matrix $\mathcal{A}_p$ weight.
We need to point out that
Theorem \ref{thm-CZ} only coincides with
\cite[Theorem 6.18]{bhyy3} for the case where $p\in(0, 1]$.
However, when $p\in(1, \infty)$, the range of parameters in
\cite[Theorem 6.18]{bhyy3} may be better than those in
Theorem \ref{thm-CZ}; see Remark \ref{rmk-compare} for the reason.
Furthermore, the space $\dot{A}^{s,0}_{p,q}(W)$
is precisely the matrix-weighted Besov--Triebel--Lizorkin space
$\dot{A}^{s}_{p,q}(W)$ with $W$ being a matrix $\mathcal{A}_{p, \infty}$
weight. Even for this space, Theorem \ref{thm-CZ} is also new.

When $m=1$ (the scalar-valued setting) and $W\equiv1$ ,
Theorem \ref{thm-CZ} coincides with \cite[Theorem 6.18]{bhyy3}.
Moreover, Theorem \ref{thm-CZ}(i) also coincides with or improves
the classical results in \cite[Theorems 3.1 3.7 and 3.13]{ftw88};
see \cite[Subsection 6.4]{bhyy3} for the detailed discussions about
these coincidences and improvements.
\end{remark}

\bigskip

\noindent Dachun Yang (Corresponding author),
Wen Yuan and Mingdong Zhang.

\medskip

\noindent Laboratory of Mathematics and Complex Systems
(Ministry of Education of China),
School of Mathematical Sciences, Beijing Normal University,
Beijing 100875, The People's Republic of China

\smallskip

\noindent{{\it E-mails:}}
\texttt{dcyang@bnu.edu.cn} (D. Yang)

\noindent\phantom{{\it E-mails:}}
\texttt{wenyuan@bnu.edu.cn} (W. Yuan)

\noindent\phantom{\it E-mails:}
\texttt{mazhang@mail.bnu.edu.cn} (M. Zhang)
\end{document}